\documentclass[10t]{coulonpaper}

\begin{document}

\newcommand{\Dcomment}[1]{\Dmodif\marginpar{\tiny\begin{center}\textcolor{blue}{#1}\end{center}}} 
\newcommand{\Rcomment}[1]{\Rmodif\marginpar{\tiny\begin{center}\textcolor{red}{#1}\end{center}}}
\newcommand{\Dmodif}{$\textcolor{blue}{\spadesuit}$}
\newcommand{\Rmodif}{\ensuremath{\textcolor{red}{\clubsuit}}}

\newcommand{\vburn}[1]{%
	\mathfrak B_{#1}%
}

\newcommand{\freepb}[3][n]{%
	#2 \operatorname{\ast}\displaylimits^{#1} #3
}

\newcommand{\GrSC}[2]{%
	Gr'(#1, #2)%
}

\newcommand{\girth}{%
	\operatorname{girth}%
}

\newcommand{\Cay}	{\operatorname{Cay}}

\newcommand{\scc}{%
        C'(1/6)
} 
\newcommand{\scsc}{%
	C'_*(1/6)%
} 

\title{Small cancellation theory over Burnside groups}
\author{Rémi Coulon, Dominik Gruber}

\maketitle

	\begin{abstract}
	We develop a version of small cancellation theory in the variety of Burnside groups.
	More precisely, we show that there exists a critical exponent $n_0$ such that for every odd integer $n\geq n_0$, the well-known classical $C'(1/6)$-small cancellation theory, as well as its graphical generalization and its version for free products, produce examples of infinite $n$-periodic groups.
	Our result gives a powerful tool for producing (uncountable collections of) examples of $n$-periodic groups with prescribed properties.
	It can be applied without any prior knowledge in the subject of $n$-periodic groups.
	
	As applications, we show the undecidability of Markov properties in classes of $n$-periodic groups, we produce $n$-periodic groups whose Cayley graph contains an embedded expander graphs, and we give an $n$-periodic version of the Rips construction. 
	We also obtain simpler proofs of some known results like the existence of uncountably many finitely generated $n$-periodic groups and the SQ-universality (in the class of $n$-periodic groups) of free Burnside groups.
\end{abstract}



\section{Introduction}
Let $n$ be an integer.
A group $G$ is \emph{periodic of exponent $n$} (or simply \emph{$n$-periodic}) if it satisfies the law $x^n = 1$, i.e. for every element $g \in G$, we have $g^n = 1$.
The \emph{Burnside variety of exponent $n$}, denoted by $\vburn n$, is the class of all $n$-periodic groups.
The study of this variety was initiated by W.~Burnside in 1902 who asked whether a finitely generated group in $\vburn n$ is necessarily finite \cite{Burnside:1902vi}. 
Burnside's problem inspired a number of significant developments in combinatorial and geometric group theory throughout the twentieth century and has been resolved negatively in the case that $n$ is large enough (see Novikov-Adian \cite{NovAdj68c}, Ol'shanski\u\i\ \cite{Olc82}, Lysenok \cite{Lys96}, Ivanov \cite{Iva94}, Delzant-Gromov \cite{DelGro08}, and Coulon \cite{Coulon:2014fr}). Despite much progress, many aspects of Burnside varieties remain unexplored, in part owing to the fact that it is generally a non-trivial task to even write down a single new example of an infinite $n$-periodic group. 
One purpose of our paper is to remedy this difficulty.

\medskip
This article provides a versatile and easy-to-apply tool for constructing examples of finitely generated infinite $n$-periodic groups with prescribed properties. While examples of such groups have already appeared in the literature, their constructions rely on heavy technical machinery \cite{Adian:1979th,Olc91a} involved to solve the Burnside problem. 
By contrast, the tool we provide can be applied without any prior knowledge on $n$-periodic groups.

\medskip
We develop a small cancellation theory in Burnside varieties which is the exact analogue of the usual $C'(1/6)$--small cancellation theory, in its classical forms and its recent graphical generalizations.
Recall that small cancellation theory has, ever since the mid--twentieth century, provided a seemingly unending source of -- often very explicit -- examples of infinite groups with striking properties in a wide range of contexts. 
It has produced results on solvability and unsolvability of algorithmic decision problems \cite{Greendlinger:1960jt,Rips:1982coa}, it has made major contributions to the understanding of features of negative and non-positive curvature in groups \cite{Gro87,Pankratev:1999vo,Wis04,Gruber:2014wo,Arzhantseva:2016uk}, it plays a major role in the understanding of random groups \cite{Gromov:1993vu}, and it has provided the only source of groups containing expanders in their Cayley graphs and therefore satisfying exceptional analytic properties \cite{Gro03,ArzDel08,Osajda:2014uk}.
Our result thus makes the method for constructing such groups available in a Burnside variety. 
Before stating our main theorem precisely, we motivate our work with a few applications.

\paragraph{Notations.}
In this article $\N$ (\resp $\N^*$) stands for the set of non-negative (\resp positive) integers.

\paragraph{Periodic monster groups.}
Small cancellation theory is a powerful tool for exhibiting groups with exotic properties. 
Gromov's monster groups are such examples.
Using a graphical version of small cancellation theory, Gromov built finitely generated groups that coarsely contain expander graphs in their Cayley graphs \cite{Gro03}, see also \cite{ArzDel08,Osajda:2014uk}.
As a consequence, they do not coarsely embed into Hilbert spaces, and therefore they do not have Yu's property A \cite{Yu2000}, and they are counterexamples to the Baum-Connes conjecture with coefficients \cite{HLS2002}. Gromov's groups and related constructions are currently the only source of examples of groups with any of these three properties. While these constructions necessarily produce groups with infinite order elements, we are able to show that such Gromov's monsters also exist in Burnside varieties.

\begin{theo}[Gromov's monster, see \autoref{res: Gromov monster}]
	There exists $n_0 \in \N$ such that for every odd exponent $n \geq n_0$, there exists a group $G \in \vburn n$ generated by a finite set $S$ such that the Cayley graph of $G$ with respect to $S$ contains an embedded (and, moreover, coarsely embedded) expander graph.
	In particular, there exists a finitely generated $n$-periodic group that does not have Yu's property A, that does not coarsely embed in a Hilbert space and that does not satisfy the Baum-Connes conjecture with coefficients.
\end{theo}

\paragraph{Decision problems in Burnside varieties.}
One very important question in group theory is to understand what properties of a group can be checked algorithmically.
The word problem is probably one of the most famous instances of this question. 
For a group $G$ given by a finite (or recursively enumerable) presentation $\pres SR$, it asks if there exists an algorithm which can decide whether or not a word in the alphabet $S \cup S^{-1}$ represents the identity element.
It was proved by Novikov that there exists a group with unsolvable word problem \cite{Novikov:1955uq}.
Building on this example, Adian and Rabin showed the following fact \cite{Adian:1958tx,Rabin:1958vv}.
Given a Markov property $\mathcal P$, there is no algorithm that takes a finite presentation and decides whether or not the corresponding group has $\mathcal P$.
Roughly speaking, this means that most of the non-trivial decision problems one can think of are unsolvable in the class of all finitely presented groups.
However if one restricts attention to a smaller class (e.g. abelian groups, nilpotent groups, hyperbolic groups, etc), then many decision problems become solvable. It is therefore natural to ask what decision problems can be solved in a Burnside variety $\vburn n$.

\medskip
We show here the exact analogue of Adian-Rabin theorem in $\vburn n$. 
As it remains unknown whether there exists an infinite finitely presented periodic group, in our approach of decision problems we consider groups which are finitely presented relative to $\vburn n$. To that end we briefly give terminology for our statement.
Given a group $G$, we write $G^n$ for the (normal) subgroup of $G$ generated by the $n$-th power of all its elements.
Recall that the \emph{free Burnside group of exponent $n$ generated by a set $S$}, is defined by 
\begin{displaymath}
	\burn Sn = \F(S)/\F(S)^n = \pres S{x^n = 1,\, \forall x}.
\end{displaymath}
It is the free element of $\vburn n$. 
This means that given a group $G \in \vburn n$, any map $S \to G$ uniquely extends to a homomorphism $\burn Sn \to G$.
We say that a group $G \in \vburn n$ is \emph{finitely presented relative to $\vburn n$} if $G$ is isomorphic to the quotient of a finitely generated free Burnside group $\burn Sn$ by the normal closure of a finite subset of $\burn Sn$.
Equivalently, $G$ is the quotient of a finitely presented group $G_0 = \pres SR$ by $G_0^n$.
In this situation we will refer to $\pres SR$ as a \emph{finite presentation relative to $\vburn n$}.

\begin{theo}[see \autoref{res: unsolvability}]
	There exists a critical exponent $n_0 \in \N$ with the following property.
	Let $n \geq n_0$ be an odd integer that is not prime.
	Let $\mathcal P$ be a subclass of $\vburn n$ for which there exist $G_+,G_- \in \vburn n$ which are finitely presented relative to $\vburn n$ such that 
	\begin{enumerate}
		\item the group $G_+$ belongs to $\mathcal P$,
		\item any $n$-periodic group containing $G_-$  as a subgroup does not belong to $\mathcal P$.
	\end{enumerate}
	Then there is no algorithm that takes as input a finite presentation relative to $\vburn n$ and determines whether the corresponding group $G \in \vburn n$ belongs to $\mathcal P$ or not.
\end{theo}

Our proof relies on a result of Kharlampovich, who showed that if $n$ is a sufficiently large exponent that is not prime, then there exists a group $G \in \vburn n$ which is finitely presented relative to $\vburn n$ with unsolvable word problem \cite{Kharlampovich:1995ea}. 
The following is an immediate consequence of our theorem.

\begin{coro}[see \autoref{res: unsolvability - applications}]
There exists a critical exponent $n_0 \in \N$, such that for every odd integer $n\geq n_0$ which is not prime, the following properties are algorithmically undecidable from finite presentations relative to $\vburn n$: begin trivial, finite, cyclic, abelian, nilpotent, solvable, amenable. 
\end{coro}

Another famous application of small cancellation theory is the Rips' construction \cite{Rips:1982coa}.
It shows that every finitely presented group $Q$ is the quotient of a (hyperbolic) small cancellation group $G$ by a finitely generated normal subgroup $N$.
In particular it allows to transfer various pathological properties of $Q$ to the subgroups of $G$ (see for instance Baumslag-Miller-Short \cite{Baumslag:1994kn}).
Our approach to small cancellation theory over Burnside groups provides an analogue of Rips' construction (see \autoref{res: rips construction}).

\paragraph{New proofs of known results.} Our main result also enables us to give explicit constructions that provide concise new proofs of the following known results.
\begin{theo}[Atabekyan \cite{Atabekyan:1987vz}, see \autoref{exa: thue-morse}]
	There exists  $n_0 \in \N$, such that for every odd integer $n \geq n_0$, for every $r \geq 2$, there are uncountably many groups of rank $r$ in $\vburn n$.
\end{theo}
	A group $G$ is \emph{SQ-universal in $\vburn n$} if for every countable group $C\in \vburn n$ there exists a quotient $Q$ of $G$ such that $C$ embeds into $Q$.
\begin{theo}[Sonkin \cite{Sonkin:2003il}, see \autoref{theo:sq}]
	There exists $n_0 \in \N$ such that for every odd exponent $n \geq n_0$, for every set $S$ containing at least two elements, the free Burnside group $\burn Sn$ is SQ-univeral in $\vburn n$.
\end{theo}

\paragraph{The small cancellation theorem.} 
Let us present now a simplified version of our main result. 
We use the usual definition of the classical $C'(\lambda)$-condition \cite[Chapter V]{LynSch77}; see also \autoref{defi: small cancellation classical}. 
Roughly speaking, this condition on a presentation requires that whenever two relators $r\neq r'$ have a common subword $u$, then $|u|<\lambda\min\{|r|,|r'|\}$.

\begin{theo}\label{res: main theo - classical}
Let $p\in\N^*$. There exists $n_p\in\N$ such that for every odd integer $n\geq n_p$ the following holds.
Let $G = \pres SR$ be a non-cyclic group given by a classical $C'(1/6)$-presentation. 
Assume that no $p$-th power of a word is a subword of an element of $R$, and no $r\in R$ is a proper power.
Then the quotient $G/G^n$ is  infinite. 
\end{theo}

\medskip
Note that there is no restriction on the cardinalities of $S$ or $R$ or on the length of the relations in $R$. 
The constant $n_p$ only depends on $p$.
In practice, $p$ will never be larger than $10$.
Hence, $n_p$ can be thought of as a universal constant.

\medskip
Our proof, in fact, yields the much more general Theorems~\ref{res: main theo - regular sc} and \ref{res: main theo - product sc}, encompassing Gromov's \emph{graphical} small cancellation theory, as well as classical and graphical small cancellation theory over free products. 
The philosophy is always similar to the one of \autoref{res: main theo - classical}: if the small cancellation presentation defines a non-elementary group, and if some  restrictions on proper powers are satisfied, then some of the standard conclusions of small cancellation theory hold. 
For example, $n$-periodic graphical small cancellation produces infinite $n$-periodic groups with prescribed subgraphs in their Cayley graphs, and $n$-periodic free product small cancellation produces infinite $n$-periodic quotients of free products of $n$-periodic groups in which each of the generating free factors survives as subgroup.

\begin{rema} 
Even in the case that both $S$ and $R$ are finite, the statement of \autoref{res: main theo - classical} is \emph{not} covered by prior results. It is known \cite{Olc91,DelGro08} that, given a torsion-free Gromov hyperbolic group $G$, there exists $n_G \in \N$ such that for all odd integers $n\geq n_G$, the quotient $G/G^n$ is infinite. Note here that $n_G$ depends on the specific group $G$ and, in fact, given an exponent $n\in \N$, the proof only applies to finitely many hyperbolic groups $G$ of a given rank. 
In our result, on the other hand, the constant $n_p$ is independent of the presentation $\pres SR$. 
We shall see in the following how this easily enables us to construct, for any $n\geq n_p$ odd, \emph{infinitely many} (and, in fact, uncountably many) examples of infinite $n$-periodic groups (say, of rank 2).
\end{rema}

\paragraph{Strategy of proof.} Our small cancellation assumption in \autoref{res: main theo - classical} has two parts. 
The first is the usual $C'(\lambda)$ small cancellation condition requiring that any two elements of $R$ must have small common subwords -- this condition is a standard tool for producing infinite ``negatively curved'' groups. 
Together with the assumption that no relation $r\in R$ is a proper power, it also ensures that $G$ is torsion-free.
This prevents $G$ from having torsion that would be incompatible with the $n$-torsion introduced later on.

\medskip
The second part of our assumption requires a uniform bound on the powers appearing as subwords of elements of $R$. This essentially tells us that the elements of $R$ are ``transverse to the Burnside relations'' of the form $x^n=1$. For some specific adhoc constructions, experts of Burnside groups implicitly observed that adding appropriate aperiodic relations should not affect significantly the proof of the infiniteness of $\burn Sn$ \cite{Olc91a}.
However their method requires to re-run the full proof of the Novikov-Adian theorem.
Our approach of \autoref{res: main theo - classical} has the following advantage: we are able to treat completely separately the relations of $G/G^n$ coming from $R$ and the Burnside relations. 
In particular, our approach provides a better geometric understanding of the bounded torsion groups that arise as small cancellation quotients of free Burnside groups.
We now briefly describe the strategy of our proof.

\medskip 

The first step is to study the geometry of groups defined by $C'(1/6)$ small cancellation presentations. 
If $G = \pres SR$ is such a group, denote by $W$ all elements of $G$ represented by subwords of elements of $R$. Consider the Cayley graph $\dot X$ of $G$ with respect to the (possibly infinite) generating set $S\cup W$. Thus, as graph, $\dot X$ is obtained from the usual Cayley graph by attaching to each embedded cycle $\gamma$ that corresponds to a relator in $R$ a complete graph on the vertices of $\gamma$.
Gruber and Sisto proved that $\dot X$ is Gromov hyperbolic, and the natural action of $G$ on $\dot X$ is non-elementary unless $G$ is virtually cyclic \cite{Gruber:2014wo}. In our work, we study the action of $G$ on $\dot X$ and prove that, in fact, our second assumption (the one regarding powers appearing as subwords of relations) ensures that the action of $G$ on $\dot X$ is acylindrical. We also determine the non-trivial elliptic elements for the action of $G$ on $\dot X$.
In the case of \autoref{res: main theo - classical}, we show that there are none.

\medskip

The second step of the proof is to ``burnsidify'' the group $G$ by killing all possible $n$-th powers.
To this end we use a theorem of Coulon that can be roughly summarized as follows \cite{Coulon:2016if}.
Assume that $G$ is a group without involution (i.e.\ element of order 2) acting acylindrically non-elementary on a hyperbolic space $\dot X$. Then there exists a critical exponent $n_0 \in \N$ such that for every odd integer $n\geq n_0$, there exists an \emph{infinite} quotient $G_n$ of $G$ with the following properties: every elliptic subgroup of $G$ embeds into $G_n$, and for every element $g \in G_n$ that is not the image of an elliptic element of $G$, we have $g^n = 1$.
In the settings of \autoref{res: main theo - classical}, $G$ has no non-trivial elliptic element, hence the quotient $G_n$ obtained from Coulon's result is exactly $G/G^n$. 
Coulon's theorem, moreover, proves that the map $G\to G_n$ preserve the small scale geometry induced by the metric of $\dot X$.
This lets us prove that certain non-trivial elements of $G$ survive in $G_n$.  
In particular, it enables us to promote small cancellation constructions of groups with certain prescribed subgraphs (e.g. expander graphs) to the setting of $n$-periodic groups.

\medskip

We explain why the critical exponent $n_p$ in \autoref{res: main theo - classical} only depends on $p$.
The hyperbolicity constant of Gruber and Sisto's space $\dot X$ is uniform, i.e. independent of the specific $C'(1/6)$-presentation $\pres SR$. Our proof shows that the acylindricity parameters for the action of $G$ on $\dot X$ only depend on the number $p$ that provides the bound on powers appearing as subwords of relators. Finally, the critical exponent given by Coulon's theorem only depends on the those parameters.
In other words, the uniform control on powers appearing as subwords of elements of $R$, i.e.\ the degree of transversality to the Burnside relators, has a very strong geometric interpretation in term of the action of $G$ on $\dot X$, which gives us the desired control on $n_p$.

\paragraph{Outline of the article.}
In \autoref{sec: main results} we state the main results of the article and give a proof of the applications presented in the introduction.
\autoref{sec: hyperbolic} reviews some basic facts about hyperbolic geometry and acylindrical actions on hyperbolic spaces.
\autoref{sec: acylindrical to burnside} is devoted to the second step of the aforementioned strategy. 
We explain how given a group $G$ acting acylindrically on a hyperbolic space $X$, we can turn $G$ into a periodic group with exponent $n$.
In particular we highlight the fact that the exponent $n$ does not depend on the group $G$, but only on the parameters of the acylindrical action.
In \autoref{sec: sc to acylindrical} we show that a group $G$ statisfying a suitable small cancellation condition acts acylindrically on its hyperbolic coned-off Cayley graph.
Moreover we explain how the parameters of this action are related to those of our small cancellation assumption.
The final section provides the proofs of the main theorems stated in \autoref{sec: main results} using the results of \autoref{sec: acylindrical to burnside} and \autoref{sec: sc to acylindrical}.

\paragraph{Acknowledgement.}
This material uses work supported by the National Science Foundation under Grant No. DMS-1440140 while the authors were in residence at the Mathematical Sciences Research Institute in Berkeley, California, during the fall 2016 semester.
The authors acknowledge the support of the Newton Institute for Mathematical Sciences in Cambridge, United Kingdom,  during the  spring 2017 program
\emph{Non-positive curvature: group actions and cohomology}. This work is partially supported by EPSRC Grant Number EP/K032208/1.
The authors were also supported by the Swiss-French Grant PHC Germaine de Staël 2016 - 35033TB.
The first author acknowledges support from the Agence Nationale de la Recherche under Grant Dagger ANR-16-CE40-0006-01.
The second author acknowledges support from the Swiss National Science Foundation Professorship FN
PP00P2-144681/1.
The authors thank for the referees for many useful comments and corrections

%
\section{Main results}
%
\label{sec: main results}

In this section we state the main results of the article and explain how they can be applied to cover all the examples presented in the introduction.

%
\subsection{Small cancellation}\label{section21}
%

Before stating its graphical generalization, we remind the reader of the classical $C'(\lambda)$-condition for reference. Our definition corresponds to the standard one given, for example, in the book of Lyndon and Schupp \cite[Chapter V]{LynSch77}.

\begin{defi}[Classical small cancellation]
\label{defi: small cancellation classical}
Given a presentation $\pres SR$, a \emph{piece} is a word $u$ that is a common prefix of two distinct cyclic conjugates of elements of $R\cup R^{-1}$. 
We say $\pres SR$ satisfies the classical $C'(\lambda)$-condition if every element of $R$ is cyclically reduced and if whenever a classical piece $u$ is a subword of a cyclic conjugate of some $r\in R$, then $|u|<\lambda|r|$.
\end{defi}

\paragraph{Group defined by a labelled graph.}
A graph is a graph in the sense of Serre \cite{Serre:1977wy}. This means a graph $\Gamma$ is an ordered pair of sets $(V,E)$ together with maps $\iota,\tau:E\to V$ (initial vertex and terminal vertex) and a fixed-point free involution $\cdot^{-1}:E\to E$ (edge inversion) such that for each $e\in E$ we have $\tau e=\iota (e^{-1})$.

Let $S$ be a set.
A \emph{labelled} graph is a graph $\Gamma = (V, E)$ together with a map $E\to S\sqcup S^{-1}$ that is compatible with the inversion map.
If we write $\Gamma$ as disjoint union of its connected components $\Gamma=\sqcup_{i\in I}\Gamma_i$ and, for each $\Gamma_i$ choose a vertex $v_i$, then the labelling induces a natural map $*_{i\in I}\pi_1(\Gamma_, v_i) \to \F(S)$.
The group $G(\Gamma)$ is defined as the quotient of $\F(S)$ by the normal closure of the image of this map.
Note that this normal closure, and hence $G(\Gamma)$, does not depend on the choice of the vertices $v_i$.

\medskip
The group $G(\Gamma)$ can be also described by a specific presentation.
Given a path $\gamma$ in $\Gamma$, the \emph{label} of $\gamma$, denoted $\ell(\gamma)$, is the product of the labels of its edges seen as an element of the free monoid on $S\sqcup S^{-1}$.
A presentation of $G(\Gamma)$ is
\begin{displaymath}
	G(\Gamma):=\left\langle S\mid \text{labels of simple closed paths in }\Gamma\right\rangle
\end{displaymath}
Let $\Cay(G(\Gamma),S)$ be the Cayley graph of $G(\Gamma)$ with respect to $S$.
If $v_i$ is a vertex in a connected component $\Gamma_i$ of $\Gamma$ and $g\in G(\Gamma)$, then there exists a unique label-preserving graph homomorphism $\Gamma_i\to\Cay(G(\Gamma),S)$ that maps $v_i$ to $g$. 

\paragraph{Small cancellation condition.}
Let $\Gamma$ be a graph labelled by a set $S$.
A \emph{piece} is a word $w$ over the alphabet $S\sqcup S^{-1}$ labelling two paths $\gamma_1$ and $\gamma_2$ in $\Gamma$ so that there is no label preserving automorphism $\phi$ of $\Gamma$ such that $\gamma_2 = \phi \circ \gamma_1$.

\begin{defi}[Graphical small cancellation {\cite[Definition~1.3]{Gruber:2015fu}}]
\label{def: graphical small cancellation}
	Let $\lambda \in (0,1)$.
	Let $\Gamma$ be a graph labelled by a set $S$.
	We say that $\Gamma$ satisfies the $C'(\lambda)$ small cancellation condition if the following holds.
	\begin{enumerate}
		\item The graph $\Gamma$ is \emph{reduced} in the following sense: two edges with the same initial vertex  cannot have the same label.
		\item For every simple loop $\gamma$ in $\Gamma$, the $w$ is a piece labelling a subpath of $\gamma$, then $\abs w < \lambda \abs \gamma$ where $\abs{\, .\, }$ stands for the length of words/paths.
	\end{enumerate}
\end{defi}

This settings extends the classical small cancellation.
Indeed \autoref{defi: small cancellation classical} exactly corresponds to the case where $\Gamma$ is a disjoint union of circle graphs, each of them being labelled by a distinct relation $r \in R$.

\medskip
In order to state the next theorem we define a strengthening of the small cancellation condition.
To perform small cancellation in the variety $\vburn n$ we need indeed an assumption to ensure that the relation we consider are ``transverse to the Burnside relations''.
The idea is to require that the words labelling path in $\Gamma$ are not large power, unless they already corresponds to an $n$-th power labelling a closed path.

\begin{defi}[Periodic small cancellation]
\label{def: graphical small cancellation - burnside}
	Let $n, p \in \N$ and $\lambda \in (0,1)$.
	Let $\Gamma$ be a graph labelled by a set $S$.
	We say that $\Gamma$ satisfies the $C'_n(\lambda,p)$-small cancellation assumption if the following holds.
	\begin{enumerate}
		\item \label{enu: graphical small cancellation - burnside - reduced}
		The graph $\Gamma$ is \emph{strongly reduced} in the following sense: every edge has two distinct vertices; 
		there is no closed path with label $st$ or $st^{-1}$ for $s,t\in S, s\neq t$.
		\item \label{enu: graphical small cancellation - burnside - sc}
		 $\Gamma$ satisfies the $C'(\lambda)$-condition.
		\item \label{enu: graphical small cancellation - burnside - power}
		 Whenever $w$ is a cyclically reduced word such that $w^p$ labels a path in $\Gamma$, then $w^n$ labels a \emph{closed} path in $\Gamma$.
	\end{enumerate}
\end{defi}

Let $n \in \N$.
Following the construction above, we define the analogue of $G(\Gamma)$ in the Burnside variety $\vburn n$.
For each connected component $\Gamma_i$ of $\Gamma$ we choose a vertex $v_i \in \Gamma_i$.
The labelling $\ell$ of $\Gamma$ induces a map $*_{i\in I}\pi_1(\Gamma_, v_i) \to \burn Sn$.
The group $G_n(\Gamma)$ is defined as the quotient of $\burn Sn$ by the normal closure of the image of this map.
Again this construction does not depend on the choice of the vertices $v_i$.
The group $G_n(\Gamma)$ can also be described as the $n$-periodic quotient of $G(\Gamma)$.
More precisely, $G_n(\Gamma)$ is the quotient of $G(\Gamma)$ by the (normal) subgroup of $G(\Gamma)$ generated by the $n$-th power of all its elements.
As previously, if $v_i$ is a vertex in a connected component $\Gamma_i$ of $\Gamma$ and $g\in G_n(\Gamma)$, then there exists a unique label-preserving graph homomorphism $\Gamma_i\to\Cay(G_n(\Gamma),S)$ that maps $v_i$ to $g$.

\begin{theo}
\label{res: main theo - regular sc}
	Let $p \in \N^*$.
	There exists a critical exponent $n_p \in \N$ such that for every odd integer $n \geq n_p$ the following holds.
	Let $S$ be a set containing at least two elements.
	Let $\Gamma$ be a graph labelled by $S$ satisfying the $C'_n(1/6,p)$ condition.
	We assume that there is no finite group $F$ generated by $S$ whose Cayley graph $\Cay(F,S)$ embeds in $\Gamma$.
	Then the following holds.
	\begin{enumerate}
		\item \label{enu: main theo - regular sc - infinite}
		The group $G_n(\Gamma)$ is infinite.
		\item \label{enu: main theo - regular sc - embedding}
		 Every connected component of $\Gamma$ embeds into $\Cay(G_n(\Gamma),S)$ via a label-preserving graph homomorphism.	
		\item \label{enu: main theo - regular sc - coarse embedding}
		 If $S$ is finite, $\Gamma$ is countable, and every connected component of $\Gamma$ is finite, then $\Gamma$ embeds and coarsely embeds in $G_n(\Gamma)$.
	\end{enumerate}
\end{theo}

Recall that a map $f: X_1\to X_2$ between two metric spaces is a coarse embedding if for all sequences of pairs of points $x_k,y_k\subset X_1\times X_1$ we have $|x_k-y_k|_{X_1}\to\infty\Longleftrightarrow |f(x_k)-f(y_k)|_{X_2}\to\infty$. 
In our theorem, we will consider $\Gamma$ as a metric space by writing it as disjoint union of its countably many connected components $\Gamma=\sqcup_{i\in\N}\Gamma_i$, and endowing it with the shortest-path metric on each component and declaring $d(x,y)=\diam(\Gamma_i)+\diam(\Gamma_j)+i+j$ for $x\in\Gamma_i$ and $y\in\Gamma_j$. 
(This is usually called the \emph{box-space metric}, and the constants $\diam(\Gamma_i)+\diam(\Gamma_j)+i+j$ we choose are irrelevant in the context of the notion of coarse equivalence.)

\medskip
The proof of \autoref{res: main theo - regular sc} is given in \autoref{sec: proofs}.
We illustrate this statement with a short proof that for an integer $n$ odd and large enough, $\vburn n$ is uncountable.

\begin{exam}[Thue-Morse sequence]
\label{exa: thue-morse}
	Let $S = \{a,b,t\}$.
	Let $u = u(a,b)$ be the infinite Thue-Morse sequence over the alphabet $\{a,b\}$.
	\begin{displaymath}
		u(a,b) = abbabaabbaababbabaababbaabbabaabbaababbaabbabaababba\dots
	\end{displaymath}
	It is the infinite word obtained from $a$ by iterating the substitution $a \to ab$ and $b\to ba$.
	It has the property that it does not contain any subword of the form $w^3$ \cite{Thue:1912wo}.
	For every $k \in \N$, we consider a subword $u_k = u_k(a,b)$ of length $k$ of $u$.
	For every $i \in \N$, we now consider a collection of words $r_i$ of the form 
	\begin{displaymath}
		r_i=tu_{100i+1}tu_{100i+2}t\dots tu_{100i+100}.
	\end{displaymath}
	\medskip
	Let $n_p$ be the critical exponent given by \autoref{res: main theo - regular sc} for $p = 3$.
	We fix an odd integer $n \geq n_p$.
	Let $I$ be a subset of $\N$.
	We define a graph $\Gamma_I$ as the disjoint union of circle graphs indexed by $\N$.
	The $i$-th circle graph of $\Gamma_I$ is labelled by $r_i$ if $i \in I$ and $r_i^n$ otherwise.
	Hence we have 
	\begin{displaymath}
		G(\Gamma_I)=\pres{a,b,t}{r_i, r_j^n,\ i \in I, j \in \N \setminus I}.
	\end{displaymath}
	Observe that, since $|u_i|\neq|u_j|$ for $i\neq j$ and since no $u_i$ contains any $t$'s, we have that no cyclically reduced word of the form $w^3$ can be read on $\Gamma_I$, unless $w^n$ labels a closed path. Furthermore, no piece contains two $t$'s, whence it is easy to check that $\Gamma_I$ satisfies the $C'(1/6)$-condition. Therefore, $\Gamma_I$ satisfies the $C'_n(1/6,3)$-assumption.
	Applying \autoref{res: main theo - regular sc} yields that each $G_n(\Gamma_I)$ is infinite.
	Of course the relations of the form $r_j^n$ for $j \in \N \setminus I$ are irrelevant for the definition of $G_n(\Gamma_I)$.
	Nevertheless keeping track of them in $\Gamma_I$ will help us to distinguish between all the $G_n(\Gamma_I)$.	
	
	\medskip
	We now prove that we obtain uncountably many isomorphism classes of $n$-periodic groups in this way.
	For every $I \subset \N$ we denote by $K_I$ the kernel of the canonical projection $\F(a,b,t) \twoheadrightarrow G_n(\Gamma_I)$.
	As $n>1$, \autoref{res: main theo - regular sc}~\ref{enu: main theo - regular sc - embedding} applied to $\Gamma_I$ implies that whenever $j\notin I$, then $r_j$ does not represent the identity in $G_n(\Gamma_I)$. 
	Consequently $K_I = K_J$ if and only if $I = J$.
	Now, given one (isomorphism type of) countable group $C$, there are only countably many kernels of homomorphisms $\F(a,b,t)\to C$. Hence, the collection $\{G_n(\Gamma_I):I\subseteq \N\}$ must contain uncountably many isomorphism types of groups.
\end{exam}

\autoref{exa: pride}, which is detailed at the end ot the article, explains why the assumption regarding powers in \autoref{def: graphical small cancellation - burnside}~\ref{enu: graphical small cancellation - burnside - power} is necessary.

%
\subsection{Small cancellation over free products}\label{section22}
%

We now describe an analogue of \autoref{res: main theo - classical} in the context of small cancellation over free products. We refer the reader to Theorem~\ref{res: main theo - product sc} for the full (graphical) statement.
We follow here the general exposition given in Lyndon-Schupp \cite{LynSch77}.
Let 
\begin{displaymath}
	F = F_1 \ast F_2 \ast \dots \ast F_m
\end{displaymath}
be a free product.
Recall that any non-trivial element $g\in F$ can be decomposed in a unique way as a product
\begin{displaymath}
	g = g_1g_2 \dots g_\ell
\end{displaymath}
where each $g_i$ is a non trivial element of some factor $F_k$, and no two consecutive $g_i$ and $g_{i+1}$ belong to the same factor.
Such a decomposition is called the \emph{normal form} of $g$.
The integer $\ell$ is the \emph{syllable length} of $g$ and we denote it by $\abs{g}_*$.
Let $g = g_1g_2 \dots g_r$ and $h = h_1h_2 \dots h_s$ be two elements of $F$ given by their normal forms.
The product $gh$ is called \emph{weakly reduced} if $g_rh_1$ is non trivial.
Note that one allows $g_r$ and $h_1$ to be in the same factor.
An element $g = g_1\dots g_\ell$ given by its normal form is \emph{cyclically reduced} if $\ell = 1$ or $g_1$ and $g_\ell$ are not in the same factor.
It is \emph{weakly cyclically reduced} if $\ell = 1$ or $g_\ell g_1$ is non-trivial.
A subset $R$ of $F$ is called \emph{symmetrized} if its elements are cyclically reduced and for every $r \in R$, all the weakly cyclic conjugates of $r$ and $r^{-1}$ belong to $R$.
In this context a \emph{piece} is an element $u\in F$ for which there exist $r_1 \neq r_2$ in $R$ which can be written in weakly reduced form as $r_1 = uv_1$ and $r_2 = uv_2$.

\begin{defi} 
\label{def: power-free small cancellation condition - free product}
	Let $\lambda \in(0,1)$.
	A symmetrized subset $R$ of $F$ satisfies the \emph{power-free $C'_*(\lambda)$ small cancellation condition} if
	\begin{enumerate}
		\item \label{enu: small cancellation condition - free product - pieces}
		For every $r \in R$, which can be written in a weakly reduced form as $r = uv$ where $u$ is a piece, then $\abs{u}_* < \lambda \abs{r}_*$.
		To avoid pathologies we also require that for every $r \in R$,  $\abs{r}_* > 1/\lambda$. 
		(For instance, without this assumption a set of the form $R = \{f_1,\dots, f_m\}$ where $f_i$ belongs to the free factor $F_i$ would satisfy the $C'_*(\lambda)$ condition.
		If each factor $F_i$ is cyclic, the quotient $G = F/ \normal R$ could be trivial.
		We do not want such a situation to fall in the context of small cancellation theory.)
		\item \label{enu: small cancellation condition - free product - torsion-free}
		No element of $R$ is a proper power.
	\end{enumerate}
\end{defi}

By analogy with the graphical case we now define the assumption needed to perform small cancellation over free products in a Burnside variety $\vburn n$.
Our assumption here is somehow stronger than its graphical analogue.
Indeed we do not allow relations to be a proper power.
Hence the definition does not depend on the Burnside exponent $n$. 
We actually cover a more general free product small cancellation condition that does allow relators that are proper powers in \autoref{res: main theo - product sc} but choose to here cover the following version for simplicity.
\begin{defi}
\label{def: small cancellation condition - free product - burnside}
	Let $p \in \N$.
	Let $\lambda \in (0,1)$.
	A symmetrized subset $R$ of $F$ satisfies the \emph{power-free $C'_*(\lambda,p)$ small cancellation condition} if the following holds.
	\begin{enumerate}
		\item $R$ satisfies the power-free $C'_*(\lambda)$-assumption
		\item For every $r \in R$, if there exists a cyclically reduced element $w$ with $\abs{w}_* > 1$ such that $r$ can be written in a weakly reduced form as $r = (w^k)v$, then $k \leq p$.
	\end{enumerate}

\end{defi}

\begin{theo}
\label{res: periodic quotient of sc groups - free product}
	Let $p \in \N^*$.
	There exists $n_p \in \N$ such that for every odd exponent $n \geq n_p$ the following holds.
	Let $F = F_1\ast \dots \ast F_m$ be a free product.
	Let $R$ be a symmetrized subset of $F$ satisfying the power-free $C'_*(1/6,p)$ condition.
	We assume that no factor in $F$ has even torsion.
	Then there exists a quotient $Q_n$ of $G = F/\normal R$ with the following properties.
	\begin{enumerate}
		\item \label{enu: periodic quotient of sc groups- free product - elliptic embeds}
		Every factor $F_k$ embeds in $Q_n$.
		\item \label{enu: periodic quotient of sc groups- free product - periodicity}
		For every $g \in Q_n$, if $g$ is not conjugate to an element in a factor $F_k$, then $g^n = 1$.
		\item \label{enu: periodic quotient of sc groups- free product- universal property}
		If every factor $F_k$ belongs to $\vburn n$, then $Q_n = G/G^n$.
		In particular $Q_n$ belongs to $\vburn n$.
	\end{enumerate}
\end{theo}

The proof of \autoref{res: periodic quotient of sc groups - free product} is given in \autoref{sec: proofs}.

%
\subsection{Decision problems and Rips' construction}
%

We now present a few applications of Theorems~\ref{res: main theo - regular sc} and \ref{res: periodic quotient of sc groups - free product}.
We start with the following analogue the Adian-Rabin theorem.

\begin{theo}
\label{res: unsolvability}
	There exists a critical exponent $n_0$ with the following property.
	Let $n \geq n_0$ be an odd integer that is not prime.
	Let $\mathcal P$ be a subset of $\vburn n$ for which there exist $G_+,G_- \in \vburn n$ which are finitely presented relative to $\vburn n$ such that 
	\begin{enumerate}
		\item the group $G_+$ belongs to $\mathcal P$,
		\item any $n$-periodic group containing $G_-$ as a subgroup does not belong to $\mathcal P$.
	\end{enumerate}
	Then there is no algorithm that takes as input a finite presentation relative to $\vburn n$ and determines whether the corresponding group $G \in \vburn n$ belongs to $\mathcal P$ or not.
\end{theo}

\begin{proof}
	Recall that all the presentations we consider here are \emph{relative to $\vburn n$}.
	Nevertheless, for simplicity, we omit in this proof the mention ``relative to $\vburn n$''.
	Let $n_p \in \N$ be the critical exponent given by \autoref{res: periodic quotient of sc groups - free product} with $p = 10$.
	Up to increasing the value of $n_p$, we can assume that $n_p \geq (665)^2$.
	Let $n \geq n_p$ be an odd exponent that is not prime.
	In particular $n$ decomposes as $n = pq$ where $p \geq 3$ is prime and $q$ has an odd divisor larger that $665$.
	According to \cite{Kharlampovich:1995ea} there exists a finitely presented group $H \in \vburn n$ whose word problem is unsolvable.
	We write $S$ for the corresponding generating set of $H$.

	\medskip
	Let $\mathcal P$ be a class of groups satisfying the assumptions of the theorem.
	Let $G_+$ be a finitely presented group in $\mathcal P$.
	Let $G_-$ be a finitely presented group such that any group in $\vburn n$ containing $G_-$ is not in $\mathcal P$.
	We write $S_-$ and $S_+$ for the generating sets of the presentation defining $G_-$ and $G_+$ respectively.
	
	\medskip
	To each word $w$ over the alphabet $S \cup S^{-1}$ we are going to produce a finitely presented test group $L_n(w)$ in $\vburn n$ such that $L_n(w)$ belongs to $\mathcal P$ if and only if $w$ represents the trivial element.
	The construction goes as follows.
	We consider $C$, $C_1$ and $C_2$, three distinct copies of $\Z/n\Z$.
	We write $t$, $x_1$ and $x_2$ for a generator of $C$, $C_1$ and $C_2$ respectively.
	We consider the following free product.
	\begin{displaymath}
		L_0 = H \ast G_+ \ast G_- \ast C \ast C_1 \ast C_2
	\end{displaymath}

	\medskip
	Let $w$ be a word over $S \cup S^{-1}$. 
	We now construct a quotient $L(w)$ of $L_0$.
	Let $h$ the element of $H$ represented by $w$.
	We write $g_i$ for the commutator $g_i = [h,x_i]$.	
	Let $u(a,b)$ be the infinite Thue-Morse sequence over the alphabet $\{a,b\}$ (see \autoref{exa: thue-morse}).
	For every $k \in \N$, $u_k(a,b)$ is a subword of length $k$ of $u(a,b)$.
	The group $L(w)$ is the quotient of $L_0$ characterized by the following families of relations: for every $ s \in S \cup S_- \cup \{t, x_1, x_2\}$,
	\begin{equation}
	\label{enu: unsolvability - relation}
		s = u_{k_{s,1}}(g_1,g_2)tu_{k_{s,2}}(g_1,g_2)t^{-1} u_{k_{s,3}}(g_1,g_2)t\cdots u_{k_{s,2p_s}}(g_1,g_2)t^{-1}.
	\end{equation}
	where the sequence $(k_{s,j})$ will be made precise later.
	We now define $L_n(w)$ as the $n$-periodic quotient of $L$, i.e. $L_n(w) = L(w)/L(w)^n$, where $L(w)^n$ is the (normal) subgroup of $L(w)$ generated by the $n$-th power of all its elements.
	
	\medskip
 	If $w$ represents the trivial element in $H$, i.e. if $h=1$, then $g_1$ and $g_2$ are trivial as well.
	Hence the relations (\ref{enu: unsolvability - relation}) force $H$, $G_-$, $C$, $C_1$ and $C_2$ to have a trivial image in $L(w)$.
	Consequently $L(w)$ is isomorphic to $G_+$.
	As $G_+$ is already $n$-periodic, $L_n(w)$ is isomorphic to $G_+$.
	In particular $L_n(w)$ belongs to $\mathcal P$.
	Assume now that $w$ does not represents the trivial element in $H$.
	In other words, $h \neq 1$.
	Then $g_1$ and $g_2$ are two non-trivial elements of $H\ast C_1$ and $H\ast C_2$ respectively.	
	One can choose the sequences $(k_{s,j})$ in such a way that the relations (\ref{enu: unsolvability - relation}) defining $L(w)$ satisfy the power-free small cancellation assumption $C'_*(1/6,10)$.
	It follows from \autoref{res: periodic quotient of sc groups - free product} that $G_-$ embeds in $L_n(w)$.
	According to our assumption on $G_-$, the group $L_n(w)$ cannot belong to $\mathcal P$. 
	Hence the group $L_n(w)$ has the announced property.
	
	\medskip
	Note that the sequence $(k_{s,j})$ can be chosen independently of $w$.
	Hence the presentation of $L(w)$ can be algorithmically computed from the respective presentations of $G_\pm$, $H$, $C$, $C_1$ and $C_2$.
	It follows from this discussion that deciding whether or not a finitely presented group of $\vburn n$ belongs to $\mathcal P$ is equivalent to solving the word problem in $H$.
	The latter problem being unsolvable, so is the former one.
\end{proof}

Recall that free Burnside groups of sufficiently large odd exponents are non-amenable \cite{Adian:1982wj}.
Hence we get the following statement.

\begin{coro}
\label{res: unsolvability - applications}
	There exists a critical exponent $n_0$ with the following property.
	Let $n \geq n_0$ be an odd integer that is not prime.
	Let $\mathcal P$ be one of the following property: being trivial, finite, cyclic, abelian, nilpotent, solvable, amenable.
	There is no algorithm to determine whether a group $G$ of $\vburn n$ given by a finite presentation relative to $\vburn n$ has $\mathcal P$.
\end{coro}

The following observation shows that the Rips construction \cite{Rips:1982coa}, a method for exhibiting pathologies among $C'(\lambda)$-small cancellation groups (and thus, in particular, among hyperbolic groups) can also be applied to our class of small cancellation groups in the Burnside variety.

\begin{theo}[Rips construction]
\label{res: rips construction}
	Let $n \in \N$ and $\lambda \in (0,1]$.
	Let $Q$ be a finitely generated group in $\vburn n$.
	There exists a graph $\Gamma$ labelled by a finite set $S$ satisfying the $C'_n(\lambda,3)$-condition with the following properties
	\begin{enumerate}
		\item $G_n(\Gamma)$ maps onto $Q$ and the kernel of this projection is finitely generated.
		\item If $Q$ is finitely presented relative to $\vburn n$, then so is $G_n(\Gamma)$.
	\end{enumerate}
\end{theo}

\begin{proof}
	Let $\pres{a_1, \dots, a_r}R$ be a presentation relative to $\vburn n$ of $Q$.
	Up to conjugating the elements of $R$ we can assume that their length is a least $10^{10}$.
	As usual we write $u = u(x,y)$ for the infinite Thue-Morse word over the alphabet $\{x,y\}$.
	For every $k \in \N$, $u_k = u_k(x,y)$ is a subword of length $k$ of $u$ (see \autoref{exa: thue-morse}).
	We now build a graph $\Gamma$ labelled by $S = \{a_1, \dots, a_r, x,y,t\}$ as follows.
	For each $i \in \intvald1r$, for each $u \in \{x,y,t\}$ for each $\epsilon \in \{ \pm 1\}$ we associate a loop in $\Gamma$ labelled by 
	\begin{equation}
	\label{eqn: rips - conjugation relation}
		a_i^{\epsilon}ua_i^{-\epsilon}tu_{k_1}tu_{k_2}\dots tu_{k_{s_i}}.
	\end{equation}
	Let $r  = r_1\dots r_m$ be a relation of $R$ written in the alphabet $\{a_1, \dots, a_r\}$.
	For every such relation we add a new loop to $\Gamma$ labelled by 
	\begin{equation}
	\label{eqn: rips - onto relation}
		r_1u_{k_1}r_2u_{k_2}\dots r_mu_{k_m}.
	\end{equation}
	In this construction the sequences $(k_j)$ implicitly depends on the relation we are considering.
	Recall that the Thue-Morse sequence does not contain any cube.
	Hence one can choose the indices $k_j$ such that the graph $\Gamma$ obtained in this way satisfies the $C'_n(\lambda,3)$-assumption.
	According to the relations (\ref{eqn: rips - conjugation relation}) the subgroup $K$ generated by $x$, $y$ and $t$ is normal in $G_n(\Gamma)$.
	By the relations (\ref{eqn: rips - onto relation}) the quotient of $G_n(\Gamma)$ by $K$ is exactly $Q$.
	Note that, if $R$ is finite, then $\Gamma$ is a finite union of disjoint loops.
	Hence $G_n(\Gamma)$ is finitely presented relative to $\vburn n$.
\end{proof}

\begin{rema}
	In our proof the kernel $K$ has rank at most $3$.
	We made this choice to keep the exposition easy.
	With some additional work one should be able to achieve rank $2$.
	
	\medskip
	For the moment we have not found relevant applications of this construction that cannot be recovered by using already existing technologies.
	Consider for instance the following construction.
	Let $n$ be an odd integer and $Q \in \vburn n$.
	Applying the \emph{standard} Rips construction, one can find a short exact sequence 
	\begin{displaymath}
		1 \to K \to G \to Q \to 1
	\end{displaymath}
	where $G$ is a \emph{torsion-free hyperbolic} group and $K$ a finitely generated group.
	Then applying Ol'shanski\u\i\ \cite{Olc91} one can consider the quotient $G/G^{kn}$ where $k$ is a large odd integer and $G^{kn}$ the (normal) subgroup of $G$ generated by the $kn$-th power of all its elements.
	This provides a short exact sequence
	\begin{displaymath}
		1 \to K_k \to G/G^{kn} \to Q \to 1
	\end{displaymath}
	where $G/G^{kn}$ is an infinite periodic group.
	In this way one can exhibit for instance a periodic group with solvable word problem but unsolvable generalized word problem.
	Nevertheless the construction given by \autoref{res: rips construction} has the following virtue.
	Contrary to the aforementioned strategy, the group $G_n(\Gamma)$ that we produce has the same exponent as the group $Q$ we started with.
\end{rema}

%
\subsection{Gromov monsters}
%

Our main theorem also enables the transposition of major recent results from analytic group theory to the variety $\vburn n$.

\begin{theo}[Gromov monster]
\label{res: Gromov monster}
    There exists $n_0$ such that for every odd exponent $n \geq n_0$, there exists a group $G \in \vburn n$ generated by a finite set $S$ whose Cayley graph with respect to $S$ contains an embedded (and, moreover, coarsely embedded) expander graph.
	In particular, there exists a finitely generated $n$-periodic group that does not have Yu's property A, that does not coarsely embed into a Hilbert space and that does not satisfy the Baum-Connes conjecture with coefficients.
\end{theo}

\begin{proof}
	Let $\Gamma = (V,E)$ be an expander graph without simple closed paths of length 1 or 2.
	In particular $\Gamma$ is the disjoint union of a collection of finite graphs $(\Gamma_k)$ with uniformly bounded valency.
	In addition we assume that	
	\begin{itemize}
		\item the valence of any vertex is at least $3$;
		\item the girth of $\Gamma_k$ tends to infinity as $k$ approaches infinity;
		\item the ratio $\diam(\Gamma_k)/ \girth(\Gamma_k)$ is uniformly bounded from above.
	\end{itemize}
	Such a graph can be obtained as follows \cite{Lubotzky:2010cg}: fix the matrices 
	\begin{displaymath}
		M_1 = \left( \begin{array}{cc}
			1 & 2 \\
			0 & 1
		\end{array}\right)
		\quad \text{and} \quad
		M_2 = \left( \begin{array}{cc}
			1 & 0 \\
			2 & 1
		\end{array}\right)
	\end{displaymath}
	and consider the sequence of Cayley graphs of the groups $SL_2(\Z/p\Z)$, where $p$ runs over all primes, with respect to the images of $M_1$ and $M_2$.
	We explain how to endow $\Gamma$ with a labelling that satisfies, for every $n\in \N$, the $C_n'(1/6, 2)$ condition. 
	Then it remains to apply \autoref{res: main theo - regular sc}.
	
	\paragraph{Labelling the graph.}
	The ideas of this paragraph come from Arzhantseva, Cashen, Gruber, and Hume \cite[Theorem~6.8]{Arzhantseva:2016uk}.
	We recall them quickly for the convenience of the reader.
	Up to replacing $(\Gamma_k)$ by a subsequence, there exists finite set $S_1$ and a labelling of $\ell_1 \colon E \to S_1\sqcup S_1^{-1}$ satisfying the  $C'(1/6)$ small cancellation condition, such that $\Gamma$ has no non-trivial label-preserving automorphism. This is proven by Osajda in \cite[Theorem~2.7]{Osajda:2014uk}. 
	However some path may be labelled by some very large powers.
	Since $\Gamma$ has uniformly bounded degree, it follows from Alon, Grytczuk, Hałuszczak and Riordan \cite{Alon:2002bv} that there exists a finite set $S_2$ and another labelling of $\ell_2 \colon E \to S_2\sqcup S_2^{-1}$ such that no word labelling an embedded path of $\Gamma$ contains a square. 
	One checks easily that for every $n \in \N$, the product labelling $\ell \colon E \to (S_1\sqcup S_1^{-1}) \times (S_2\sqcup S_2^{-1})$ sending an edge $e$ to $(\ell_1(e), \ell_2(e))$ satisfies $C'_n(1/6,2)$ condition by observing that if $(w_1,w_2)$ is a piece for $\ell$, then $w_1$ is a piece for $\ell_1$, and if $(w_1,w_2)$ is square, then $w_2$ is a square.
\end{proof}

\subsection{A new proof of SQ-universality}	

\begin{theo}[SQ-universality]\label{theo:sq}
	There exists $n_0 \in \N$ such that for every odd exponent $n \geq n_0$, for every set $S$ containing at least two elements, the free Burnside group $\burn Sn$ is SQ-univeral in $\vburn n$.
	\end{theo}

The following proof builds on a construction used by Gruber in \cite[Example~1.13]{Gruber:2015vc} to show that any countable group embeds in a group of rank $2$ defined by a $C'(1/6)$-labelled graph.

\begin{proof}
Let $n_p$ be the critical exponent given by \autoref{res: main theo - regular sc} for $p = 6$.
Let $n \geq n_p$ be an odd integer. 
Let $A$ be a countable group in $\vburn n$.
We fix an epimorphism $\F(U) \twoheadrightarrow A$ where $U$ is an infinite countable set. 
We will define below a graph $\Gamma_0$ labelled by two letters with the following properties. (Of course it is sufficient to consider the case $|S|=2$, as any free Burnside group of rank greater than two surjects onto one of rank two.)
\begin{enumerate}
	\item The fundamental group of $\Gamma_0$ is isomorphic to the free group on $U$ (i.e.\ has countably infinite rank).
	\item The labelling satisfies the $C'_1(1/6,6)$-assumption. 	\item The label-preserving automorphism group of $\Gamma_0$ is trivial.
\end{enumerate}
Let $K$ be the kernel of the epimorphism $\F(U) \twoheadrightarrow A$. Let $\Gamma$ be the cover of $\Gamma_0$ corresponding to $K$, i.e. the quotient of the universal cover of $\Gamma_0$ by $K$. The labelling on $\Gamma_0$ induces a labelling on $\Gamma$. The labelling of $\Gamma$ satisfies the $C_n'(1/6,6)$-assumption: the $C'(1/6)$-condition is argued in \cite[Remark~1.12]{Gruber:2015vc}. The labels of paths in $\Gamma$ are exactly the labels of paths in $\Gamma_0$. Hence, if for a cyclically reduced word $w$ we have $w^6$ as label of a path in $\Gamma$, then $w$ labels a closed path in $\Gamma_0$. Since $A$ is $n$-periodic, we deduce that $w^n$ labels a closed path in $\Gamma$.
Hence, we may apply \autoref{res: main theo - regular sc}. 

Observe that $A$ embeds in $G(\Gamma)$ as subgroup preserving and acting freely on an embedded copy $\overline \Gamma$ in $\Cay(G(\Gamma),S)$ of the labelled graph $\Gamma$ \cite[Example~1.13]{Gruber:2015vc}. 
By \autoref{res: main theo - regular sc}, the canonical projection $f:\Cay(G(\Gamma),S)\to\Cay(G_n(\Gamma),S)$ restricted to $\overline\Gamma$ is an isomorphism. If $\pi$ denotes the epimorphism $G(\Gamma)\to G_n(\Gamma)$, then, for any $g\in G(\Gamma)$ and any vertex $v$ of $\Cay(G(\Gamma),S)$, we have $f(gv)=\pi(g)f(v)$. Thus, if $g\in A$ satisfies $\pi(g)=1$, then, for every $v\in\overline\Gamma$: $f(gv)=\pi(g)f(v)=f(v)$, whence $gv=v$ and $g=1$. Thus $A$ is a subgroup of $G_n(\Gamma)$.

\medskip

\noindent {\bf Construction of $\Gamma_0$.} We now define a graph $\Gamma_0$ labelled by $\{a,b\}$ satisfying the $C'_1(1/6,6)$-small cancellation condition, with infinite rank fundamental group and no non-trivial automorphism. 
In fact, no simple path in $\Gamma_0$ will be labelled by a 6-th power, and any non-simple path labelled by a proper power will be closed.

\medskip
We first define a sequence $u_k$ of subwords of the Thue-Morse sequence with $3(k-1)<\abs{u_k}\leq 3k$, and $u_k$ starts and ends with the letter $b$. Take any subword $\hat u_k$ of length $3(k-1)+1$ of the Thue-Morse sequence that starts with $b$. 
Now, since $a^3$ is not a subword of the Thue-Morse sequence, extending $\hat u_k$ by at most two letters yields a word $u_k$ as desired.

\medskip
Fix $N \geq 10^{10}$.
For every $i \in \N$, $i\geq 1$,  we let
\begin{displaymath}
	r_i=a^3u_{iN+1}a^3u_{iN+2}\dots a^3u_{iN+N}.
\end{displaymath}
These are cyclically reduced words by construction. 
Consider the collection $R= \set{r_i}{i>0}$, and consider $\hat \Gamma_0$ the disjoint union of cycle graphs $\Lambda_i$ labelled by $r_i$. Note that subwords of cyclic conjugates of elements of $R$ of the form $a^3$ occur exactly in the places written in the definition of the $r_i$ (i.e.\ they do not occur in any $u_k$ and contain no letters of any $u_k$). 
Since the $u_k$ have pairwise distinct lengths, this implies that any reduced piece contains at most one subpath labelled by $a^3$. Let $w$ be a piece such that $w$ is a subword of a cyclic conjugate of $r_i$. Then, by our observation, we have that $w$ is subword of a word of the form $a^2u_ka^3u_\ell a^2$ for suitable $k,\ell$ and, in particular $|w|\leq 8+ 6(iN+N)=6(i+1)N+8$. 
On the other hand, we have $\abs{r_i} > 3iN^2+3N$. 
We deduce $\abs w<\abs {r_k}/6 - 2$ because $N\geq 10^{10}$.

\medskip
On $\Lambda_i$, we can read the word $r_i$ starting from some base vertex. Let $v_i$ be the vertex separating $a$ and $a^2$ in the first occurrence of $a^3$ in $r_i$ and $w_i$ be the vertex separating $a$ and $a^2$ in the fourth occurrence of $a^3$. 
We now build a connected graph $\Gamma_0$ with a reduced labelling by connecting $v_i$ to $w_{i+1}$ by a line graph of length 3 labelled by the word $ba^{-1}b$, for each $i\geq 1$. 
We call $\Gamma_0$ the resulting graph.

\medskip
A word is \emph{positive} if any letter occuring in it occurs with positive exponent. Observe that, when checking pieces in $\Gamma_0$, we only need to check paths labelled by positive words, because any simple closed path is (up to inversion) labelled by a positive word.
Now the positive words on $\Gamma_0$ are exactly the word $a$, the positive words read on $\Lambda_i$ for $i\in \N$, and the words obtained from positive words of $\Lambda_i$ by possibly appending a $b$ at the beginning or at the end. 
Hence, the (relevant) pieces have increased in length by at most 2, and the $C'(1/6)$-condition is still satisfied.

\medskip
Now, by construction, no word read on a simple path in any of the $\Lambda_i$ for $i\geq 1$ is a 4-th power. 
We first argue that $\Gamma_0$ does not contain any \emph{simple} path $\pi$ labelled by a 6-th power: a positive word on $\Gamma_0$ corresponds to a positive word on some $\Lambda_i$ for $i\geq 1$ with possibly a $b$ appended in the beginning or at the end. Hence a positive word cannot contain a 6-th power. Let $\pi$ be labelled by a word that is not positive and whose inverse is not positive and that is a proper power. Then the label of $\pi$ contains, up to inversion, a subword of the form $a^{-1}wa^{-1}$, where $w$ is a positive word. The two occurrences of $a^{-1}$ correspond to the two line graphs attached to the same $\Lambda_i$ for $i\geq 2$. Thus, $w$ is of the form $bvb$, where $v$ is one of the two labels of simple paths from $v_i$ to $w_i$ in $\Lambda_i$. We already observed that such a path labelled by $v$ is not a piece, i.e.\ there is a unique path in $\Gamma$ labelled by $v$. Hence, the label of $\pi$ can contain at most one occurrence of $v$, contradicting that it is a proper power.

\medskip
If a \emph{non-simple} path is labelled by a proper power of a cyclically reduced word $w$, then we argue that $w$ labels a closed path.
Suppose that $\pi$ is a non-simple path whose label is $w^p$ for $p>1$ and $w$ cyclically reduced. 
As $w$ is cyclically reduced, $\pi$ is reduced, i.e.\ has no back-tracking. Thus, $\pi$ contains a subpath $\gamma$ that is a simple closed path. 
Now, because $p>1$, there exists a subword $u$ of $w^{p-1}$ labelling a subpath of $\gamma$ such that $|u|\geq |\gamma|/2$. 
(After all, the label of $\gamma$ is a subword of $w^p$.) 
But this and the small cancellation condition imply that $u$ cannot be a piece. Since the property of being a piece goes to subwords, we deduce that $w^{p-1}$ is not a piece. 
Hence, as $\Gamma_0$ does not admit any non-trivial label-preserving automorphisms by construction, any two paths labelled by $w^{p-1}$ must start from the same vertex. As $\Gamma_0$ contains a path labelled by $w^p$, we conclude that $w$ labels a closed path.
\end{proof}


%
\section{Hyperbolic geometry}\label{section:hyperbolic}
%
\label{sec: hyperbolic}

Let $X$ be a geodesic metric space.
Given two points $x,x' \in X$ we write $\dist[X] x{x'}$, or simply $\dist x{x'}$, for the distance between them.
The \emph{Gromov product} of three points $x,y,z \in X$  is defined by 
\begin{displaymath}
	\gro xyz = \frac 12 \left\{  \dist xz + \dist yz - \dist xy \right\}.
\end{displaymath}
For the remainder of this section, we assume that the space $X$ is \emph{$\delta$-hyperbolic}, i.e. for every $x,y,z,t \in X$,
\begin{equation}
\label{eqn: hyperbolicity condition 1}
	\gro xzt \geq \min\left\{ \gro xyt, \gro yzt \right\} - \delta.
\end{equation}
In this article we always assume the hyperbolicity constant $\delta$ is positive.
We write $\partial X$ for the Gromov boundary of $X$.
Note that we did not assume the space $X$ to be proper, thus we use the boundary defined with sequences converging at infinity  \cite[Chapitre~2, Définition~1.1]{CooDelPap90}.
A major fact of hyperbolic geometry is the stability of quasi-geodesics that we will use in the following form.

\begin{prop}{\rm \cite[Corollaries 2.6 and 2.7]{Coulon:2014fr}} \quad
\label{res: stability (1,l)-quasi-geodesic}
	Let $l_0 \geq 0$.
	There exists $L=L(l_0,\delta)$ which only depends on $\delta$ and $l_0$ with the following properties.
	Let $l \in \intval 0{l_0}$.
	Let $\gamma \colon I \rightarrow X$ be an $L$-local $(1,l)$-quasi-geodesic.
	\begin{enumerate}
		\item The path $\gamma$ is a (global) $(2,l)$-quasi-geodesic.
		\item For every $t,t',s \in I$ with $t \leq s \leq t'$, we have $\gro{\gamma(t)}{\gamma(t')}{\gamma(s)} \leq l/2 + 5 \delta$.
		\item For every $x \in X$, for every $y,y'$ lying on $\gamma$, we have $d(x,\gamma) \leq \gro y{y'}x + l + 8 \delta$.
		\item The Hausdorff distance between $\gamma$ and any other $L$-local $(1,l)$-quasi-geodesic joining the same endpoints (possibly in $\partial X$) is at most $2l+5\delta$.
	\end{enumerate}
\end{prop}
		
\paragraph{Remark.}
Using a rescaling argument, one can see that the best value for the parameter $L=L(l,\delta)$ satisfies the following property: for all $l,\delta \geq 0$ and $\lambda >0$, $L(\lambda l, \lambda \delta) = \lambda L(l,\delta)$.
This allows us to define a parameter $L_S$ that will be use all the way through.
	
\begin{defi}
\label{def: constant LS}
	Let $L(l,\delta)$ be the best value for the parameter $L=L(l,\delta)$ given in \autoref{res: stability (1,l)-quasi-geodesic}.
	We denote by $L_S$  a number larger than $500$ such that $L(10^5\delta,\delta) \leq L_S\delta$.
\end{defi}

%
\subsection{Group action on a hyperbolic space.}
%

Let $x$ be a point of $X$.
Recall that an isometry $g$ of $X$  is either \emph{elliptic}, i.e. the orbit $\langle g \rangle \cdot x$ is bounded, \emph{loxodromic}, i.e. the map from $\Z$ to $X$ sending $m$ to $g^m x$ is a quasi-isometric embedding or \emph{parabolic}, i.e. it is neither loxodromic or elliptic \cite[Chapitre~9, Théorème~2.1]{CooDelPap90}.
Note that these definitions do not depend on the point $x$.

\paragraph{WPD and acylindrical action.}
Let $G$ be a group acting by isometries on $X$.
For our purpose we need to require some properness for this action.
We will use two notions: weak proper discontinuity and acylindricity.

\begin{defi}[WPD action \cite{Bestvina:2002dr}]
\label{def: wpd action}
	The action of $G$ on $X$ is \emph{weakly properly discontinuous} if for every loxodromic element $g \in G$, the following holds:
	for every $x \in X$, for every $d \geq 0$, there exists $m \in \N$ such that the set of elements $u \in G$ satisfying $\dist{ux}x \leq d$ and $\dist{ug^mx}{g^mx}$ is finite.
\end{defi}

We now recall the definition of an acylindrical action on a metric space.
For our purpose we need to keep in mind the parameters that appear in the definition.

\begin{defi}[Acylindrical action]
\label{def: acylindricity}
	Let $N,L,d \in \R_+^*$.
	The group $G$ acts \emph{$(d,L,N)$-acylindrically} on $X$ if the following holds:
	for every $x, y \in X$ with $\dist xy \geq L$, the number of elements $u \in G$ satisfying $\dist{ux}x \leq d$ and $\dist{uy}y \leq d$ is bounded above by $N$.
	The group $G$ acts \emph{acylindrically} on $X$ if for every $d >0$ there exist $N, L > 0$ such that $G$ acts $(d,L,N)$-acylindrically on $X$.
\end{defi}

If the action of $G$ on $X$ is acylindrical, then it is also WPD.
Since $X$ is a hyperbolic space, one can decide whether an action is acylindrical by looking at a single value of $d$.

\begin{prop}[Dahmani-Guirardel-Osin {\cite[Proposition~5.31]{Dahmani:2017ef}}]
\label{res: acylindrical action vs hyp space}
	The action of $G$ on $X$ is acylindrical if and only if there exist $N,L >0$ such that the action is $(100\delta, L, N)$-acylindrical.
\end{prop}

\paragraph{Remark.}
F.~Dahmani, V.~Guirardel and D.~Osin work in a class of geodesic spaces.
Nevertheless, following the proof of \cite[Proposition~5.31]{Dahmani:2017ef} one observes that the statement also holds for length spaces.
Moreover one gets the following \emph{quantitative} statement.
Assume that the action of $G$ on $X$ is $(100\delta, L, N)$-acylindrical, then for every $d > 0$ the action is $(d, L(d), N(d))$-acylindrical where
\begin{eqnarray*}
	L(d) & = &  L + 4d + 100\delta, \\
	N(d) & = & \left(\frac d{5\delta} + 3 \right)N.
\end{eqnarray*}

\paragraph{Classification of group actions.}
We assume here that the action of $G$ on $X$ is WPD.
We denote by $\partial G$ the set of all accumulation points of an orbit $G \cdot x$ in the boundary $\partial X$.
This set does not depend on the point $x$.
One says that the action of $G$ on $X$ is 
\begin{itemize}
	\item \emph{elliptic}, if $\partial G$ is empty, or equivalently if one (hence any) orbit of $G$ is bounded;
	\item \emph{parabolic}, if $\partial G$ contains exactly one point;
	\item \emph{lineal}, if $\partial G$ contains exactly two points;
\end{itemize}
If the action of $G$ is elliptic, parabolic or lineal, we will say that this action is \emph{elementary}.
In this context, being elliptic (\resp parabolic, lineal, etc) refers to the action of $G$ on $X$.
However, if there is no ambiguity we will simply say that $G$ is elliptic (\resp parabolic, lineal, etc).
If $g$ is a loxodromic element of $G$, we write $g^-$ and $g^+$ for the repulsive and attractive points of $g$ in $\partial X$.
The subgroup $E^+(g)$ of $G$ fixing point-wise $\{g^-,g^+\}$ is a lineal subgroup. 
The set $F$ of all elliptic elements of $E^+(g)$ forms a normal subgroup of $E^+(g)$ such that $E^+(g)/F$ is isomorphic to $\Z$.
We say that $g$ is \emph{primitive} if its image in $E^+(g)/F$ is $\pm 1$.

%
\subsection{Invariants of a group action.}
%

In this section we recall several numerical invariants associated to a WPD action.
They will be useful to control the value of the critical exponent $n_p$ in Theorems~\ref{res: main theo - regular sc} and \ref{res: periodic quotient of sc groups - free product}.

\paragraph{Exponent of the holomorph.}
Let $F$ be a finite group.
Its \emph{holomorph}, denoted by $\hol F$, is the semi-direct product $\sdp F {\aut F}$, where $\aut F$ stands for the automorphism group of $F$.
The \emph{exponent} of $\hol F$ is the smallest integer $n$ such that $\hol F$ belongs to $\vburn n$.
Assume now that the action of $G$ on $X$ is WPD.
It is known that every lineal subgroup $E$ of $G$ is virtually cyclic.
In particular, it admits a maximal finite normal subgroup $F$ and $E/F$ is isomorphic either to $\Z$ of the infinite dihedral group $\dihedral$; see for instance \cite[Corollary~3.30]{Coulon:2016if}.

\begin{defi}
\label{res: invariant e}
	The number $e(G, X) \in \N \cup \{ \infty\}$ is the least common multiple of the exponents of $\hol F$, where $F$ runs over the maximal finite normal subgroups of all maximal lineal subgroups of $G$.
\end{defi}

\begin{rema}
\label{rem: cyclic lineal and holomorph}
If the lineal subgroups of $G$ are all cyclic, then $e(G,X) = 1$.
In general this quantity can be infinite.
However is the action of $G$ on $X$ is acylindrical, there exists $N \in \N$ the maximal finite normal subgroup of every lineal subgroup has cardinality at most $N$, see for instance \cite[Lemma~6.8]{Osin:2016gv}.
Therefore $e(G,X)$ is finite.
\end{rema}

\paragraph{Injectivity radius.}
To measure the action of an element $g \in G$ on $X$ we define the \emph{translation length} and the  \emph{asymptotic translation length} as
\begin{displaymath}
	\len[espace=X] g = \inf_{x \in X} \dist {gx}x, 
	\quad \text{and} \quad
	\len[espace=X,stable] g = \lim_{n \rightarrow + \infty} \frac 1n \dist{g^nx}x.
\end{displaymath}
These two lengths are related as follows \cite[Chapitre 10, Proposition 6.4]{CooDelPap90}.
\begin{equation}
\label{eqn: translation lengths}
	\len[stable]g \leq \len g \leq \len[stable] g + 16\delta.
\end{equation}

\begin{defi}[Injectivity radius]
\label{res: inj radius}
	The injectivity radius of $G$ on $X$ denoted by $\inj[X]G$ is the infimum of $\len[stable, espace=X]g$ over all loxodromic elements $g \in G$.
\end{defi}

\begin{lemm}[Bowditch {\cite[Lemma~2.2]{Bowditch:2008bj}}]
\label{res: lower bound inj radius}
	Let $L, N > 0$.
	Assume that the action of $G$ on $X$ is $(100\delta, L, N)$-acylindrical.
	Then the injectivity radius of $G$ on $X$ is bounded below by 
	\begin{displaymath}
		\inj[X]G \geq \delta/N
	\end{displaymath}
\end{lemm}

\paragraph{Remark.}
Although he does not provide a precise estimate, B.~Bowditch already stresses in his proof that the lower bound only depends on $\delta$, $N$ and $L$.
For the agreement of the reader we compute this lower bound.

\begin{proof}
	According to \autoref{res: acylindrical action vs hyp space}, the action of $G$ on $X$ is $(200\delta, L', N')$-acylindrical where $L' =   L + 900\delta$ and $N'  =  50N$.
	Let $g$ be a loxodromic element of $G$.
	According to (\ref{eqn: translation lengths}), it is sufficient to prove that $[g^{N'}] \geq 66\delta$.
	Assume on the contrary that it is false.
	There exists an $L_S\delta$-local $(1, \delta)$-quasi-geodesic $\gamma \colon \R \rightarrow X$ joining $g^-$ to $g^+$ \cite[Lemma~3.2]{Coulon:2016if}.
	Let $x$ and $y$ be two points on $\gamma$ such that $\dist xy \geq L'$.
	Let $j \in \intvald 0{N'}$.
	We know that $\gamma$ is contained in the $52\delta$-neighbourhood of the axis of $g^j$  \cite[Lemma~2.32]{Coulon:2014fr} hence $\dist{g^jx}x \leq \len{g^j} + 112\delta$.
	It follows from (\ref{eqn: translation lengths}) that 
	\begin{displaymath}
		\dist{g^jx}x \leq \len{g^j} + 112\delta \leq j\len[stable]g + 128\delta \leq N'\len[stable]g + 128\delta \leq \len{g^{N'}} + 128 \delta \leq 200 \delta.
	\end{displaymath}
	The same observation holds with $y$.
	It follows from the acylindricity that the set $\{1, g, g^2, \dots, g^{N'}\}$ contains at most $N'$ elements.
	Hence $g$ has finite order, which contradicts our assumption.
\end{proof}

\paragraph{The invariants $\nu$ and $A$.}
The critical exponent $n_p$ that appears in Theorems~\ref{res: main theo - regular sc} and \ref{res: periodic quotient of sc groups - free product} will depend on two more numerical invariants that are defined as follows.

\begin{defi}
\label{def: invariant nu}
	The invariant $\nu (G,X)$ (or simply $\nu$) is the smallest positive integer $m$ satisfying the following property. 
	Let $g$ and $h$ be two isometries of $G$ with $h$ loxodromic.
	If $g$, $h^{-1}gh$,..., $h^{-m}gh^m$ generate an elementary subgroup which is not loxodromic, then $g$ and $h$ generate an elementary subgroup of $G$.
\end{defi}

To any element $g \in G$, we associate to $g$ an \emph{axis} $A_g$ defined as the set of points $x \in X$ such that $\dist {gx}x  < \len g + 8 \delta$.
Given $\alpha \in \R_+$, we write $A_g^{+\alpha}$ for its $\alpha$-neighborhood.
If $g_0$, \dots, $g_m$ are $m$ elements of $G$ we denote by $A(g_0,\dots, g_m)$ the quantity
\begin{displaymath}
	A(g_0,\dots, g_m) = \diam \left(A_{g_0}^{+13\delta} \cap \dotsc \cap A_{g_m}^{+13\delta}\right).
\end{displaymath}
Recall that the parameter $L_S$ is the constant given by the stability of quasi-geodesics (\autoref{def: constant LS}).

\begin{defi}
\label{def: invariant A}
	Assume that $\nu = \nu(G,X)$ is finite.
	We denote by $\mathcal A$ the set of $(\nu+1)$-uples $(g_0,\dots,g_\nu)$ such that $g_0,\dots,g_\nu$ generate a non-elementary subgroup of $G$ and for all $j \in \intvald 0\nu$, $\len{g_j} \leq L_S\delta$. 
	The parameter $A(G,X)$ is given by
	\begin{displaymath}
		A(G,X) = \sup_{(g_0,\dots,g_\nu) \in \mathcal A} A\left(g_0, \dots,g_\nu\right).
	\end{displaymath}
\end{defi}

\begin{lemm}[Coulon {\cite[Lemmas~6.12 -- 6.14]{Coulon:2016if}}]
\label{res: upper bound nu and A}
	Let $L, N > 0$.
	Assume that the action of $G$ on $X$ is $(100\delta, L, N)$-acylindrical.
	Then the invariants $\nu(G,X)$ and $A(G, X)$ are bounded above as follows
	\begin{displaymath}
		\nu(G,X) \leq N\left(2 + \frac L{\delta}\right) \quad \text{and} \quad A(G,X) \leq 10L_S^2N^3(L + 5 \delta).
	\end{displaymath}
\end{lemm}

\paragraph{Remark.}
The statements in \cite{Coulon:2016if} do not mention an explicit upper bound for $\nu(G,X)$ and $A(G, X)$.
However the above inequalities directly follow from the proofs given therein.
Our estimates are very generous.
The important point to notice is that they only depend on $\delta$, $L$ and $N$.

\medskip
The purpose of these two invariants can be illustrated as follows.
The Margulis Lemma tells us that if $G$ is a discrete group of isometries of a simply-connected manifold with pinched negative curvature $X$, then there exists $\epsilon > 0$, so that for every $x \in X$, the subset $U(x) = \set{g \in G}{\dist {gx}x \leq \epsilon}$ generates a virtually nilpotent subgroup of $G$.
Such a statement does not hold any more if the curvature of $X$ is not bounded from below.
For instance it fails is $X$ is a tree or more generally a Gromov hyperbolic space.
Controlling $\nu$ and $A$ allows us to recover the following analogue of Margulis Lemma.

\begin{prop}[Coulon {\cite[Corollary~4.45]{Coulon:2016if}}]
\label{res: overlap multiples axes}
	Let $m$ be an integer such that $m \leq \nu (G,X)$.
	Let $g_0, \dots, g_m$ be $m+1$ elements of $G$.
	If they do not generate an elementary subgroup, then  
	\begin{equation*}
		A\left(g_0, \dots,g_m\right) \leq (\nu + 2)\sup_{0 \leq i \leq m} \len {g_i} + A(G,X) + 680\delta.
	\end{equation*}
\end{prop}

Observe that the quantity $A(g,h)$ is measures the overlap between the respective axes of $g$ and $h$, and thus can be thought of as a geometric measure of the length of the maximal piece between $g$ and $h$.
Controlling such quantities play a key role to produce infinite torsion groups by iterated small cancelation theory, see \cite{Coulon:2014fr,Coulon:2016if}


%
\section{From acylindrical action to periodic quotient}
%
\label{sec: acylindrical to burnside}

In this section we prove the following fact.
If a group $G$ admits a non-elementary acylindrical action on a hyperbolic space $X$, then one can exploit the negative curvature of $X$ to produce a (partially) periodic quotient of $G$.
A precise statement is given in \autoref{res: acylindricity gives uniform exp} below.
We want to stress the fact that the critical exponent $n_p$ appearing in \autoref{res: acylindricity gives uniform exp} only depends on the parameters of the action and not on the group $G$ or the space $X$.

\begin{prop}[compare with Coulon {\cite[Theorem~6.15]{Coulon:2016if}}]
\label{res: acylindricity gives uniform exp}
	Let $N, L, \delta, r >0$.
	There exists $N_1 \in \N$ such that the following holds.
	Let $G$ be a group acting $(100\delta, L, N)$-acylindrically on a $\delta$-hyperbolic length space $X$.
	We assume that $G$ is non-elementary, has no even torsion and $e(G,X)$ is odd.
	For every odd integer $n \geq N_1$ that is a multiple of $e(G,X)$, there exists a quotient $B_n$ of $G$ with the following properties.
	\begin{enumerate}
		\item \label{enu: acylindricity gives uniform exp - elliptic embeds}
		Every elliptic subgroup of $G$ embeds in $B_n$.
		\item \label{enu: acylindricity gives uniform exp - periodicity}
		For every $b \in B_n$ that is not the image of an elliptic element we have $b^n = 1$.
		\item \label{enu: acylindricity gives uniform exp - universal property}
		If every elliptic subgroup of $G$ belongs $\vburn n$ then $B_n$ is isomorphic to $G/G^n$. 
		In particular $B_n$ lies in $\vburn n$.
		\item \label{enu: acylindricity gives uniform exp - infiniteness}
		There exist infinitely many elements in $B_n$ which are not the image of an elliptic element of $G$.
		\item \label{enu: acylindricity gives uniform exp - one-to-one}
		For every $g \in G \setminus\{1\}$, for every $x \in X$, if $\dist {gx}x \leq r$, then the image of $g$ in $B_n$ is not trivial.
	\end{enumerate}
\end{prop}

\begin{rema}
In \cite{Adian:1976vg}, S.I.~Adian introduced a notion of free product in the class $\vburn n$.
Our result can be used to recover most of its properties.
Let $n$ be an odd integer.
Consider $A$ and $B$ two groups of exponent $n$.
It follows from Bass-Serre theory that the (regular) free product $G = A \ast B$ acts $(0,1,0)$-acylindrically on the corresponding Bass-Serre tree $X$, which is a $0$-hyperbolic space.
Provided $n$ is large enough (this value does not depend on $A$ or $B$) one can apply \autoref{res: acylindricity gives uniform exp} to $G$ and $X$.
We get as output a group of exponent $n$, denoted by $\freepb AB$, in which $A$ and $B$ embed.
Moreover if $H$ is a group of $\vburn n$ containing $A$ and $B$ and generated by these two subgroups, then $H$ is a quotient of $\freepb AB$.
In this article we always assumed that the hyperbolicity constant $\delta$ is \emph{positive}.
Nevertheless this is not an issue as we may look at the Bass-Serre tree as a $\delta$-hyperbolic space for arbitrarily small positive $\delta$.
\end{rema}

\medskip
The proof of \autoref{res: acylindricity gives uniform exp} is essentially done by the first author in \cite[Theorem~6.15]{Coulon:2016if}.
However the statement there does not make the dependency between $N_1$ and all the other parameters explicit.
For the agreement of the reader we recall the main steps of the proof, focusing on the ones  which are crucial for the control of $N_1$.
The main ideas are the following.
One defines by induction a sequence of quotients
\begin{displaymath}
	G = G_0 \rightarrow G_1 \rightarrow  \dots \rightarrow G_k \rightarrow G_{k+1}\rightarrow  \dots 
\end{displaymath}
where $G_{k+1}$ is obtained from $G_k$ by adding new relations of the form $h^n$ with $h$ running over all small loxodromic elements of $G_k$.
The quotient $B_n$ is then the direct limit of these groups.
The difficulty is to control the geometry of $G_k$ at each step to make sure the that the sequence of groups does not ultimately collapse.
This is the role of the next statement that will be used as the induction step in our process.

\begin{prop}[Coulon {\cite[Proposition~6.1]{Coulon:2016if}}]
\label{res: SC - induction lemma}
	There exist positive constants $\delta_1$, $A_0$, $r_0$, $\alpha$ such that for every positive integer $\nu_0$ there is an integer $N_0$ with the following properties.
	Let $G$ be a group without involution (i.e.\ element of order 2) acting by isometries on a $\delta_1$-hyperbolic length space $X$.
	We assume that this action is WPD, non-elementary and without parabolic.
	Let $N_1 \geq N_0$  and $n \geq N_1$ be an odd integer.
	We denote by $P$ the set of primitive loxodromic elements $h \in G$ such that $\len h \leq L_S\delta_1$.
	Let $K$ be the (normal) subgroup of $G$ generated by $\{h^n, h \in P\}$ and $\bar G$ the quotient of $G$ by $K$.
	We make the following assumptions.
	\begin{enumerate}
		\item \label{enu: SC - induction lemma - e}
		$e(G,X)$ divides $n$.
		\item \label{enu: SC - induction lemma - nu}
		$\nu(G,X) \leq \nu_0$.
		\item \label{enu: SC - induction lemma - A}
		$A(G,X) \leq \nu_0A_0$.
		\item \label{enu: SC - induction lemma - rinj}
		 $\rinj[X]G \geq r_0/\sqrt{N_1}$.
	\end{enumerate}
	Then there exists a $\delta_1$-hyperbolic length space $\bar X$ on which $\bar G$ acts by isometries.
	This action is WPD, non-elementary and without parabolic.
	The group $\bar G$ has no involution.
	Moreover it satisfies  Assumptions~\ref{enu: SC - induction lemma - e}-\ref{enu: SC - induction lemma - rinj}.
	In addition, the map $G \rightarrow \bar G$ has the following properties.
	\begin{labelledenu}[P]
		\item \label{enu: SC - induction lemma - translation length}
		For every $g \in G$, if $\bar g$ stands for its image in $\bar G$, we have 
		\begin{displaymath}
			\len[stable, espace= \bar X]{\bar g} \leq \frac {\alpha}{\sqrt {N_1}}\len[stable, espace= X]g. 
		\end{displaymath}
		\item  \label{enu: SC - induction lemma - elementary subgroup}
		For every elliptic or parabolic subgroup $E$ of $G$, the map $G \rightarrow \bar G$ induces an isomorphism from $E$ onto its image $\bar E$ which is elementary and non-loxodromic.
		\item \label{enu: SC - induction lemma - elliptic element}
		Let $\bar g$ be an elliptic element of $\bar G$. 
		Then $\bar g^n = 1$ or $\bar g$ is the image of an elliptic element of $G$.
		\item \label{enu: SC - induction lemma - conjugate of elliptic}
		Let $u,u' \in G$ such that $\len u < L_S\delta_1$ and $u'$ is elliptic. 
		If the respective images of $u$ and $u'$ are conjugated in $\bar G$, then so are $u$ and $u'$ in $G$.
	\end{labelledenu}
\end{prop}

\paragraph{Vocabulary.}
Let $n \geq N_1$ be two integers as in the proposition above.
Let $G$ be a group acting by isometries on a metric space $X$.
We say that $(G,X)$ \emph{satisfies the induction hypotheses for exponent $n$} if it satisfies the assumptions of \autoref{res: SC - induction lemma}, including Points~\ref{enu: SC - induction lemma - e}-\ref{enu: SC - induction lemma - rinj}.

\begin{proof}[Proof of \autoref{res: acylindricity gives uniform exp}]
	The parameters $\delta_1$, $A_0$, $r_0$ and $\alpha$ are the universal constants given by \autoref{res: SC - induction lemma}.
	Recall that the action of $G$ on $X$ is non-elementary and $(100 \delta, L, N)$-acylindrical.
	We fix $\nu_0 = N(2 + L/\delta)$.
	The critical exponent $N_0 \in \N$ is the one provided by \autoref{res: SC - induction lemma}.
	We define a rescaling parameter $a > 0$ as follows
	\begin{equation}
	\label{eqn: acylindricity gives uniform exp - a}
	    a = \min\left\{ \frac{\delta_1}\delta, \frac{\nu_0A_0}{10L_S^2N^3(L+5\delta)}, \frac{L_S\delta_1}r \right\}
	\end{equation}
	We now choose $N_1 \geq N_0$ such that 
	\begin{equation}
	\label{eqn: acylindricity gives uniform exp - n_1}
	    \frac{a\delta}N \geq \frac{r_0}{\sqrt{N_1}}
	    \quad \text{and} \quad
	    \frac{\alpha}{\sqrt{N_1}} < 1.
	\end{equation}
	Observe that $a$ and thus $N_1$ only depend on $\delta$, $L$, $N$ and $r$.
	From now on, we fix an odd integer $n \geq N_1$ which is a multiple of $e(G,X)$.

	\paragraph{The base of induction.} 
    We denote by $X_0 = aX$ the space $X$ rescaled by $a$, i.e. for every $x,x' \in X_0$, we have $\dist[X_0] x{x'} = a\dist[X]x{x'}$.
    In addition we let $G_0 = G$.
    One checks that the action of $G_0$ on $X_0$ is $(100a\delta,aL,N)$-acylindrical.
    According to \autoref{res: lower bound inj radius} and \autoref{res: upper bound nu and A} we have 
    \begin{displaymath}
        \inj[X_0]{G_0} \geq \frac{a\delta}N,
        \quad
        \nu(G_0,X_0) \leq N\left(2 + \frac L\delta\right)
        \quad \text{and} \quad
        A(G_0,X_0) \leq 10L_S^2N^3a(L+5\delta)
    \end{displaymath}
    It follows from our choice of $a$ and $N_1$ that $X_0$ is $\delta_1$-hyperbolic, $\nu(G_0,X_0) \leq \nu_0$,  $A(G_0,X_0) \leq \nu_0A_0$ \cite[Lemma~2.45]{Coulon:2014fr} and $\inj[X_0]{G_0} \geq r_0/\sqrt{N_1}$.
	In other words, $(G_0,X_0)$ satisfies the induction hypotheses for exponent $n$.
	
	\paragraph{The inductive step.} 
	Let $k \in \N$.
	We assume that we already constructed the group $G_k$ and the space $X_k$ such that $(G_k,X_k)$ satisfies the induction hypotheses for exponent $n$.
	We denote by $P_k$ the set of primitive loxodromic elements $h \in G_k$ such that $\len[espace = X_k] h \leq L_S\delta_1$.
	Let $K_k$ be the (normal) subgroup of $G_k$ generated by $\{h^n, h \in P_k\}$.
	We write $G_{k+1}$ for the quotient of $G_k$ by $K_k$
	By \autoref{res: SC - induction lemma}, there exists a metric space $X_{k+1}$ such that $(G_{k+1},X_{k+1})$ satisfies the induction hypotheses for exponent $n$.
	Moreover the projection $G_k \twoheadrightarrow G_{k+1}$ fulfills the properties \ref{enu: SC - induction lemma - translation length}-\ref{enu: SC - induction lemma - conjugate of elliptic} of \autoref{res: SC - induction lemma}.

	\paragraph{Direct limit.} 
	The direct limit of the sequence $(G_k)$ is a quotient $B_n$ of $G$.
	We claim that this group satisfies the announced properties.
	Following the same strategy as in \cite[Theorem~6.9]{Coulon:2016if} we prove the following statements:
	every elliptic subgroup of $G$ embeds into $B_n$;
	every element $b \in B_n$ which is not the image of an elliptic element of $G$ satisfies $b^n = 1$;
	there are infinitely many elements in $B_n$ which are not the image of an elliptic element of $G$.
	So we are left to prove Points~\ref{enu: acylindricity gives uniform exp - universal property} and \ref{enu: acylindricity gives uniform exp - one-to-one}.
	
	\medskip
	We start with Point~\ref{enu: acylindricity gives uniform exp - universal property}.
	Assume that every elliptic subgroup of $G$ has exponent $n$.
	It follows from the previous discussion that for every $b \in B_n$, $b^n = 1$.
	Hence the projection $G \twoheadrightarrow B_n$ induces an epimorphism $G/G^n \twoheadrightarrow B_n$.
	On the other hand, all the relations added to define $B_n$ are $n$-th power.
	In other words the kernel $K$ of the projection $G \twoheadrightarrow B_n$ is contained in $G^n$.
	Hence $G/G^n$ and $B_n$ are isomorphic.
	
	\medskip
	We finish with Point~\ref{enu: acylindricity gives uniform exp - one-to-one}.
	Let $g \in G \setminus\{1\}$ and $x \in X$ such that $\dist[X]{gx}x \leq r$.
	It follows from our choice of $a$ that $\dist[X_0]{gx}x \leq L_S \delta_1$.
	In particular $\len[espace = X_0] g \leq L_S\delta$.
	If $g$ is elliptic in $G$, then \ref{enu: SC - induction lemma - conjugate of elliptic} tells us that the image of $g$ in $G_1$ is a non-trivial elliptic element.
	If $g$ is loxodromic in $G$, then by construction the image of $g$ in $G_1$ is elliptic.
	Moreover this image is non-trivial \cite[Theorem~5.2(4)]{Coulon:2016if}.
	A proof by induction using \ref{enu: SC - induction lemma - translation length} and \ref{enu: SC - induction lemma - conjugate of elliptic}  now shows that for every $k \in \N$, the image of $g$ in $G_k$ is non trivial.
	Hence neither is its image in $B_n$.
\end{proof}


\section{From small cancellation to acylindrical action} 
\label{sec: sc to acylindrical}

In this section we study the action of a graphical small cancellation group on its hyperbolic cone-off space, see \autoref{def: cone-off space}. We first provide definitions and notations and then show that, in the cases we consider, the action is acylindrical with universal constants. We then proceed to determine the elliptic elements and the maximal elementary subgroups for this action. Finally, we prove a result that will justify our concise statement of \autoref{res: main theo - regular sc}.

\subsection{Definitions and notation}\label{section:definitions_and_notation}
Recall the definitions of small cancellation conditions given in \autoref{section21}. We shall often use the word \emph{piece} to mean either a word as defined in \autoref{section21}, or as a path (in a labelled graph) whose label is a piece in that sense.

We first discuss graphical small cancellation over the free group. We  give a slight generalization of the definition for graphical small cancellation over free groups \autoref{def: graphical small cancellation - burnside} that will turn out to be suitable for proving acylindricity of the action on the hyperbolic cone-off space. 

\begin{defi}
\label{def: small cancellation - bunrside variable exp}
	Let $p \in \N$ and $\lambda \in (0,1)$.
	Let $\Gamma$ be a graph labelled by a set $S$.
	We say that $\Gamma$ satisfies the $C'(\lambda,p)$-small cancellation assumption if the following holds.
	\begin{enumerate}
		\item \label{enu: small cancellation - bunrside variable exp - sc}
		$\Gamma$ satisfies the $C'(\lambda)$-condition.
		\item \label{enu: small cancellation - bunrside variable exp - powers}
		 Whenever $w$ is a cyclically reduced word such that $w^p$ labels a path in $\Gamma$, then for every $n \in \N$, the word $w^n$ labels a path in $\Gamma$.
	\end{enumerate}
\end{defi}

The following definitions will enable us to do small cancellation theory that produces quotients of a given free product of groups $*_{i\in I} G_i$. We recall the definition of graphical small cancellation over free products of Gruber \cite{Gruber:2015ii}. Given a graph $\Gamma$ labelled by a set $S$, the \emph{reduction} of $\Gamma$ is the quotient of $\Gamma$ by the following equivalence relation on the edges of $\Gamma$: $e\sim e'$ if and only if $\ell(e)=\ell(e')$, and there exists a path from $\iota e$ to $\iota e'$ whose label is trivial in $\F(S)$. Here, and henceforth, $\ell(e)$ denotes the label in $S\sqcup S^{-1}$ of an edge $e$.

\begin{defi}[Completion] 
\label{def: small cancellation condition - free product - completion}
	Let $\Gamma$ be a graph labelled by a set $S:=\sqcup_{i\in I}S_i$, where each $S_i$ is a generating set of a group $G_i$. 
	The \emph{completion} $\overline\Gamma$ of $\Gamma$ is defined as the reduction of the graph obtained from $\Gamma$ by performing the following operations.
	\begin{itemize}
		\item onto each edge labelled by $s\in S_i$ for some $i$, attach a copy of $\Cay(G_i,S_i)$ along an edge labelled by $s$;
		\item if, for some $i$, no element of $S_i$ occurs as a label of $\Gamma$, then add a copy of $\Cay(G_i,S_i)$ (as its own connected component).
	\end{itemize}
	A word in the free monoid on $S\sqcup S^{-1}$ is \emph{locally geodesic} if it labels a geodesic in $\Cay(*_{i\in I}G_i,S)$, and a path in $\overline\Gamma$ (or another $S$-labelled graph) is locally geodesic if its label is locally geodesic. 
\end{defi}	
If $\{C_j\}_{j\in J}$ is a collection of copies of $\Cay(G_i,S_i)$ for some fixed $i$ such that $\cup_{j\in J}C_j$ is connected, then the reduction step identifies all the $C_j$ to a single copy of $\Cay(G_i,S_i)$. We deduce $\overline{\overline\Gamma}=\overline\Gamma$.
%

\begin{defi}[Small cancellation condition]
\label{def: small cancellation condition - free product}
		Let $\lambda\in (0,1)$. Let $\Gamma$ be a graph labelled by a set $S:=\sqcup_{i\in I}S_i$, where each $S_i$ is a generating set of a group $G_i$. We say $\Gamma$  satisfies the $C'_*(\lambda)$-condition if
	\begin{itemize}
		\item $\Gamma=\overline\Gamma$;
		\item every $\Cay(G_i,S_i)$ is an embedded subgraph of $\Gamma$;
		\item for every locally geodesic piece $w$ that is a subword of the label of a simple closed path $\gamma$ in $\Gamma$ such that the label of $\gamma$ is non-trivial in $*_{i\in I}G_i$, we have $|w|<\lambda|\gamma|$.	\end{itemize}
	Here, as usual, $|\cdot|$ denotes the (edge-)length of a path, respectively the length of a word (i.e.\ its number of letters). We call the embedded copies of $\Cay(G_i,S_i)$ \emph{attached} Cayley graphs.
	When we say an $S$-labelled graph satisfies the $C'_*(\lambda)$-condition, we shall assume that we have $S=\sqcup_{i\in I}S_i$ for given generating sets $S_i$ of given groups $G_i$. 
	If $\Gamma=\overline\Gamma$, the group $G(\Gamma)$, using our previous definition in \autoref{section22}, coincides with the quotient of $*_{i\in I}G_i$ by all the words read on closed paths in $\Gamma$.
\end{defi}

The assumption $\Gamma=\overline\Gamma$ is a mere technicality to allow for efficient notation. We recall from \cite{Gruber:2015ii} that if $\Gamma$ satisfies the $C'_*(1/6)$-condition, then each generating factor $G_i$ is a subgroup of $G(\Gamma)$, and each component of $\Gamma$ injects into $\Cay(G(\Gamma),S)$.

\medskip

The case that $\Gamma$ can be realized as the completion of a disjoint union of cycle graphs and that $S_i=G_i$ recovers the classical $C'_*(\lambda)$-condition, whose power-free simplification is given in \autoref{def: power-free small cancellation condition - free product}, see \autoref{rema:cylinder_free_classical}.

In order to prove our acylindricity statement also in the free product case, we require the following property for $\Gamma$. It will be used in \autoref{lem:tripleautomorphism}.

\begin{defi}[Cylinder-free]\label{defi:cylinder-free} Let $\Gamma$ satisfy the $C_*'(\lambda)$-condition for some $\lambda$. We say $\Gamma$ is \emph{cylinder-free} if for every connected component $\Gamma_0$ and every two disjoint attached Cayley graphs $C_1$ and $C_2$ of $\Gamma_0$, any label-preserving automorphism $\phi$ of $\Gamma_0$ with $\phi(C_1)=C_1$ and $\phi(C_2)=C_2$ is the identity.
\end{defi}

\begin{rema}\label{rema:cylinder_free_classical}
Let $\Theta$ be a cycle graph labelled by a set $S:=\sqcup_{i\in I}S_i$, where each $S_i$ is a generating set of a group $G_i$. Suppose there is a simple closed path $\gamma$ in $\Theta$ with label $w=w_1w_2\dots w_k$ for $k>1$, where each $w_j$ is in the free monoid on $S_{i_j}\sqcup S_{i_j}^{-1}$ for some $i_j\in I$, such that each $w_j$ is non trivial in $G_{i_j}$, and $i_j\neq i_{j+1}$ for each $j$, and $i_k\neq i_1$. In other words, if $g_j$ is the element of $G_{i_j}$ represented by $w_j$, then $g_1g_2\dots g_k\in *_{i\in I}G_i$ is in normal form and cyclically reduced in the sense of \autoref{section22}. We can realize $\overline\Theta$ as follows: onto each subpath $\gamma_j$ of $\gamma$ with label $w_j$, attach a copy $C_j$ of $\Cay(G_{i_j},S_{i_j})$ along a path in $\Cay(G_{i_j},S_{i_j})$ with label $w_j$. Since $\gamma_j$ is a simple path, $C_j$ is an embedded copy of $\Cay(G_{i_j},S_{i_j})$. Since $w_j$ is non trivial in $G_{i_j}$, the image $\overline\gamma_j$ of $\gamma_j$ in $C_j$ has distinct endpoints. Therefore, for no $j$ two attached copies of $\Cay(G_{i_j},S_{i_j})$ intersect, whence the graph we obtain is reduced. Thus, up to possibly adding as separate connected components copies of $\Cay(G_i,S_i)$ for those $S_i$ which do not appear in $w$, we have obtained $\overline\Theta$. Observe that any attached Cayley graph $C_j$ shares exactly two vertices with other attached Cayley graphs by construction.

%
%
%
%
%
%

More generally, suppose $\Gamma$ is the completion of a disjoint union of cycle graphs with cyclically reduced labels, as discussed above, and let $C$ be an attached Cayley graph in a component $\Gamma_0$ of $\Gamma$. Then there are at most two vertices $v_1$ and $v_2$ that $C$ shares with other attached Cayley graphs. A label-preserving automorphism $\phi$ of $\Gamma_0$ that preserves $C$ also preserves the set $\{v_1,v_2\}$. As $\phi$ is uniquely determined by the image of any one vertex, this means that there are at most two options for $\phi$: being the identity, and permuting $v_1$ and $v_2$. In particular, the order of $\phi$ divides 2.

We also observe: if an automorphism $\phi$ of $\Gamma_0$ preserves an attached $C=\Cay(G_i,S_i)$, then the action of $\phi$ on $C$ corresponds to left-multiplication by the element of $G_i$ represented by the label of any path in $C$ from $v$ to $\phi(v)$ for any vertex $v$ in $C$. Thus, the order of $\phi$ equals the order of an element of $G_i$. In particular, if $\Gamma$ is the completion of a disjoint union of cycle graphs with cyclically reduced labels, as discussed above, and no generating factor has even torsion, then $\Gamma$ is cylinder-free.
\end{rema}

\begin{exam} A classical $C_*'(1/6)$-presentation that is not cylinder-free is given, for example, by the quotient of $\Z*\Z*\Z*\Z/2\Z$, where each copy $G_i$ of $\Z$ is generated by $t_i$, $i=1,2,3$, and $\Z/2\Z$ is generated by $s$, by the relation $t_1t_2t_3st_3^{-1}t_2^{-1}t_1^{-1}s$. 
The corresponding graph $\Gamma$ has one non-trivial automorphism $\phi$ that leaves invariant the two attached Cayley graphs corresponding $\Z/2\Z$ and, for each $i=1,2,3$, exchanges the two attached Cayley graphs corresponding to $G_i$.

Up to taking inverses and cyclic conjugates, all locally geodesic simple closed paths are labelled by $t_1t_2t_3s^{\epsilon_1}t_3^{-1}t_2^{-1}t_1^{-1}s^{\epsilon_2}$, where $\epsilon_1,\epsilon_2\in\{1,-1\}$, and hence have length 8. All pieces have length at most 1 (observe, e.g.\, that by the above considerations, the paths labelled by $t_1t_2t_3$ are \emph{not} pieces). Thus, the $C_*'(1/8+\epsilon)$-condition is satisfied for any $\epsilon>0$.

\end{exam}

\begin{defi}
\label{def: small cancellation condition -free product - burnside}
	Let $\lambda \in (0,1)$, let $p \in \N$.
	Let $\Gamma$ be a graph labelled by a set $S:=\sqcup_{i\in I}S_i$, where each $S_i$ is a generating set of a group $G_i$. 
	We say $\Gamma$ satisfies the $C_*'(\lambda, p)$-condition if 
	\begin{enumerate}
		\item it satisfies the $C_*'(\lambda)$-condition;
		\item for every cyclically reduced word $w$ over the alphabet $S$, if $w^p$ is the label of a path in $\Gamma$, the for every $n \in \N$, there is a path in $\Gamma$ labelled by $w^n$;
		\item it is cylinder-free.
	\end{enumerate}
\end{defi}

We have argued in Remark~\ref{rema:cylinder_free_classical} that a presentation satisfying the power-free $C_*(\lambda,p)$-condition of Definition~\ref{def: power-free small cancellation condition - free product} and for which no generating factor contains elements of order two can be regarded as satisfying the (graphical) $C_*(\lambda,p)$-condition we just defined.

\medskip

\paragraph{The hyperbolic cone-off space $\dot X$.} 
Let $\Gamma$ be a graph labelled be a set $S$. We associate to $\Gamma$ its cone-off space defined by Gruber and Sisto \cite{Gruber:2014wo}.
\begin{defi}[The cone-off space]
\label{def: cone-off space}
	Let $\Gamma$ be a graph labelled by $S$.
	The \emph{cone-off space}, denoted by $\dot X(\Gamma)$ (or simply $\dot X$) is the Cayley graph of $G(\Gamma)$ with respect to $S \cup W$, where $W$ stands for the set of all elements of $G(\Gamma)$ represented by the label of a path in $\Gamma$.
\end{defi}
Observe that in the free product case, if $\Gamma$ is its own completion (as is the case for a $C_*'(\lambda)$-graph), then the image in $G(\Gamma)$ of each one of the generating factors is contained in $W$. 
We record  an immediate consequence of \cite[Remark~4.11]{Gruber:2014wo}:

\begin{theo}[Uniform hyperbolicity of $\dot X$]\label{thm:hyperbolicity} Let $\Gamma$ be a $C'(1/6)$-labelled graph or a $C'_*(1/6)$-labelled graph. Then the vertex set of any geodesic triangle in $\dot X$ is 5-slim. 
\end{theo}

In particular, the vertex set of $\dot X$ is 40-hyperbolic in the sense of \autoref{section:hyperbolic} by \cite[Chapitre 1, Proposition 3.6]{CooDelPap90}.

\medskip

In the following, we will show results concerning quotients of free groups and quotients of free products, which are, in principle, separate cases, whence we have to carry out two proofs. However, we shall see that the geometric arguments involved in both proofs are very similar and, in many cases, exactly the same. We now introduce standing assumptions and notation that will let us efficiently handle the two cases in parallel.

\medskip

\noindent {\bf Notation.} We denote by $\Gamma$ a labelled graph and by $S$ the set of labels. Recall our convention that if we say $\Gamma$ satisfies the $C'_*(\lambda)$-condition, then we assume  $S=\sqcup_{i\in I} S_i$, where the $S_i$ are generating sets of groups $G_i$. We denote $X:=\Cay(G(\Gamma),S)$ and $\dot X$ the cone-off space from \autoref{def: cone-off space}. We will consider $X$ as a subgraph of $\dot X$.

\medskip
We from now on assume that, in the free case, the $\scc$-condition is satisfied and, in the free product case, the $\scsc$-condition is satisfied. Moreover, for simplicity, we assume that no component of $\Gamma$ is a single vertex, and that every generator occurs on some edge of $\Gamma$. (The latter is no restriction: in the free product case, it follows from the definition of $\Gamma$. In the free case we may just add, for any generator that does not occur already, a new component to $\Gamma$ that is merely an edge labelled by the generator. This does not change the metric on the vertex set of $\dot X$.) A \emph{relator} is the image of a component of $\Gamma$ in $X$ under a label-preserving graph homomorphism. Note that such a relator is in fact an (isometrically) \emph{embedded} image of a component \cite{Oll06,Gruber:2015fu,Gruber:2014wo}.

We denote by $F$ the free group on $S$ in the free case, respectively the free product of the generating factors $G_i$ in the free product case.

\medskip

\noindent {\bf $F$-reduced paths.}  In the free product case, we define the following terminology, which should be thought of as a way of dealing with homotopy classes of paths not necessarily in a tree (which corresponds to the free group case) but in a tree of Cayley complexes (which corresponds to the free product case).

We say a path in $X$ is $F$-reduced if its label has the form $w_1w_2\dots w_k$, where each $w_i$ is a word contained in a single generating factor that does not represent the identity, and any two consecutive $w_i$ come from distinct generating factors. (In other words, if $w_i$ represents an element $g_i$ of a generating factor, then $g_1g_2\dots g_k$ is in normal form in the sense of \autoref{section22}.) A closed path is cyclically $F$-reduced if every one of its cyclic shifts is $F$-reduced. (In other words, $g_1g_2\dots g_k$ is weakly cyclically reduced.) We say two paths are $F$-equivalent if they have the same starting vertex and have the same label as elements of $F$.  
A path is $F$-homotopically trivial if its label is trivial in $F$. 
An $F$-tree is a subgraph of $X$ where every two vertices are connected by a unique $F$-equivalence class of $F$-reduced paths or, equivalently, a connected subgraph where any closed path is $F$-homotopically trivial.
Given an $F$-reduced path $p$, an $F$-subpath $q$ is an $F$-reduced path $F$-equivalent to a subpath of $p$ such that the label $\ell(q)$ of $q$ is a  subword of $\ell(p)$ in the free product sense. This means: if the element of $F$ represented by $\ell(p)$ is written as $g_1g_2\dots g_k$ in normal form, then there exist $1\leq i\leq j\leq k$ such that $p$ may be written as $p_1p_2p_3$ with $\ell(p_1)$ representing $g_1g_2\dots g_{i-1}$, $\ell(p_2)$ representing $g_i g_{i+1}\dots g_j$, and $q$ being $F$-equivalent to $p_2$.

\medskip

We have already explained the notion of locally geodesic paths in the free product setting. In the free group case, we shall take locally geodesic to mean reduced (i.e. without back-tracking).

\medskip

In the free group case, all terms defined above are defined in the free group case by simply omitting the prefix ``$F-$'' (i.e.\ an $F$-reduced path is simply a reduced path, a an $F$-tree is simply a tree, \dots). Using the same words for both cases will allow us to streamline statements.

\subsection{Acylindricity of the action on $\dot X$}

We set out to prove the following result. Notice that the constants we produce are universal, i.e.\ independent of the specific graph $\Gamma$ under consideration.

\begin{theo}[Acylindricity theorem]\label{thm:acylindricity} For all $p\in\N$ and $\epsilon\geq 0$ there exist $L>0$ and $N\in\N$ with the following property.
Let $\Gamma$ be a labelled graph satisfying the graphical $C'(1/6,p)$-condition or the graphical $C_*'(1/6,p)$ condition. Then  the action of $G(\Gamma)$ on the cone-off space $\dot X$ is $\bigl(\epsilon, L,N\bigr)$-acylindrical.
\end{theo}

The actual constants we obtain are $L=18\epsilon+25$ and $N=(9\epsilon+4)^3(8p+100)$, i.e.\ $L$ does not depend on $p$, while $N$ does.

\begin{rema} 
When considering the action of $G(\Gamma)$ on $\dot X$, one may instead consider the action of $G(\Gamma)$ on the graph $R$ with vertex set $\{\text{relators in }X\}$ and where any two vertices are connected by an edge if their corresponding relators in $X$ intersect.  The graph $R$ comes with the $G(\Gamma)$-action induced by the action of $G(\Gamma)$ on $X$, and $R$ is $G(\Gamma)$-equivariantly $(1,1)$-quasi-isometric to $\dot X$. (Recall here that we assume that every generator occurs on $\Gamma$.) The two main propositions in our proof of \autoref{thm:acylindricity}, namely Propositions~\ref{prop:parallel} and \ref{prop:action}, are phrased purely in terms of $R$, as are many of our intermediate results. Nonetheless, we will often need the underlying space $X$ in our arguments. 
\end{rema}

\subsubsection{Convexity of geodesics in $\dot X$} 
In this subsection, we link the metric properties of geodesics in $\dot X$ to metric properties of $X$, strengthening \cite[Proposition~3.6]{Gruber:2014wo}. 
We remark here that Lemmas~\ref{lem:chords} and \ref{lem:convex} are our only applications of van Kampen diagrams in the proof of \autoref{thm:acylindricity}. 
While more extensive usage of diagrams (following techniques of \cite{Gruber:2014wo, Arzhantseva:2016uk}) would allow us to provide better acylindricity constants, limiting their usage enables us to more clearly present our proof, in particular in view of the dual approach to both free group and free product cases.

\medskip

We recall a tool from graphical small cancellation theory and basic facts about it, see \cite{Gruber:2015fu,Gruber:2015ii} for details.

\medskip

\noindent {\bf $\Gamma$-reduced diagrams.} A \emph{diagram} over a presentation $\langle S\mid R\rangle$ is a finite connected $S$-labelled graph $D$ with a fixed embedding in $\R^2$, such that each bounded region (\emph{face}) has a boundary word in $R$. Given a closed path $\gamma$ in $\Cay(G,S)$, where $G$ is the group defined by $\langle S\mid R\rangle$, a diagram for $\gamma$ is a diagram that admits a label-preserving map $\partial D\to\gamma$, where $\partial D$ is the boundary of the unbounded component defined by $D$ inside $\R^2$. It is a classical fact that for every closed path in $\Cay(G,S)$ there exists a diagram. To simplify notation, we will often not distinguish in notation between subpaths of $\gamma$ and their preimages in $\partial D$ where this does not cause ambiguity.  A disk diagram is a diagram without cut-vertices.

Suppose we have an $S$-labelled graph $\Gamma$ and the presentation $\langle S\mid \text{labels of simple}$ $\text{closed paths }$ $\text{in }\Gamma\rangle$. A diagram $D$ over $\Gamma$ is a diagram over this presentation. If $\Pi$ is a face of $D$, then $\partial \Pi$ admits a map $\partial\Pi\to\Gamma$ (possibly more than one), which we call lift. We say $D$ is $\Gamma$-reduced if for any two faces $\Pi$ and $\Pi'$ and any path $a$ in $\Pi\cap\Pi'$, no two lifts $\partial\Pi\to\Gamma$ and $\partial\Pi'\to\Gamma$ restrict to the same map on $a$ and any path $a$ in $\Pi\cap\Pi'$ is locally geodesic. In the free product case, we also require that any face whose boundary word is trivial in $F$ actually has a boundary word that is contained in a single generating factor. 

If $\Gamma$ satisfies the $\scc$-condition or the $\scsc$-condition, then any closed path $\gamma$ in $\Cay(G(\Gamma),S)$ admits a $\Gamma$-reduced diagram $D$ over $\Gamma$, see \cite[Lemma~2.13]{Gruber:2015fu} or \cite[Theorem~1.23]{Gruber:2015vc} for the free group case and \cite[Lemma~3.8]{Gruber:2015ii} or \cite[Theorem~1.35]{Gruber:2015vc} for the free product case.

In a $\Gamma$-reduced diagram, any face is simply connected, and for any two faces $\Pi$ and $\Pi'$ and any path $\alpha$ in $\Pi\cap\Pi'$ we have $|\alpha|<\min\{|\partial\Pi|,|\partial\Pi'|\}/6$, because $\alpha$ is a locally geodesic piece. Finally, we have that every face that intersects at least one other face in at least one edge has a label that is non-trivial in $F$. 

An arc in a diagram is an embedded line graph whose endpoints have degrees different from 2 and all whose other vertices have degree 2. Given a face $\Pi$, $d(\Pi)$ is the number of arcs in $\partial D$, $i(\Pi)$ is the number of arcs in $\partial \Pi$ that $\Pi$ shares with other faces (interior arcs), $e(\Pi)$ is the number of arcs it shares with $\partial D$ (exterior arcs). A $(3,7)$-diagram is a diagram where $e(\Pi)=0$ implies $i(\Pi)\geq 7$. Note that a $\Gamma$-reduced diagram over a $\scc$-labelled graph or a $\scsc$-labelled graph $\Gamma$ is a $(3,7)$-diagram. 

\medskip

The following is a classical fact from small cancellation theory, see e.g.\ \cite[Chapter V]{LynSch77}.

\begin{lemm}[Greendlinger's lemma] Let $D$ be a $(3,7)$-disk diagram that is not a single face. Then $D$ contains two faces $\Pi_1$ and $\Pi_2$ with $e(\Pi_1)=1=e(\Pi_2)$ and $i(\Pi_1)\leq 3$ and $i(\Pi_2)\leq 3$.
\end{lemm}

\begin{lemm}\label{lem:chords} Let $\Gamma_1,\Gamma_2,\dots,\Gamma_k$ be relators such that $\Gamma_{i}\cap\Gamma_{i+1}\neq\emptyset$ (indices mod $k$). If $k=3$, then $\Gamma_1\cap\Gamma_2\cap\Gamma_3\neq\emptyset$. If $k=4$, then $\Gamma_1\cap\Gamma_3\neq\emptyset$ or $\Gamma_2\cap\Gamma_4\neq\emptyset$.
\end{lemm}

\begin{proof} Assume the claim is false for $3\leq k\leq 4$ and corresponding $\Gamma_i$. Let $\gamma=\gamma_1\gamma_2\dots\gamma_k$ be a closed path, where each $\gamma_i$ is in $\Gamma_i$. Let $D$ be a $\Gamma$-reduced diagram for $\gamma$. Choose $\gamma$ such that, among all possible choices, the number of edges of $D$ is minimal. Then $D$ is a disk diagram, and the label of $\partial D$ is non-trivial in $F$.

Consider a path $\pi$ in the intersection of a face $\Pi$ with some $\gamma_i$. Then there are lifts $\pi\to\partial\Pi\to\Gamma$ and $\pi\to\gamma_i\to\Gamma_i$. By our assumption on edge-minimality, these two lifts may never coincide: otherwise, we could replace the copy of $\pi$ that is a subpath of $\gamma_i$ by a copy of the complement of $\pi$ in $\Pi$. Thus $\gamma_i$ is a piece (with respect to $\Gamma$) and, by minimality, it is locally geodesic. Hence, by the small cancellation condition, $D$ is not a single face.

Let $\Pi$ be a face with $i(\Pi)\leq 3$ and $e(\Pi)=1$. Then the exterior arc of $\Pi$ is not a concatenation of at most 3 pieces, since the label of $\partial \Pi$ is non-trivial in $F$. Hence, if $k=3$, we have a contradiction. If $k=4$, this arc intersects all the $\gamma_i$ in edges. This can be true for at most one face $\Pi$, which contradicts Greendlinger's lemma.
\end{proof}

\begin{defi}[{\cite[Definition~2.11]{Gruber:2014wo}}]\label{defi:n-gons}
 A \emph{$(3,7)$-bigon} is a $(3,7)$-diagram with a decomposition of $\partial D$ into two reduced subpaths $\partial D=\gamma_1\gamma_2$ with the following property: Every face $\Pi$ of $D$ with $e(\Pi)=1$ for which the exterior arc in $\partial\Pi$ is contained in one of the $\gamma_i$ satisfies $i(\Pi)\geqslant 4$. 
A face $\Pi$ for which there exists an exterior arc in $\partial\Pi$ that is not contained in any $\gamma_i$ is called \emph{distinguished}.
\end{defi}

\begin{lemm}[Strebel's bigons {\cite[Theorem~35]{Strebel:1990wa}}]\label{lem:bigons}
 Let $D$ be a $(3,7)$-bigon. Then any one of its disk components is either a single face, or it has the shape $I_1$ depicted in Figure~\ref{figure:bigons}. This means $D$ has exactly two distinguished faces that have interior degree 1 and exterior degree 1. Moreover, any non-distinguished face has interior degree 2 and exterior degree 2 and intersects both sides of $D$. 
\end{lemm}

\begin{figure}
\begin{center}
\begin{tikzpicture}[line cap=round,line join=round,x=1.0cm,y=1.0cm,line width=1pt]
\draw [shift={(2,-3)}] plot[domain=0.93:2.21,variable=\t]({1*5*cos(\t r)+0*5*sin(\t r)},{0*5*cos(\t r)+1*5*sin(\t r)});
\draw [shift={(2,5)}] plot[domain=4.07:5.36,variable=\t]({1*5*cos(\t r)+0*5*sin(\t r)},{0*5*cos(\t r)+1*5*sin(\t r)});
\draw (0,1.58)-- (0,0.42);
\draw (4,1.58)-- (4,0.42);
\draw (1,1.9)-- (1,0.1);
\draw (3,1.9)-- (3,0.1);
\draw [dotted] (1.5,1)-- (2.5,1);
\end{tikzpicture}
\end{center}
\caption{A diagram $D$ of \emph{shape $I_1$}. 
All faces except the two distinguished ones are optional, i.e. $D$ may have as few as 2 faces. 
}
\label{figure:bigons}
\end{figure}
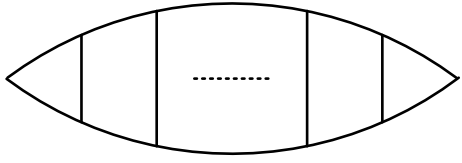

\begin{defi}\label{def:geodesic} We say a sequence of relators $\Gamma_1,\Gamma_2,\dots,\Gamma_n$ is geodesic if $\Gamma_i\cap\Gamma_{i+1}\neq\emptyset$, and if there exists no sequence of relators $\Gamma_1=\Theta_1,\Theta_2,\dots,\Theta_k=\Gamma_n$ with $\Theta_i\cap\Theta_{i+1}\neq\emptyset$ and $k<n$.

\end{defi}

Note that since we assume that every generator occurs on $\Gamma$, we have that for any two relators, there exists a geodesic sequence containing them. Moreover, given vertices $x,y\in X$ with $d_{\dot X}(x,y)=n$, there exists a geodesic sequence of relators $\Gamma_1,\Gamma_2,\dots,\Gamma_n$ with $x\in \Gamma_1$ and $y\in\Gamma_n$ by the very definition of $\dot X$.

\begin{lemm}\label{lem:convex} If $\Gamma_1,\Gamma_2,\dots,\Gamma_n$ is a geodesic sequence, then $\Gamma_1\cup\Gamma_2\cup\dots\cup\Gamma_n$ is convex in $X$.
\end{lemm}

\begin{proof} Let $x$ be a vertex in $\Gamma_i$ and $y$ a vertex in $\Gamma_j$. If $i=j$, then any geodesic from $x$ to $y$ stays in $\Gamma_i$, because $\Gamma_i$ is convex by \cite[Lemma~2.15]{Gruber:2014wo}. Thus we may assume $i<j$ and, if (given $x$ and $y$) the choices of $i$ and $j$ are not unique, choose them such that $|i-j|$ is minimal. Then there exist non-trivial paths $\sigma_t$ in $\Gamma_t$ for $i\leq t\leq j$ such that $\sigma:=\sigma_i\sigma_{i+1}\dots\sigma_j$ is a path from $x$ to $y$. Let $\gamma$ be a geodesic in $X$ from $x$ to $y$. Let $D$ be a $\Gamma$-reduced diagram for $\gamma\sigma^{-1}$ and, given $x$ and $y$, among the possible choices for the $\sigma_t$, choose them such that the number of edges of $D$ is minimal. Observe that this implies that $\sigma$ is reduced.

Consider a non-trivial path $\pi$ in the intersection of a face $\Pi$ with some $\sigma_t$. Then, as in the proof of \autoref{lem:chords}, $\pi$ is a locally geodesic piece and the label of $\Pi$ is non-trivial in $F$. Since $\Gamma_1,\Gamma_2,\dots,\Gamma_n$ is a geodesic sequence, $\Pi$ can intersect at most 3 consecutive $\sigma_t$. (Recall that the 1-skeleton of any face maps to a subgraph of a relator in $X$.) Therefore, if $\Pi$ has $e(\Pi)=1$ and its exterterior arc contained in $\sigma$, then $i(\Pi)\geq 4$. The same conclusion holds if $\Pi$ has $e(\Pi)=1$ and its exterior arc contained in $\gamma$, because $\gamma$ is a geodesic in $X$. Therefore, $D$ is a $(3,7)$-bigon.

Let $\Delta$ be a disk-component of $D$. If it were a single face $\Pi$, then, as observed above, $\Pi$ would intersect $\sigma$ in at most 3 locally geodesic pieces and the geodesic $\gamma$ in a path of length at most $|\partial\Pi|/2$. This contradicts the fact that any locally geodesic piece has length less than $|\partial\Pi|/6$. Therefore, by \autoref{lem:bigons}, it has shape $I_1$ as shown in Figure~\ref{figure:bigons}. Thus any face $\Pi$ has interior degree at most 2, whence its intersection with $\sigma$ has length greater than $|\partial\Pi|/6$ -- therefore, this intersection is not a piece and not contained in a single $\sigma_t$. This shows that in $\Delta$ there may be at most one face intersecting 3 consecutive $\sigma_t$ in edges -- otherwise we would have a path $\sigma_t\sigma_{t+1}\dots\sigma_{t+k}$ contained in $k$ faces, where $i<t$ and $t+k<j$. This would contradict the fact that $\Gamma_1,\Gamma_2,\dots,\Gamma_n$ is a geodesic sequence. However, $\Delta$ has two faces with interior degree 1, neither of which can intersect $\sigma$ in at most 2 locally geodesic pieces. This is a contradiction to the existence of a disk component, whence $\gamma=\sigma$.
\end{proof}

\subsubsection{Parallels in $\dot X$}

We show that geodesic quadrangles in $\dot X$ are uniformly slim in the following sense:

\begin{defi}\label{def:parallel} 
We say two geodesic sequences of relators $\Gamma_1,\Gamma_2,\dots,\Gamma_n$ and $\Theta_1,\Theta_2,\dots,\Theta_{n'}$ are \emph{parallel} if $n=n'$, and $\Gamma_i\cap\Theta_i\neq\emptyset$ for all $i$. 
We say they are \emph{properly} parallel if $\Gamma_i\neq\Theta_i$ for all $i$.
\end{defi}

\begin{prop}\label{prop:parallel} Let $\epsilon\geq 0$, and let $\Gamma_1,\Gamma_2,\dots,\Gamma_n$ and $\Theta_1,\Theta_2,\dots,\Theta_{n'}$ be geodesic sequences such that $d_{\dot X}(\Gamma_1,\Theta_1)\leq \epsilon$ and $d_{\dot X}(\Gamma_n,\Theta_{n'})\leq \epsilon$. Then there exist $k,k',l,l'$ with $\max\{k,k',l,l'\}\leq 9\epsilon +3$ and $k+l=k'+l'$ such that $\Gamma_k,\Gamma_{k+1},\dots,\Gamma_{n-l+1}$ and $\Theta_{k'},\Theta_{k'+1},\dots,\Theta_{n'-l'+1}$ are parallel.
\end{prop}

\begin{lemm}\label{lem:intersection_tree}
Let $\Gamma_1,\dots,\Gamma_n$ and $\Theta_1,\dots,\Theta_{n'}$ be geodesic sequences such that $\Gamma_i\neq \Theta_j$ for every $i,j$. Then $T:=(\Gamma_1\cup\Gamma_2\cup\dots\cup\Gamma_n)\cap(\Theta_1\cup\Theta_2\cup\dots\cup\Theta_{n'})$ is an $F$-tree.
\end{lemm}

\begin{proof} Since $T$ the intersection of two convex subgraphs (see \autoref{lem:convex}), it is connected. 
We proceed by contradiction: assume there is a closed path in $T$ that is not $F$-trivial. 
Given a path $\gamma$ in $T$, we define its weight to be  $\omega(\gamma):=(k,e)$, where $k$ is minimal such that $\gamma$ is contained in $\Gamma_1\cup\Gamma_2\cup\dots\cup\Gamma_k$ and $e$ is the number of edges of $\gamma$ that are in $\Gamma_k$ but not in $\Gamma_{k-1}$. Consider $\N\times\N$ with the lexicographic order. Let $\gamma$ be a closed path in $T$, not $F$-trivial, such that $\omega(\gamma)$ is minimal among all such paths. We choose $\gamma$ to be cyclically $F$-reduced.

We claim: $\gamma$ is contained in a single $\Gamma_{i_0}$. Suppose it is not. Let $\omega(\gamma)=(k,e)$. Then (some cyclic shift of) $\gamma$ contains a subpath $\pi$ with:  $\iota \pi$ and $\tau \pi$ are contained in $\Gamma_{k-1}\cap\Gamma_{k}$, all edges of $\pi$ are in $\Gamma_{k}$ but not in $\Gamma_{k-1}$, $\pi$ has at least one edge, and $\gamma$ does not contain an edge of $\Gamma_{k+1}$. We have that $T$ is convex, and so are $\Gamma_{k-1}$ and $\Gamma_{k}$.  Therefore, $T\cap\Gamma_{k-1}\cap\Gamma_{k}$ is convex and hence connected, and there exists an $F$-reduced path $\pi'$ in $T\cap\Gamma_{k-1}\cap\Gamma_{k}$ with the same endpoints as $\pi$.  Now we may replace the subpath $\pi$ of (the cyclic shift of) $\gamma$ by $\pi'$ to obtain a closed path $\gamma'$ with either $\omega(\gamma')=(k-1,e')$ for some $e'$ or $\omega(\gamma')=(k,e')$ for some $e'<e$. In any case, $\omega(\gamma')<\omega(\gamma)$ whence, by our minimality assumption, $\gamma'$ is $F$-trivial. We may write a cyclic shift of $\gamma'$ as $\pi'\eta$, where both $\pi'$ and $\eta$ are $F$-reduced. This implies that ${\pi'}^{-1}$ and $\eta$ are in the same $F$-equivalence class of $F$-reduced paths. Since $\pi'$ is contained in $\Gamma_{k-1}\cap\Gamma_{k}$, so is $\eta$. We deduce that the original path $\gamma$ was contained in $\Gamma_{k}$, contradicting our assumption.

Let $T_{i_0}:=\Gamma_{i_0}\cap T$. Analogously to above, we define a weight function $\omega'$ with respect to the geodesic sequence $\Theta_1,\dots,\Theta_{n'}$ and show that a closed path $\theta$ in $T_{i_0}$ that is not $F$-trivial and that minimizes $\omega'$ is contained in a single $\Theta_{j_0}$. Thus, such a $\theta$ is contained in $\Gamma_{i_0}\cap\Theta_{j_0}$. But any path in this intersection is a piece, whence $\theta$ cannot exist, and $T$ is an $F$-tree.
\end{proof}

\begin{lemm}\label{lem:uniqueintersectionpath} Let $n\geq 3$, and let $\Gamma_1,\dots,\Gamma_n$ and $\Gamma_1,\Theta_2,\Theta_3,\dots,\Theta_{n-1},\Gamma_{n}$ be geodesic sequences such that $\Gamma_i\neq \Theta_j$ for every $2\leq i,j\leq {n-1}$. Then there is an up to $F$-equivalence unique $F$-reduced path in $(\Gamma_{2}\cup\Gamma_3\cup\dots\cup\Gamma_{n-1})\cap(\Theta_2\cup\Theta_3\cup\dots\cup\Theta_{n-1})$ that connects a vertex in $\Gamma_1$ to a vertex in $\Gamma_n$ and that does not contain any edge in $\Gamma_1$ or in $\Gamma_n$.
\end{lemm}

\begin{proof} As both $\Gamma_1\cup\Gamma_2\cup\dots\cup\Gamma_{n-1}\cup\Gamma_n$ and $\Gamma_1\cup\Theta_2\cup\dots\cup\Theta_{n-1}\cup\Gamma_n$ are convex by \autoref{lem:convex}, their intersection is convex and, in particular, connected. Hence, the intersection contains a path $\alpha$ connecting a vertex of $\Gamma_1$ to a vertex of $\Gamma_n$. Observe that the first edge of $\alpha$ outside $\Gamma_1$ must be contained in $\Gamma_2\cap\Theta_2$ and, likewise, the last edge outside $\Gamma_n$ must be contained in $\Gamma_{n-1}\cap\Theta_{n-1}$. Denote $T:=(\Gamma_2\cup\Gamma_{3}\cup\dots\cup\Gamma_{n-1})\cap(\Theta_2\cup\Theta_3\cup\dots\cup\Theta_{n-1})$. By Lemmas~\ref{lem:convex} and \ref{lem:intersection_tree}, $T$ is a convex $F$-tree, and it intersects both convex graphs $\Gamma_1$ and $\Gamma_n$. Since $T\cap\Gamma_1$ and $T\cap\Gamma_2$ are sub-$F$-trees of $T$, there is an up to $F$-equivalence unique $F$-reduced path in $T$ connecting them as in the claim.
\end{proof}

\begin{lemm}\label{lem:2pieceunique}
Let $\Gamma_1,\Gamma_2,\Gamma_3$ be a geodesic sequence of relators. Up to $F$-equivalence, there exists at most one $F$-reduced path that is a concatenation of at most two pieces in $\Gamma_2$ that intersects $\Gamma_1$ exactly in a vertex and $\Gamma_3$ exactly in a vertex.
\end{lemm}

\begin{proof}
Suppose there are $F$-reduced paths $\pi$ and $\hat\pi$ as in the statement. Since $\Gamma_i\cap\Gamma_{i+1}$ is connected, there exists an $F$-reduced path $\rho$ in $\Gamma_1\cap\Gamma_2$ from $\iota \pi$ to $\iota \hat \pi$; similarly, there exists an $F$-reduced path $\hat \rho$ in $\Gamma_2\cap\Gamma_3$ from $\tau\hat\pi$ to $\tau \pi$. By construction, the concatenation $\rho\hat \pi\hat \rho$ is $F$-reduced. The path $\rho\hat \pi \hat \rho \pi^{-1}$ is closed, and it is a concatenation of at most 6 pieces. Therefore, it is $F$-homotopically trivial. Hence the two $F$-reduced paths $\rho\hat \pi\hat \rho$ and $\pi$ are $F$-equivalent. Observe that the property of intersecting $\Gamma_1$, respectively $\Gamma_3$, in exactly a vertex is preserved by $F$-equivalence. Thus, $\rho$ and $\hat \rho$ must be trivial (i.e. length 0), and $\pi$ and $\hat \pi$ are $F$-equivalent.
\end{proof}

\begin{lemm}\label{lem:fixed_interesection_path} Let $\Gamma_1,\Gamma_2,\dots,\Gamma_n$ be a geodesic sequence of relators with $n\geq 5$. Then there exists a geodesic path $\pi$ in $X$ contained in $\Gamma_3\cup\Gamma_4\cup\dots\cup\Gamma_{n-2}$ intersecting both $\Gamma_2$ and $\Gamma_{n-1}$ in exactly a vertex each such that: if $\Theta_1,\Theta_2,\dots,\Theta_n$ is a properly parallel geodesic sequence for $\Gamma_1,\Gamma_2,\dots,\Gamma_n$, then $\Theta_1\cup\Theta_2\cup\dots\cup\Theta_{n}$ contains $\pi$. 
\end{lemm}

\begin{proof} If there is an $F$-reduced path from $\Gamma_1$ to $\Gamma_3$ in $\Gamma_2$ intersecting $\Gamma_1$ and $\Gamma_3$ only in a vertex, respectively, that is made up of at most 2 pieces, let $x_0$ be the first vertex in it (i.e. the one still in $\Gamma_1$); otherwise let $x_0$ be any vertex in $\Gamma_1\cap\Gamma_2$. If there is an $F$-reduced path from $\Gamma_{n-2}$ to $\Gamma_n$ in $\Gamma_{n-1}$ made up of 2 pieces, let $y_0$ be the last vertex in it (i.e. the one already in $\Gamma_n$); otherwise let $y_0$ be any vertex in $\Gamma_{n-1}\cap\Gamma_n$. Let $\pi_0$ be a geodesic in $X$ from $x_0$ to $y_0$. By convexity, $\Gamma_2\cup\Gamma_3\cup\dots\cup\Gamma_{n-1}$ contains $\pi_0$. Let $\pi$ be the maximal subpath that intersects $\Gamma_2$ only in a vertex and $\Gamma_{n-1}$ only in a vertex. Let $x$ be its initial vertex and $y$ its terminal vertex.

\medskip

\noindent {\bf Case 1.} Suppose $\Gamma_2$ does not intersect both $\Theta_1$ and $\Theta_3$, and $\Gamma_{n-1}$ does not intersect both $\Theta_{n-2}$ and $\Theta_n$. First, assume $\Gamma_2$ does not intersect $\Theta_1$ but does intersect $\Theta_3$. Then $\Theta_2$ intersects $\Gamma_1$ by \autoref{lem:chords}. Using again \autoref{lem:chords}, we deduce that there exist vertices $v_1$ in  $\Gamma_1\cap\Gamma_2\cap\Theta_2$, $v_2$ in $\Gamma_2\cap\Theta_2\cap\Theta_3$, and $v_3$ in $\Gamma_2\cap\Gamma_3\cap\Theta_3$. Hence, there is an $F$-reduced path in $\Gamma_2\cap\Theta_2$ from $v_1$ to $v_2$ and an $F$-reduced path in $\Gamma_2\cap\Theta_3$ from $v_2$ to a vertex $v_3$, and we may assume that both paths do not contain edges of $\Gamma_1\cup\Gamma_3$. Note that each of the two paths is a piece. The concatenation of the two paths (after possibly performing a reduction) is an $F$-reduced path as in \autoref{lem:2pieceunique}, whence it must contain $x_0$, because $F$-equivalence preserves endpoints. Therefore, $\Theta_1\cup\Theta_2\cup\Theta_3$ contains $x_0$. 

The case that $\Gamma_2$ does not intersect $\Theta_3$ but does intersect $\Theta_1$ is symmetric. The case that $\Gamma_2$ intersects neither $\Theta_1$ nor $\Theta_3$ produces a path from $\Gamma_1$ to $\Gamma_3$ that lies in $\Gamma_2\cap\Theta_2$, i.e.\ that is a single piece, and we again deduce that $\Theta_1\cup\Theta_2\cup\Theta_3$ contains $x_0$. 

Symmetrically, we also deduce that $\Theta_{n-2}\cup\Theta_{n-1}\cup\Theta_n$ contains $y_0$ and conclude that $\Theta_1\cup\Theta_2\cup\dots\cup\Theta_n$ contains $\pi_0$ by convexity and thus also the subpath $\pi$.

\medskip

\noindent {\bf Case 2.} Suppose $\Gamma_2$ intersects both $\Theta_1$ and $\Theta_3$ or $\Gamma_{n-1}$ intersects both $\Theta_{n-2}$ and $\Theta_n$. First, assume $\Gamma_2$ intersects both $\Theta_1$ and $\Theta_3$. The possibilities are as follows:

\begin{itemize} 
\item $\Gamma_{n-1}$ intersects $\Theta_{n-2}$. Then $\Gamma_2,\Theta_3,\Theta_4,\dots,\Theta_{n-2},\Gamma_{n-1}$ is a geodesic sequence and, hence, contains $\pi_0$ by convexity and thus the subpath $\pi$. Since $\pi$ has no edges in $\Gamma_2\cup\Gamma_{n-1}$, we deduce that $\Theta_3\cup\Theta_4\cup\dots\cup\Theta_{n-2}$ contains $\pi$.
\item $\Gamma_{n-1}$ does not intersect $\Theta_{n-2}$. Then $\Theta_{n-2}\cup\Theta_{n-1}\cup\Theta_n$ contains $y_0$ as discussed in case 1. The sequence $\Gamma_2,\Theta_3,\Theta_4,\dots,\Theta_n$ is geodesic, because it is a connected subsequence of $\Theta_1,\Gamma_2,\Theta_3,\Theta_4,\dots,\Theta_n$, and it contains the endpoints of $\pi_0$. Hence, by convexity, it contains $\pi_0$ and its subpath $\pi$. Since $\pi$ has no edges in $\Gamma_2$, we have that that $\Theta_3\cup\Theta_4\cup\dots\cup\Theta_{n}$ contains $\pi$.
\end{itemize}
The case that $\Gamma_{n-1}$ intersects both $\Theta_{n-2}$ and $\Theta_n$ is symmetric.
\end{proof}

\begin{coro}\label{cor:5_intersection} Let $\Gamma_1,\Gamma_2,\dots,\Gamma_5$ and $\Theta_1,\Theta_2,\dots,\Theta_5$ be geodesic sequences that are both parallel to a geodesic sequence $\Xi_1,\Xi_2,\dots,\Xi_5$. Then $(\Gamma_1\cup\Gamma_2\cup\dots\cup\Gamma_5)\cap(\Theta_1\cup\Theta_2\cup\dots\cup\Theta_5)\neq\emptyset$.
\end{coro}
\begin{proof}
 If both are properly parallel to $\Xi_1,\Xi_2,\dots,\Xi_5$, then \autoref{lem:fixed_interesection_path} implies the claim. If $\Gamma_i=\Xi_i$ or $\Theta_i=\Xi_i$ for some $i$, the claim is obvious.
\end{proof}

\begin{lemm}\label{lem:intersection_implies_parallel} If $\Gamma_1,\Gamma_2,\dots,\Gamma_{n+k}$ and $\Theta_1,\Theta_2,\dots,\Theta_n$ are geodesic sequences with $\Gamma_1\cap\Theta_1\neq\emptyset$ and $\Gamma_{n+k}\cap\Theta_n\neq\emptyset$ for $k\geq 0$, then $k\in\{0,1,2\}$. Moreover:
\begin{itemize}
\item[0.] If $k=0$, then $\Gamma_1,\dots,\Gamma_n$ and $\Theta_1,\dots,\Theta_n$ are parallel, or $\Gamma_2,\dots,\Gamma_n$ and $\Theta_1,\dots,\Theta_{n-1}$ are parallel, or $\Gamma_1,\dots,\Gamma_{n-1}$ and $\Theta_2,\dots,\Theta_n$ are parallel. If $k=0$, and $\Theta_1=\Gamma_1$ or $\Theta_n=\Gamma_n$, then $\Gamma_1,\dots,\Gamma_n$ and $\Theta_1,\dots,\Theta_n$ are parallel. 
\item[1.] If $k=1$, then $\Gamma_1,\dots,\Gamma_n$ and $\Theta_1,\dots,\Theta_n$ are parallel, or $\Gamma_2,\dots,\Gamma_{n+1}$ and $\Theta_1,\dots,\Theta_{n}$ are parallel.
 \item[2.] If $k=2$, then $\Gamma_2,\dots,\Gamma_{n+1}$ and $\Theta_1,\dots,\Theta_{n}$ are parallel.
\end{itemize}
\end{lemm}

\begin{proof} The claim that $k\in\{0,1,2\}$ follows immediately from geodesicity. We observe: 

\begin{itemize}
\item[A.] Let $1\leq i<n+k$ and $1\leq j<n$, and let $x$ be a vertex in $\Gamma_i\cap\Theta_{j}$ and $y$ a vertex in  $\Gamma_{n+k}\cap\Theta_{n}$. Let $\alpha$ be a geodesic from $x$ to $y$ in $X$. If $\alpha$ is contained in $\Gamma_i\cap\Theta_{j}$, then $\Gamma_i\cap\Theta_{j}\cap\Gamma_{n+k}\cap\Theta_{n}\neq\emptyset$, and we have $i+1=n+k$ and $j+1=n$. Otherwise, by convexity, the first edge of $\gamma$ outside $\Gamma_i\cap\Theta_{j}$ must lie in $\Gamma_{i+1}\cup\Theta_{j+1}$. Hence, in both cases we have: $\Gamma_i$ intersects $\Theta_{j+1}$, or $\Theta_{j}$ intersects $\Gamma_{i+1}$.
\item[B.] If $\Gamma_i$ intersects $\Theta_{i+ 2}$ for some $i$, then for every $j\notin \{i,i+1\}$, $\Theta_j$ cannot intersect $\Gamma_{j+1}$, because the sequence is geodesic. The symmetric observation holds when exchanging the variables $\Gamma$ and $\Theta$.
\item[C.] $\Gamma_i$ cannot intersect $\Theta_{i\pm k}$ for $k\geq 3$.
\end{itemize}

\medskip

\noindent{\bf Case 0.} We prove claim 0, i.e.\ the case $k=0$. We break this up into three subcases.

\medskip

\noindent{\bf Case 0.1.} Suppose $\Gamma_i$ intersects $\Theta_{i+2}$ for some $1\leq i\leq n-2	$. We claim: $\Gamma_j$ intersects $\Theta_{j+1}$ for every $i\leq j\leq n-1$. Note that if $i=n-2$, then \autoref{lem:chords} yields the claim. Otherwise, we may apply observation A to $\Gamma_i$ and $\Theta_{i+2}$ to deduce: $\Gamma_i$ intersects $\Theta_{i+3}$, or $\Gamma_{i+1}$ intersects $\Theta_{i+2}$. Now C forbids the former, whence we have that $\Gamma_{i+1}$ intersects $\Theta_{i+2}$.

Now assume we have proven our claim for $j$ with $i+1\leq j\leq n-2$, and consider $\Gamma_j$ and $\Theta_{j+1}$. Observation A shows that $\Gamma_{j}$ intersects $\Theta_{j+2}$, or  $\Gamma_{j+1}$ intersects $\Theta_{j+1}$. In the first case, we are at the beginning of case 0.1 with the index $i$ replaced by $j$ and, hence, are able to show that $\Gamma_{j+1}$ intersects $\Theta_{j+2}$. In the second case, we deduce from A that $\Gamma_{j+2}$ intersects $\Theta_{j+1}$, or that $\Gamma_{j+1}$ intersects $\Theta_{j+2}$. Now, as the former is ruled out by B, the latter holds. Thus, we may use induction to conlude our claim.
Observe that proving that $\Gamma_j$ intersects $\Theta_{j+1}$ for every $1\leq j\leq i$ is symmetric. Thus, we conclude that $\Gamma_j$ intersects $\Theta_{j+1}$ for every $1\leq j\leq n-1$.

\medskip

\noindent {\bf Case 0.2.} Suppose $\Theta_i$ intersects $\Gamma_{i+2}$ for some $i$. This is symmetric to case 0.1, and $\Theta_j$ intersects $\Gamma_{j+1}$ for every $1\leq j\leq n-1$. 

\medskip

\noindent {\bf Case 0.3.} Suppose for no $i$ we have $\Gamma_i$ intersects $\Theta_{i\pm2}$. Then, by iteratively applying A, we deduce that $\Gamma_i$ intersects $\Theta_i$ for every $1\leq i\leq n$. Observe that if $\Gamma_1=\Theta_1$ or if $\Gamma_n=\Theta_n$, then we must be in this case by geodesicity.

\medskip

\noindent {\bf Case 1.} Let $k=1$. Suppose $\Gamma_1$ intersects $\Theta_2$ or $\Theta_{n-1}$ intersects $\Gamma_{n+1}$. In the first case, $\Gamma_1,\Theta_2,\Theta_3,\dots,\Theta_{n},\Gamma_{n+1}$ is a geodesic sequence, and claim 0 shows that for each $2\leq i\leq n$, $\Gamma_i$ and $\Theta_i$ intersect. In the second case, $\Gamma_1,\Theta_1,\Theta_2,\dots,\Theta_{n-1},\Gamma_{n+1}$ is a geodesic sequence and, again, claim 0 shows that for each $2\leq i\leq n$, $\Gamma_i$ and $\Theta_{i-1}$ intersect.
 
Now assume that neither $\Gamma_1$ intersects $\Theta_2$, nor does $\Gamma_{n+1}$ intersect $\Theta_{n-1}$. Then observation A implies that $\Gamma_2$ intersects $\Theta_1$ and $\Gamma_n$ intersects $\Theta_n$. If $\Gamma_i$ intersects $\Theta_i$ for all $2\leq i\leq n-1$, then we are done. If $\Gamma_i$ does not intersect $\Theta_i$ for some $2\leq i\leq n-1$, then, since $\Gamma_{n+1}$ does not intersect $\Theta_{n-1}$, claim 0 shows that $\Gamma_2,\Gamma_3,\dots,\Gamma_{n+1}$ and $\Theta_1,\Theta_2,\dots,\Theta_n$ are parallel.
 
\medskip

\noindent {\bf Case 2.} Let $k=2$. Then $\Gamma_1,\Theta_1,\Theta_2,\dots, \Theta_n,\Gamma_{n+2}$ is a geodesic sequence. We apply claim 0.
\end{proof}

\begin{coro}\label{cor:intersection_implies_parallel} Let $\Gamma_1,\Gamma_2,\dots,\Gamma_n$ and $\Theta_1,\Theta_2,\dots,\Theta_{n'}$ be geodesic sequences such that $\Gamma_1\cap\Theta_1\neq \emptyset$ and $\Gamma_n\cap\Theta_{n'}\neq\emptyset$. Then $\Gamma_2,\Gamma_3,\dots,\Gamma_{n-1}$ is parallel to a connected subsequence of $\Theta_1,\Theta_2,\dots,\Theta_{n'}$, and this subsequence contains $\Theta_3,\dots,\Theta_{n'-2}$.
\end{coro}

\begin{lemm}\label{lem:parallel_skip} Let $\Gamma_1,\Gamma_2,\dots,\Gamma_n$ and $\Theta_1,\Theta_2,\dots,\Theta_{n'}$ be geodesic sequences and $\Xi$ and $\hat\Xi$ be relators such that both $\Gamma_1$ and $\Theta_1$ intersect $\Xi$ and both $\Gamma_n$ and $\Theta_{n'}$ intersect $\hat\Xi$. Then  $\Gamma_9,\Gamma_{10},\dots,\Gamma_{n-8}$ is parallel to a connected subsequence of $\Theta_7,\Theta_8,\dots,\Theta_{n'-6}$. 
\end{lemm}

\begin{proof}
 Choose a geodesic sequence $\Xi=\Xi_1,\Xi_2,\dots,\Xi_{n''}=\hat \Xi$. \autoref{cor:intersection_implies_parallel} implies that both $\Gamma_2,\Gamma_3,\dots,\Gamma_{n-1}$ and $\Theta_2,\Theta_3,\dots,\Theta_{n'-1}$ are parallel to connected subsequences of $\Xi_1,\Xi_2,\dots,\Xi_{n''}$ that both contain $\Xi_3,\Xi_4,\dots,\Xi_{n''-2}$. 
 Hence both $\Gamma_2,\Gamma_3,\dots,\Gamma_8$ and $\Theta_2,\Theta_3,\dots,\Theta_8$ contain parallel sequences for $\Xi_3,\Xi_4,\Xi_5,\Xi_6,\Xi_7$. Thus we may invoke  \autoref{cor:5_intersection} to deduce that there exist $1\leq i,j\leq 8$ such that $\Gamma_i$ and $\Theta_j$ intersect.
 Similarly, we obtain $0\leq i',j'\leq 7$ such that $\Gamma_{n-i'}$ and $\Theta_{n'-j'}$ intersect. Now \autoref{cor:intersection_implies_parallel} implies that $\Gamma_9,\Gamma_{10},\dots,\Gamma_{n-8}$ is parallel to a connected subsequence of $\Theta_1,\Theta_2,\dots,\Theta_{n'}$.
 
 Since $\Gamma_9$ cannot intersect $\Theta_k$ for $k\leq 6$ and $\Gamma_{n-8}$ cannot intersect $\Theta_{n'-k}$ for $k\leq 5$, our claim follows.
\end{proof}

\begin{proof}[Proof of \autoref{prop:parallel}] By assumption, there exist $k,k'\leq\epsilon$ and geodesic sequences $\Xi_1,\Xi_2,\dots, \Xi_k$ and $\hat \Xi_1,\hat \Xi_2,\dots \hat \Xi_{k'}$ such that $\Xi_1$ intersects $\Gamma_1$, $\Xi_k$ intersects $\Theta_1$, $\hat \Xi_1$ intersects $\Gamma_n$ and $\hat \Xi_{k'}$ intersects $\Theta_{n'}$. Let $K:=\max\{k,k'\}$. Now for each $1\leq i\leq K$, choose a geodesic sequence $\overline\Xi_i$ from $\Xi_{\min\{i,k\}}$ to $\hat \Xi_{\min\{i,k'\}}$. Denote by $\overline \Xi_{K+1}$ the sequence $\Theta_1,\Theta_2,\dots,\Theta_{n'}$.

We claim: for each $2\leq i\leq K+1$, $\Gamma_{1+8(i-1)},\dots,\Gamma_{n-8(i-1)}$ is parallel to a connected subsequence of $\xi_i$ of $\overline\Xi_i$ that does not contain the endpoints of $\overline\Xi_i$. We argue by induction. For $i=2$, the claim follows from \autoref{lem:parallel_skip}. Now suppose we have shown our claim for some $2\leq k\leq K+1$. Then $\xi_k$ is a parallel to a connected subsequence of $\Xi_{k+1}$ by \autoref{lem:intersection_implies_parallel}. Therefore, \autoref{lem:parallel_skip} implies that $\Gamma_{1+8k},\dots,\Gamma_{n-8k}$ is parallel to a connected subsequence of $\Xi_{k+1}$, and the claim is proved.


Thus $\Gamma_{1+8\lfloor\epsilon\rfloor},\dots,\Gamma_{n-8\lfloor\epsilon\rfloor}$ is parallel to a connected subsequence of $\Theta_{1},\dots,\Theta_{n'}$. Now since $d_{\dot X}(\Gamma_1,\Theta_1)\leq \epsilon$, we have that $\Gamma_{1+8\lfloor\epsilon\rfloor}$ cannot intersect $\Theta_j$ with $j> 9\epsilon+3$, and a similar obervation holds of $\Gamma_{n-8\lfloor\epsilon\rfloor}$ and $j<n'-9\epsilon-2$, whence our result follows.
\end{proof}

\subsubsection{The action on parallels in $\dot X$}

We show that, given two geodesic sequences $\gamma$ and $\theta$ of relators, there are only boundedly many elements in $g\in G(\Gamma)$ such that $\gamma$ and $g\theta$ are parallel.

\begin{prop}\label{prop:action} Let $\Gamma$ be a $C'(1/6,p)$-labelled graph, or a $C_*'(1/6,p)$-labelled graph. Let $\Gamma_1,\Gamma_2,\dots,\Gamma_n$ and $\Theta_1,\Theta_2,\dots,\Theta_n$ be parallel geodesic sequences with $n\geq 21$. Then there exist at most $8p+100$ elements $g\in G(\Gamma)$ such that $\Gamma_1,\Gamma_2,\dots,\Gamma_n$ and $g\Theta_1,g\Theta_2,\dots,g\Theta_n$ are parallel.
\end{prop}

We first explain how this yields \autoref{thm:acylindricity}.

\begin{proof}[Proof of \autoref{thm:acylindricity} using \autoref{prop:action}] Let $\epsilon \geq0$, and let $x\in G$ with $d_{\dot X}(1,x)=n\geq 18\epsilon+25$.  Let $g\in G(\Gamma)$ with $d(1,g)\leq\epsilon$ and $d(x,gx)\leq\epsilon$. Let $\Gamma_1,\dots,\Gamma_n$ be a geodesic sequence of relators from $1$ to $x$. Then $d_{\dot X}(\Gamma_1,g\Gamma_1)\leq\epsilon$ and $d_{\dot X}(\Gamma_n,g\Gamma_n)\leq\epsilon$. Hence, there exist $0\leq k,k',l,l'\leq 9\epsilon+3$ for which the statement of \autoref{prop:parallel} holds. Observe that since $k+l=k'+l'$, the choice of $k,l,k'$ determines $l'$. We obtain parallel sequences of length at least $n-k-l+2\geq n-2(9\epsilon+3)+2\geq 21$. Thus, we are in the case of \autoref{prop:action} and, given $l,l',k,k'$, there are at most $8p+100$ possibilities for $g$. Therefore, there are at most $(9\epsilon+4)^3(8p+100)$ possibilities for $g$ in total.
\end{proof}

We will now prove Proposition~\ref{prop:action}. In the following, we still assume that $\Gamma$ satisfies the $\scc$-condition or the $\scsc$-condition.

\begin{lemm}\label{lem:stabilizing_segments} Let $\Gamma_1,...,\Gamma_n$ be a geodesic sequence, and let $g$ with $g\Gamma_1=\Gamma_1$ and $g\Gamma_n=\Gamma_n$. Then we have $g\Gamma_i=\Gamma_i$ for all $1\leq i\leq n$.
\end{lemm}

\begin{proof} We proceed by contradiction and, hence, may assume $n\geq 3$ and $g\Gamma_i\neq \Gamma_i$ for all $2\leq i\leq n-1$. Note that being geodesic implies $g\Gamma_j\neq \Gamma_i$ for every $i\neq j$. Let $\alpha$ be a geodesic 
in $X$, starting in $\Gamma_1$ and ending in $\Gamma_n$. 
Since $\alpha$ starts in $\Gamma_1=g\Gamma_1$ and ends in $\Gamma_n=g\Gamma_n$, both 
$\Gamma_1\cup\Gamma_2\cup\dots\cup\Gamma_n$ and $g\Gamma_1\cup g\Gamma_2\cup\dots\cup g\Gamma_n$ contain $\alpha$ by convexity. Let $\alpha'$ be the subpath of $\alpha$ from the last vertex in $\Gamma_1$ to the first vertex in $\Gamma_n$. Then $\alpha'$ is a path from $\Gamma_1\cap g\Gamma_1$ to $\Gamma_n\cap g\Gamma_n$ in $(\Gamma_1\cup \dots\cup\Gamma_n)\cap(g\Gamma_1\cup ...\cup g\Gamma_n)$ as in the statement of Lemma~\ref{lem:uniqueintersectionpath}.

Consider the path $g\alpha'$. The same arguments as above, together with the facts $\Gamma_1=g\Gamma_1$ and $\Gamma_n=g\Gamma_n$, show that $g\alpha'$ is a path from $\Gamma_1\cap g\Gamma_1$ to $\Gamma_n\cap g\Gamma_n$ in $(\Gamma_1\cup \dots\cup\Gamma_n)\cap(g\Gamma_1\cup \dots\cup g\Gamma_n)$ as in the statement of Lemma~\ref{lem:uniqueintersectionpath}. Hence, by the uniqueness statement of Lemma~\ref{lem:uniqueintersectionpath}, we have $g\alpha'=\alpha'$ by Lemma~\ref{lem:uniqueintersectionpath}. Therefore, $g=1$. This is a contradiction to $g\Gamma_2\neq \Gamma_2$.
\end{proof}

The following lemma is a variation on \cite[Lemma~4.11]{Gruber:2014wo}.
\begin{lemm}\label{lem:tripleautomorphism} Let $\Gamma$ be a $\scc$-labelled graph or a cylinder-free $\scsc$-labelled graph. Let $\Gamma_0,\Gamma_1,\Gamma_2$ be a geodesic sequence of relators. If $g\in G(\Gamma)$ satisfies $g\Gamma_i=\Gamma_i$ for all $i$, then $g=1$.
\end{lemm}

\begin{proof} Let $\gamma_1$ be a path in $\Gamma_1$ intersecting both $\Gamma_0$ and $\Gamma_2$ such that $|\gamma_1|=d_X(\Gamma_0,\Gamma_2)$. 
Let $g\in G(\Gamma)$ with $g\Gamma_i=\Gamma_i$ for all $i\in\{0,1,2\}$. Then left-multiplication by $g$ gives rise to label-preserving automorphisms $\phi_i$ of the $\Gamma_i$ with $\phi_0(\iota\gamma_1)=\phi_1(\iota\gamma_1)$ and $\phi_2(\tau\gamma_1)=\phi_1(\tau\gamma_1)$. We show that $\phi_1$ is the identity, such that $g$ must be the identity as well.

Suppose $\phi_1$ is not the identity. Then there exist locally geodesic paths $\pi$ from $\iota \gamma_1$ to $\phi_1(\iota\gamma_1)$ in $\Gamma_0\cap\Gamma_1$ and $\rho$ from $\tau\gamma_1$ to $\phi_1(\tau\gamma_1)$ in $\Gamma_1\cap\Gamma_2$. Note that both $\pi$ and $\rho$ are pieces and, more strongly, for any $k$, the paths $\pi_k:=\pi\phi_1(\pi)\phi_1^2(\pi)\dots\phi_1^k(\pi)$ and $\rho_k:=\rho\phi_1(\rho)\phi_1^2(\rho)\dots\phi_1^k(\rho)$ are pieces. Moreover, so are locally geodesic paths with the same endpoints and labelled by the same elements of $F$ as $\pi_k$, respectively $\rho_k$; we denote these by $\hat \pi_k$ and $\hat \rho_k$.

We first consider the free case: if $\phi_1$ is not the identity, then $\ell(\rho_k)=\ell(\rho)^k$ in $F$, whence $\ell(\pi)^k$ is freely non-trivial. Also observe that no subpath of $\hat \pi_k$ or of $\hat \rho_k$ can be closed, as otherwise we would have a simple closed path that is a piece. In particular, $\phi_1$ must have infinite order. By construction, for any $k$, the path $\theta_k:=\gamma_1\hat\rho_k\phi_1^k(\gamma_1^{-1})\hat\pi_k^{-1}$ is a simple closed path. Since $|\hat\pi_k|\to\infty$ and $|\hat\rho_k|\to\infty$ as $k\to\infty$ and $\hat\pi_k$ and $\hat\rho_k$ are pieces, this path eventually violates the $C'(1/6)$-condition.

In the free product case, if $\pi$ or $\rho$ are not contained in a single attached Cayley graph, then, again $|\hat\pi_k|\to\infty$ or $|\hat\rho_k|\to\infty$ as $k\to\infty$, and, again, the $\scsc$-condition is violated eventually. Hence, we may assume that $\pi$ is contained in an attached component $C_1$ and $\rho$ is contained in an attached component $C_2$. Observe that, since $\phi(\iota \pi)=\tau \pi$, we have $\phi(C_1)\cap C_1\neq\emptyset$, whence $\phi$ leaves $C_1$ invariant. The same observation holds for $C_2$. Since $\Gamma_0$ and $\Gamma_2$ are disjoint, so are $C_1$ and $C_2$. Thus, cylinder-freeness implies the claim.
\end{proof}

\begin{lemm}\label{lem:path_intersection_shift}
Assume that $\Gamma$ satisfies the $C'(1/6,p)$-condition or the $C_*'(1/6,p)$ condition, and let $\epsilon> 0$. Let $\gamma$ and $\theta$ be geodesics in $X$, with $d_{\dot X}(\iota\gamma,\tau\gamma)>1$ and $d_{\dot X}(\iota\theta,\tau\theta)>1$. Then there exist at most $\lceil\epsilon p\rceil$ elements $g\in G(\Gamma)$ such that $g\theta$ is an $F$-subpath of $\gamma$ with $d_{\dot X}(\iota\gamma,g\iota\theta)\leq \epsilon$. 
\end{lemm}

\begin{proof}[Proof in the free group case]
Let $r$ be a shortest word in the generators such that there exist $w$ a proper initial subword of $r$ and $N$ a positive integer with: $r^Nw$ is a reduced word and $r^Nw=\ell(\theta)$. Being shortest, $r$ is not a proper power of a word. We thus have: if $x$ is a word such that $xr$ is an initial subword of $r^Nw$, then there exists some $k$ with $x=r^k$. 

For simplicity of notation, assume that $\theta$ is a subpath of $\gamma$ as required in the assumption, i.e. $g=1$ satisfies the claim, and, moreover, $d_X(\iota\gamma,\iota\theta)$ is minimal among all possible $G(\Gamma)$-translates of $\theta$ which are subpaths of $\gamma$. (This is no restriction, up to replacing $\theta$ by a translate.) Let $g\in G(\Gamma)$ be non-trivial satisfying the assumptions, and let $x$ be the label of the subpath of $\gamma$ from $\iota\theta$ to $g\iota\theta$.  Then $xr$ is an initial subword of $r^Nw$. Thus $x=r^k$ for some $k$ by the above observation. It remains to restrict $k$.

By \cite[Proposition~3.6]{Gruber:2014wo}, a path in $X$ labelled by geodesic word $y$ can be covered by at most $|y|_{\dot X}$ relators. Since $d_{\dot X}(1,r^Nw)>1$, we have that not all powers of $r$ appear on $\Gamma$. Hence, by the $p$-condition, at most the $p-1$-st power occurs, and the same holds true for any cyclic conjugate of $r$. Hence, $\epsilon\geq|r^k|_{\dot X}> k/p$, whence $1\leq k<\epsilon p$. Thus, including the case $g=1$, we get at most $\lceil\epsilon p\rceil$ elements.
\end{proof}

We now give (local) terminology for the free product case: if $w$ is a word in the generators, then a subword $u$ is a \emph{syllable} if it is a maximal subword whose letters come from a single generating factor. A word $w$ is \emph{reduced} if all its syllables represent non-trivial elements of their respective generating factors. A concatenation of non-empty reduced words $w_1w_2\dots w_k$ is \emph{strongly reduced} if the terminal syllable of $w_i$ is in a different generating factor than the initial syllable of $w_{i+1}$. An \emph{initial $F$-subword} $u$ of $w$ is defined as follows: if $u$ is a word with $k$ syllables, then the first $k-1$ coincide with the first $k-1$ syllables of $w$, and the $k$-th syllable lives in the same generating factor as the $k$-th syllable of $w$; $u$ is a \emph{proper} initial $F$-subword if it has fewer syllables than $w$.

\begin{proof}[Proof in the free product case]
Let $r$ be a shortest word in the generators such that there exist $w$ a proper initial $F$-subword of $r$ and $N$ a positive integer with: $r^Nw$ is strongly reduced and $r^Nw=\ell(\theta)$ in $F$. Observe that, since $d_{\dot X}(\iota \theta,\tau\theta)>1$, $\ell(\theta)$ is not contained in a single generating factor and, therefore, neither is $r$. Hence, being shortest, $r$ does not represent a proper power of an element of $F$. We thus have: if $x$ is a word such that $xr$ is strongly reduced and equal in $F$ to an initial subword of $r^Nw$, then there exists some $k$ with $x=r^k$ in $F$. 

We repeat the proof as before: for simplicity of notation, assume that $\theta$ is an $F$-subpath of $\gamma$ as required in the assumption, i.e. $g=1$ satisfies the claim, and, moreover, $d_X(\iota\gamma,\iota\theta)$ is minimal among all possible $G(\Gamma)$-translates of $\theta$ which are $F$-subpaths of $\gamma$. (This is no restriction, up to replacing $\theta$ by a translate.) Let $g\in G(\Gamma)$ be non-trivial satisfying the assumptions, and let $x$ be the label of the subpath of $\gamma$ from $\iota\theta$ to $g\iota\theta$.  Then $xr$ is strongly reduced by definition of an $F$-subpath, and it is equal in $F$ to an initial subword of $r^Nw$. Thus $x=r^k$ in $F$ for some $k$ by the above observation. It remains to restrict $k$.

By \cite[Proposition~3.6]{Gruber:2014wo}, a path in $X$ labelled by geodesic word $y$ can be covered by at most $|y|_{\dot X}$ relators. Also observe: given two $F$-equivalent paths, both are covered by the same collection of relators. We conclude as in the free group case.
\end{proof}

\begin{proof}[Proof of Proposition~\ref{prop:action}] Let $N=10$. First, let $g\in G(\Gamma)$ such that $\Gamma_1,\Gamma_2,\dots,\Gamma_N$ and $g\Theta_1,g\Theta_2,\dots,g\Theta_N$ are properly parallel. Let $\gamma$ be the geodesic path in $\Gamma_3\cup\Gamma_4\cup\dots\cup\Gamma_{N-2}$ from $\Gamma_2$ to $\Gamma_{N-1}$  that is contained in every properly parallel sequence for $\Gamma_2,\Gamma_3,\dots,\Gamma_{N-1}$ obtained in Lemma~\ref{lem:fixed_interesection_path}. In particular, $\gamma$ contains no edges of $\Gamma_2$ or $\Gamma_{N-1}$. Let $\theta$ be a geodesic path in $\Theta_5\cup\Theta_6\cup\dots\cup\Theta_{N-4}$ from $\Theta_4$ to $\Theta_{N-3}$ that is contained in every sequence properly parallel to $\Theta_4,\dots,\Theta_{N-3}$ from Lemma~\ref{lem:fixed_interesection_path}. Note that $\gamma$ and $\theta$ are defined independently of $g$.

Let $T:=(\Gamma_1\cup\Gamma_2\cup\dots\cup\Gamma_N)\cap(g\Theta_1\cup g\Theta_2\cup\dots\cup g\Theta_3)$, which, by Lemma~\ref{lem:intersection_tree}, is an $F$-tree. Let $\pi_0$ be the up to $F$-equivalence unique path in $T$ from $\Gamma_1\cap g\Theta_1$ to $\Gamma_N\cap g\Theta_N$ that does not contain edges of $\Gamma_1\cap \Theta_1$ or of $\Gamma_N\cap \Theta_N$. Such a path exists by construction and is unique (up to $F$-equivalence) because we are considering connected subsets of an $F$-tree. Let $\pi$ be the $F$-subpath of $\pi_0$ that intersects $\Gamma_2$ and $\Gamma_{N-1}$ in exactly a vertex each. Since $\gamma$ is contained in $T$ and it is an $F$-reduced path that intersects $\Gamma_2$ and $\Gamma_{N-1}$ in exactly a vertex each, we conclude that, up to $F$-equivalence, $\pi$ and $\gamma$ coinicide.

Observe that $\Gamma_2\cap g\Theta_5=\emptyset$ and $g\Theta_{N-4}\cap\Gamma_{N-1}=\emptyset$, as we are considering geodesic sequences. Therefore, the initial vertex of $\gamma$ is contained in $g\Theta_1\cup g\Theta_2\cup g\Theta_3\cup g\Theta_4$ and, likewise, its terminal vertex is contained in $g\Theta_{N-3}\cup g\Theta_{N-2}\cup g\Theta_{N-1}\cup g\Theta_{N}$. Therefore, $\gamma$ contains an $F$-subpath in $g\Theta_5\cup g\Theta_6\cup\dots\cup g\Theta_{N-4}$ intersecting both $g\Theta_4$ and $g\Theta_{N-3}$ in exactly a vertex. As above, we conclude that, up to $F$-equivalence, this $F$-subpath coincides with $g\theta$. In other words: $g\theta$ is an $F$-subpath of $\gamma$. We have $d_{\dot X}(\iota\gamma,\tau\gamma)\geq (N-1)-2-1>1$ and $d_{\dot X}(\iota\theta,\tau\theta)\geq (N-3)-4-1>1$. Moreover, $d_{\dot X}(\iota \gamma,g\iota\theta)\leq 4$. Thus, we may apply Lemma~\ref{lem:path_intersection_shift} with $\epsilon=4$ to conclude: there exist at most $4p$ elements $g\in G(\Gamma)$ such that the initial subsegmens of length $N$ of $\Gamma_1,\Gamma_2,\dots,\Gamma_n$ and $g\Theta_1,g\Theta_2,\dots,g\Theta_n$ are properly parallel.

For the case that the terminal subsegments of length $N$ are properly parallel, the same observation holds, bringing our count of elements $g$ to at most $8p$.

Finally, consider the case that both the initial subsegment and the terminal subsegment of length $N$ are not properly parallel. This means there exist $1\leq i\leq N$ and $n-N+1\leq j\leq n$ with $\Gamma_i=g\Theta_i$ and $\Theta_j=g\Gamma_j$. If we have another $g'$ with  $\Gamma_i=g'\Theta_i$ and $\Theta_j=g'\Gamma_j$, then Lemma~\ref{lem:stabilizing_segments} shows that $g'\Theta_t=g\Theta_t$ for every $i\leq t\leq j$. In particular, this holds for $N\leq t\leq N+2$, because $N+2\leq n-N+1$. Therefore, given $i$ and $j$, there exists at most 1 such $g$ by Lemma~\ref{lem:tripleautomorphism}. This gives a grand total of at most $8p+N^2$ elements.
\end{proof}


\subsection{Characterization of elliptic elements}
\label{sec: elliptics}

In this section, we give a complete characterization of the elliptic elements for the action of $G(\Gamma)$ on $\dot X$. In particular, we show that every element is either elliptic or hyperbolic. As a corollary, we obtain a generalization of the Torsion Theorem for small cancellation groups to graphical small cancellation theory, even over free products, thus generalizing the corresponding classical results \cite{Greendlinger1960torsion,McCool1968}, see also \cite{Gro03,Oll06,Gruber:2015fu,Steenbock2015} for results on the torsion-freeness of certain graphical small cancellation groups.
\begin{prop}\label{prop:elliptic}
Let $\Gamma$ be a $\scc$-labelled graph or a $\scsc$-labelled graph, and let $g\in G(\Gamma)$. Then the following are equivalent:
\begin{itemize}
\item $g$ is elliptic for the action on $\dot X$,
\item $g$ is not hyperbolic for the action on $\dot X$,
\item $g$ is conjugate to an element of $G(\Gamma)$ represented by a word all whose powers occur on $\Gamma$.
\end{itemize}
\end{prop}

\begin{proof}
Let $g\in G(\Gamma)$, and let $w$ be a word in $S$ minimizing the lengths of words representing elements of the set $\{h\in G(\Gamma):\exists n\in \N: h^n\sim g\}\subseteq G(\Gamma)$, where $\sim$ denotes conjugacy in $G(\Gamma)$. We consider the following two cases studied by Gruber in \cite[Section~4]{Gruber:2015ii}:
\begin{itemize}
\item[(1)]\label{case:elliptic} every power of $w$ occurs on $\Gamma$, or
\item[(2)]\label{case:hyperbolic} there exists some $C_0>0$ such that the longest subword of a power of $w$ occurring on $\Gamma$ has length at most $C_0$.
\end{itemize}

In case (1), the element $h$ represented by $w$ is elliptic, whence so is $g$. It remains to show that in case (2), $h$ is hyperbolic, which will imply hyperbolicity of $g$.

It is shown in \cite[Section 4]{Gruber:2015ii} that, for every $n$, there exist a geodesic word $g_n$ representing $w^n$ and a $\Gamma$-reduced diagram $B_n$ with boundary word $w^ng_n^{-1}$ such that
\begin{itemize}
\item every disk component of $B_n$ is a single face or has shape $I_1$ as in Figure~\ref{figure:bigons}, its two sides being subpaths of the sides of $B_n$ corresponding to $w^n$ (denoted $\omega_n$) and corresponding to $g_n^{-1}$ (denoted $\gamma_n$), 
\item every face $\Pi$ of $B_n$ has $|\partial\Pi|<6C_0$,
\item the intersection of every face $\Pi$ with $\omega_n$ has length less than $(2/3)|\partial \Pi|$.
\item if $\Pi$ and $\Pi'$ are consecutive faces in a disk component, then $|\Pi\cap\gamma_n|>|\partial\Pi|/6$ or $|\Pi'\cap\gamma_n|>|\partial\Pi'|/6$. (Here, consecutive means: sharing interior edges.)
\end{itemize}
The first 3 bullets are stated explicitly in \cite{Gruber:2015ii}, and the last bullet is deduced as follows: if both $\Pi$ and $\Pi'$ do not satisfy the claim, then $|\Pi\cap\omega_n|>|\Pi|/2$ and $|\Pi'\cap\omega_n|>|\partial\Pi'|/2$ by the small cancellation hypothesis. Therefore, both $|\partial\Pi|\geq2|w|$ and $|\partial\Pi'|\geq 2|w|$ by the minimality hypothesis on $w$. Hence, \cite[Lemma~4.10]{Gruber:2015ii} shows that both $\Pi$ and $\Pi'$ are special in the sense of \cite[Lemma~4.11]{Gruber:2015ii}, whence we may apply \cite[Lemma~4.11]{Gruber:2015ii} as follows: if $a$ is the arc in the intersection of $\Pi$ and $\Pi'$, then $|\Pi\cap\omega_n|+|a|<|w|+|\partial\Pi|/6\leq 2|\partial\Pi|/3$. Apart from $a$, $\Pi$ has at most one additional interior arc, and this arc must have length less than $|\partial\Pi|/6$. Therefore, $|\Pi\cap\gamma_n|>|\partial\Pi|/6$, a contradiction.

Let $\sigma$ be a covering segment, i.e. a copy of a path in $\Gamma$ that is a subpath of $\gamma_n$. If $\sigma$ completely contains $\partial\Pi\cap\gamma_n$ for a face $\Pi$ with $|\partial\Pi\cap\gamma_n|>|\partial\Pi|/6$, then $\sigma\cup\Pi^{(1)}$ lifts to $\Gamma$, because $\sigma\cap\Pi$ is not a piece.  We show that $\sigma$ cannot contain $(\Pi_1\cup\Pi_2)\cap\gamma_n$ for 2 consecutive faces $\Pi_1$ and $\Pi_2$ in a disk component: for a contradiction, assume $\sigma$ does contain it. Suppose that $|\Pi_1\cap\sigma|>|\partial\Pi_1|/6$. (Here we use the fourth bullet; the case where this holds for $\Pi_2$ follows by symmetry.) Then $\sigma\cup\Pi_1^{(1)}$ lifts to $\Gamma$, because $\sigma\cap\Pi_1$ is not a piece. Since $B_n$ is $\Gamma$-reduced, two consecutive faces in a disk component cannot lift together to $\Gamma$. Hence, $(\Pi_1\cap\Pi_2)\cup(\Pi_2\cap\sigma)$ is a piece (considering this embedded line graph as path). Thus, $\partial\Pi_2$ is a concatenation of at most 2 pieces and $\Pi_2\cap\omega_n$. But this implies $|\Pi_2\cap\omega_n|>(2/3)|\partial\Pi_2|$, a contradiction to the third bullet. Hence, $\sigma$ cannot contain the intersection of $\gamma_n$ with a disk component of $B_n$ of shape $I_1$, and, moreover, the parts of $\sigma$ contained in disk components of shape $I_1$ have total length less than $4(6C_0)/2=12C_0$ (using the second bullet).

If $\sigma$ contains $\Pi\cap\gamma_n$, where $\Pi$ is its own disk component, then, $\sigma\cup\Pi^{(1)}$ lifts to $\Gamma$, and we can replace $\sigma$ by a path $\tilde \sigma$  containing $\Pi\cap \omega_n$ instead of $\Pi\cap\gamma_n$, and $\tilde\sigma$ is also a copy of a path in $\Gamma$. Since $g_n$ is a geodesic word, we have $|\tilde\sigma|\geq |\sigma|$. We conclude that the parts of $\sigma$ contained in disk components consisting of single faces together with the parts of $\sigma$ not contained in any faces at all have total length at most $C_0$.

We conclude $|\sigma|<13C_0$. Since every disk component of $B_n$ has shape $I_1$, we have $|g_n|\geq n|w|/C_0$.  Hence, by \cite[Proposition~3.6]{Gruber:2014wo}, we obtain $d_{\dot X}(1,h^n)> n(|w|/C_0)/(13C_0)$, whence $h$ and, therefore, $g$ is hyperbolic.
\end{proof}

\begin{coro}[Torsion theorem]\label{cor:torsion_theorem} Let $\Gamma$ be a $\scc$-labelled graph or $\scsc$-labelled graph, and let $g\in G(\Gamma)$ be of order $n\in\N\setminus \{1\}$. Then there exist a connected component $\Gamma_0$ of $\Gamma$ and a label-preserving automorphism $\phi$ of $\Gamma_0$ of order $n$ such that $g$ is conjugate to the element of $G(\Gamma)$ represented by the label of a path $v\to\phi(v)$ for a vertex $v$ in $\Gamma_0$, or (only in the free product case) $g$ is conjugate to an order $n$ element of a generating factor.
\end{coro}

\begin{proof} Since $g$ has finite order, it must act elliptically on $\dot X$. Therefore, by Proposition~\ref{prop:elliptic}, a conjugate of $g$ is represented by a word $w$ such that every power of $w$ occurs on $\Gamma$. Now since every component of $\Gamma$ embeds into $\Cay(G(\Gamma),S)$, non-triviality and finite order of $g$ imply that some proper power of $w$ must occur on a \emph{closed} path $\gamma$ in a component $\Gamma_0$ of $\Gamma$, say starting from some vertex $v$, labelled by $w^k$ for some $k>1$. We choose $k$ minimal with this property. Then $g$ has order $k$.

If (in the free product case) $w$ is contained in a single generating factor, then the second claim holds, because each generating factor embeds in $G(\Gamma)$. Otherwise, $\gamma$ is not $F$-homotopically trivial, and the small cancellation condition implies that $\gamma$ cannot be a piece. Thus, there exists an automorphism $\phi$ of $\Gamma_0$ such that $\phi(v)$ lies on $\gamma$ and the initial subpath $\pi$ of $\gamma$ from $v$ to $\phi(v)$ is labelled by $w$. Since the path $\pi\phi(\pi)\phi^2(\pi)\dots\phi^{k-1}(\pi)$ is labelled by $w^k$ and, therefore, is closed, we obtain that $\phi$ has order $k$. This argument also applies in the free group case.
\end{proof}

\subsection{Description of maximal elementary subgroups}

We show that the elementary closure of every hyperbolic element is infinite cyclic or infinite dihedral. In particular, if there is no even torsion, it must be infinite cyclic.

\begin{prop}\label{prop:virtually_cyclic} Let $\Gamma$ be a $\scc$-labelled graph or a cylinder-free $\scsc$-labelled graph. Let $g$ be a hyperbolic element for the action of $G(\Gamma)$ on $\dot X$, and let $h$ be an elliptic element such that $g$ and $h$ commute. Then $h=1$.
\end{prop}

\begin{proof}
Proposition~\ref{prop:elliptic} tells us that there exists $t\in G(\Gamma)$ such that the conjugate $tht^{-1}$ of $h$ is represented by a (possibly empty) cyclically reduced word all whose powers appear on $\Gamma$. We have that $tht^{-1}$ and $tgt^{-1}$ commute if and only if $g$ and $h$ do, $h$ is elliptic if and only if $tht^{-1}$ is, and $g$ is hyperbolic if and only if $tgt^{-1}$ is. Hence, without loss of generality, we assume that $h$ itself is represented by a cyclically reduced word $w$ such that all powers of $w$ appear on $\Gamma$. In the free product case, we can additionally assume that $w=w_1w_2\dots w_k$, where each $w_j$ contains letters from a single generating factor and does not represent the identity in that factor, any two consecutive $w_j$ correspond to different generating factors, and, if $k>1$, then $w_1$ and $w_k$ do not correspond to the same generating factor.

First, assume there exists a path $\pi$ in a component $\Gamma_0$ of $\Gamma$ such that $\pi$ is labelled by $w$, and such that there exists an automorphism of $\Gamma_0$ that takes $\iota \pi$ to $\tau \pi$. This implies: if $\Gamma_1$ is the relator in $X$ that is the image of $\Gamma_0$ under the map induced by $\iota \pi\mapsto 1\in G(\Gamma)$, then $h\Gamma_1=\Gamma_1$.  Now choose $k$ such that $d_{\dot X}(1,g^k)\geq 3$. In particular, this ensures $\Gamma_1\cap g\Gamma_1=\emptyset.$ Choose a geodesic sequence $\Gamma_1,\Gamma_2,\dots,\Gamma_l=g\Gamma_1$. Then $h\Gamma_1=\Gamma_1$ and $h\Gamma_l=hg\Gamma_1=gh\Gamma_1=g\Gamma_1=\Gamma_l$. Therefore, Lemmas~\ref{lem:stabilizing_segments} and \ref{lem:tripleautomorphism} prove the claim.

Now assume that there does not exist such an automorphism as above. This implies: any path labelled by a power of $w$ is a piece in $\Gamma$. In the free group case, denote by $\Theta$ a bi-infinite line graph labelled by the powers of $w$. In the free product case, denote by $\Theta$ the completion (see Definition~\ref{def: small cancellation condition - free product}) of a bi-infinite line-graph labelled by the powers of $w$. Consider the graph $\hat \Gamma:=\Gamma\sqcup\Theta$. We claim that  $\hat\Gamma$ satisfies the same small cancellation condition as $\Gamma$: all powers of $w$ label pieces in $\Gamma$. Hence, in the free group case, adding $\Theta$ does not introduce new pieces. In the free product case, observe that whenever a path $p$ in $\Gamma$ is a piece, then so is every path in the completion of the support of $p$. Therefore, again, adding $\Theta$ does not introduce new pieces. Constructing the completion of $\Theta$ as described in \autoref{rema:cylinder_free_classical} readily shows that all closed paths in $\Theta$ are $F$-trivial. Thus, there are also no new simple closed paths to consider when checking the small cancellation condition. 

We have observed that $\hat \Gamma$ defines the same group as $\Gamma$ and thus the same Cayley graph $X$. Furthermore, the assumption that all powers of $w$ appear on $\Gamma$ implies that $\hat\Gamma$ also yields the same coned-off space $\dot X$. Now the connected component $\Theta$ of $\hat\Gamma$ admits a path $\pi$ labelled by $w$ and an autmorphism taking $\iota \pi$ to $\tau \pi$. Hence, we may apply our above argument for $\hat\Gamma$ instead of $\Gamma$.
\end{proof}

\begin{coro}\label{coro:virtually_cyclic}
 Let $H$ be a virtually cyclic subgroup of $G(\Gamma)$ that contains a hyperbolic element. Then $H$ is either infinite cyclic or infinite dihedral. In particular, if $G(\Gamma)$ has no even torsion, then $H$ is infinite cyclic.
\end{coro}

\begin{proof}
As $H$ is virtually infinite cyclic, we may write $H=K\ltimes C$, where $C$ is either infinite cyclic or infinite dihedral and contains a hyperbolic element $g$. Both $C_1:=\langle g\rangle$ and the kernel $C_2$ of the action by conjugation of $C$ on $K$ have finite index in $C$. Thus, $C_1\cap C_2$ is non-trivial and contains a hyperbolic element commuting with every element of $K$, whence $K$ is trivial.
\end{proof}

\subsection{The case that $G(\Gamma)$ acts elementarily}

The following proposition will handle the degenerate case of Theorem~\ref{res: main theo - regular sc}.

\begin{prop}\label{prop:elementary} Let $p\in \N$, $n$ odd, and $\Gamma$ be a $C_n'(1/6,p)$-labelled graph whose set of labels $S$ has at least two elements. Then either $G(\Gamma)$ acts non-elementarily on $\dot X$, or $G(\Gamma)\in\vburn{n}$. If $G(\Gamma)$ is finite, then $\Gamma$ contains $\Cay(G(\Gamma),S)$.
\end{prop}
 
 Notice that $p$ does not appear in the proof. The actual (weaker) version of \ref{enu: graphical small cancellation - burnside - power} of \autoref{def: graphical small cancellation - burnside} we require is: whenever, for a cyclically reduced word $w$, all powers of $w$ label paths in $\Gamma$, then $w^n$ labels a closed path in $\Gamma$.
 
\begin{proof} 

We shall assume that any two connected components of $\Gamma$ are non-isomorphic. Suppose $G:=G(\Gamma)$ contains a hyperbolic element. Then, since the action is acylindrical by Theorem~\ref{thm:acylindricity}, $G$ is either virtually cyclic or acts non-elementarily on $\dot X$. 

Observe that $\Gamma$ contains neither loops nor bigons (i.e.\ simple closed paths of length 1 or 2): in view of the strong reducedness assumption that is part of the $C_n'(1/6,p)$-condition, it only remains to rule out the existence of a simple closed path $\gamma$ with label $s^2$ for some $s\in S$. If such a $\gamma$ exists, then all powers of $s$ label paths in $\Gamma$, whence there exists a closed path $\gamma'$ with label $s^n$ by the $C_n'(1/6,p)$-condition. Since the labelling of $\Gamma$ is reduced, a path in $\Gamma$ is uniquely determined by its starting vertex and its label. Therefore, since $n$ is odd, $\gamma$ and $\gamma'$ cannot start from the same vertex, and the same is true for $\gamma$ and any translate of $\gamma'$ by a label-preserving automorphism of $\Gamma$. Thus $\gamma$ is a piece, contradicting the $C'(1/6)$-condition, whence $\gamma$ cannot exist.

If $G$ is virtually cyclic, then Corollary~\ref{coro:virtually_cyclic} implies that $G$ is infinite cyclic and, in particular, abelian. Consider two different elements $s_1$ and $s_2$ of $S$ and a $\Gamma$-reduced diagram $D$ for the word $s_1s_2s_1^{-1}s_2^{-1}$.

Since $\Gamma$ contains neither loops nor bigons, $D$ has no cut-points. If $D$ has at least two faces, then we may apply Greendlinger's lemma to deduce that $D$ has at least two faces $\Pi_1$ and $\Pi_2$ of interior degree at most 3 each. In particular, we have for each $i=1,2$ that $|\partial\Pi_i\cap\partial D|>|\partial \Pi_i\setminus\partial D|\geq 3$, whence $|\partial D|\geq 8$, a contradiction.

Thus, $D$ consists of a single face with boundary length $4$. This implies that neither $s_1$ nor $s_2$ can label pieces. Hence, there exist automorphisms $\phi_1$ and $\phi_2$ of $\Gamma$ and edges $e_i$ with $\ell(e_i)=s_i$ and $\phi(\iota e_i)=\tau e_i$. Hence the $C_n'(1/6,p)$-condition implies that the elements of $G(\Gamma)$ represented by $s_i$ have orders dividing $n$. Since $s_1$ and $s_2$ were arbitrary, $G(\Gamma)$ is a quotient of $\bigoplus_{s\in S} \Z/n\Z$, which contradicts the existence of an infinite order element.

\medskip

Thus we may assume that $G(\Gamma)$ contains no hyperbolic element. Note that since $S$ is non-empty, this implies that $\Gamma$ is non-empty. 

Assume $\Gamma$ has no component with non-trivial fundamental group. Then $G(\Gamma)$ is a non-trivial free group. As every element is elliptic, Proposition~\ref{prop:elliptic} implies that there exist arbitrarily high powers of freely non-trivial words read on paths in $\Gamma$ and, as every component has trivial fundamental group, these paths must be simple and, in particular, not closed. This contradicts the $C_n'(1/6,p)$-condition. On the other hand, assume $\Gamma$ has more than one component with non-trivial fundamental group. Then it follows from \cite{Gruber:2014wo} that $G(\Gamma)$ does contain a hyperbolic element for the action on $\dot X$, contradicting our assumption. 
Therefore, from now on we assume that $\Gamma$ has exactly one component $\Gamma_0$ with non-trivial fundamental group.

Let $g\in G(\Gamma)$ be a non-trivial element. Since $g$ is not hyperbolic, Proposition~\ref{prop:elliptic} implies that $g$ is elliptic and conjugate to an element of $G(\Gamma)$ that is represented by a word $w$ read on $\Gamma$ such that every power of $w$ can be read on $\Gamma$. By assumption, we have that $w^n$ is read on a closed path in $\Gamma_0$, and say this loop is based at a vertex $v$. The $C'(1/6)$-condition implies that a path labelled by $w^n$ cannot be a piece, whence there exists an automorphism $\phi$ of $\Gamma_0$ such that $w$ is the label of a path from $v$ to $\phi(v)$. Note that this implies that $\phi$ has order dividing $n$.

The label-preserving map $\Gamma_0\to X$ that takes $v$ to $1$ induces a homomorphism $\mathrm{Aut}(\Gamma_0)\to G(\Gamma)$, sending $\psi$ to the element represented by a path from $v$ to $\psi(v)$. Denote by $H$ the image of this homomorphism. Then our above considerations show: $G(\Gamma)=\cup_{g\in G(\Gamma)} gHg^{-1}$. Moreover, every element of $H$, and therefore every element of $G(\Gamma)$, has order dividing $n$, whence we have $G(\Gamma)\in \vburn{n}$.

Suppose $G(\Gamma)$ is finite. It is an easy exercise in group theory to show that whenever a group $G_1$ is the union the of conjugates of a finite index subgroup $G_2$, then $G_1=G_2$. Hence, if $G(\Gamma)$ is finite, then $G(\Gamma)=H$, which implies that $\Gamma$ is actually the (finite) Cayley graph of $G(\Gamma)$ with respect to $S$.
\end{proof}



%
\section{Proofs of the main results}
%
\label{sec: proofs}

We now give a proof of Theorems~\ref{res: main theo - regular sc} and \ref{res: periodic quotient of sc groups - free product}.
These statements actually follow from a more general result.
Indeed as explained at the beginning of \autoref{sec: sc to acylindrical} we are working with a graph $\Gamma$ satisfying a weaker hypotheses -- namely the $C'(\lambda,p)$ and $C'_*(\lambda,p)$ conditions -- which allows the group $G(\Gamma)$ to have infinite order elliptic elements for its action on the corresponding coned-off Cayley graph.
As usual our result has two variants, one for the usual graphical small cancellation and one for graphical small cancellation over free products.
While their proofs are exactly the same, we found it clearer to state them separately.

\begin{theo}
\label{res: periodic quotient of sc groups - standard}
	Let $p \in \N^*$ and $r \in \R_+$.
	There exists $n_{p,r} \in \N$ such that for every odd exponent $n \geq n_{p,r}$ the following holds.
	Let $\Gamma$ be a graph labelled by a set $S$ satisfying the $C'(1/6,p)$-condition.
	Assume that $G=G(\Gamma)$ has no even torsion. 
	We focus on the action of $G$ of the cone-off space $\dot X =\dot X(\Gamma)$.
	There exists a quotient $Q$ of $G$ with the following properties.
	\begin{enumerate}
		\item \label{enu: periodic quotient of sc groups - elliptic embeds}
		Every elliptic subgroup of $G$ embeds in $Q$.
		\item \label{enu: periodic quotient of sc groups - periodicity}
		For every $g \in Q$ that is not the image of an elliptic element we have $g^n = 1$.
		\item \label{enu: periodic quotient of sc groups - universal property}
		If every elliptic subgroup of $G$ belongs $\vburn n$, then $Q$ is isomorphic to $G/G^n$. 
		In particular $Q$ is isomorphic to $G_n(\Gamma)$ and belongs to $\vburn n$.
		\item \label{enu: periodic quotient of sc groups - one-to-one}
		If the action of $G$ on $\dot X$ is non elementary, then $Q$ is infinite.
		Moreover the projection $G \twoheadrightarrow Q$ is one-to-one on small balls in the following sense.
		For every $g \in G\setminus\{1\}$, for every $x \in \dot X$, if $\dist[\dot X]{gx}x \leq r$, then the image of $g$ in $Q$ is not trivial.
		In particular, if $r\geq 1$, then every connected component of $\Gamma$ embeds in the Cayley graph of $Q$ with respect to $S$.
	\end{enumerate}
\end{theo}

\begin{theo}
\label{res: periodic quotient of sc groups - product}
	Let $p \in \N^*$ and $r \in \R_+$.
	There exists $n_{p,r}\in \N$ such that for every odd exponent $n \geq n_{p,r}$ the following holds.
    Let $(F_i)_{i \in I}$ be a collection of groups.
    For each $i \in I$ we fix a generating set $S_i$ of $F_i$ and let $S =\sqcup_{i\in I}S_i$. 
	Let $\Gamma$ be a graph labelled by $S$ satisfying the $C_*'(1/6,p)$-condition.
	Assume that $G=G(\Gamma)$ has no even torsion. 
	We focus on the action of $G$ of the cone-off space $\dot X =\dot X(\Gamma)$.
	There exists a quotient $Q$ of $G$ with the following properties.
	\begin{enumerate}
		\item \label{enu: periodic quotient of sc groups - elliptic embeds - product}
		Every elliptic subgroup of $G$ embeds in $Q$. In particular, every $F_i$ embeds in $Q$.
		\item \label{enu: periodic quotient of sc groups - periodicity - product}
		For every $g \in Q$ that is not the image of an elliptic element we have $g^n = 1$.
		\item \label{enu: periodic quotient of sc groups - universal property - product}
		If every elliptic subgroup of $G$ belongs $\vburn n$, then $Q$ is isomorphic to $G/G^n$. 
		In particular $Q$ is isomorphic to $G_n(\Gamma)$ and belongs to $\vburn n$.
		\item \label{enu: periodic quotient of sc groups - one-to-one - product}
		If the action of $G$ on $\dot X$ is non elementary, then $Q$ is infinite.
		Moreover the projection $G \twoheadrightarrow Q$ is one-to-one on small balls in the following sense.
		For every $g \in G\setminus\{1\}$, for every $x \in \dot X$, if $\dist[\dot X]{gx}x \leq r$, then the image of $g$ in $Q$ is not trivial.
		In particular, if $r\geq 1$, then every connected component of $\Gamma$ embeds in the Cayley graph of $Q$ with respect to $S$.
	\end{enumerate}
\end{theo}

\begin{proof}[Proof of Theorems~\ref{res: periodic quotient of sc groups - standard} and \ref{res: periodic quotient of sc groups - product}]
Let $p \in \N^*$ and $r \in \R_+$.
We define a hyperbolicity constant $\delta = 80$.
Let $\Gamma$ be a labelled graph satisfying the conditions of \autoref{res: periodic quotient of sc groups - standard} or \autoref{res: periodic quotient of sc groups - product}.
We write $G = G(\Gamma)$ for the corresponding group and $\dot X = \dot X(\Gamma)$ for its coned-off Cayley graph. 
According to \autoref{thm:hyperbolicity} the cone-off space $\dot X$ is $\delta$-hyperbolic.
Moreover by \autoref{thm:acylindricity}, there exist constants $L$ and $N$, only depending on $p$, such that the action of $G$ on $\dot X$ is $(100\delta,L,N)$-hyperbolic.  
We assumed that $G$ has no even torsion.
Hence if the action of $G$ is elementary, then $G$ is either elliptic or isomorphic to $\Z$ (\autoref{coro:virtually_cyclic}).
In such a situation the first three conclusions of both theorems are obvious.
Hence from now on we assume that the action of $G$ on $\dot X$ is non-elementary.

\medskip
Since $G$ has no even torsion, every lineal subgroup of $G$ is torsion free (\autoref{coro:virtually_cyclic}).
It follows that the $e(G,\dot X)=1$ (see \autoref{res: invariant e} for the definition of this parameter).
Consequently we can apply \autoref{res: acylindricity gives uniform exp}. We denote by $n_{p,r}$ the critical exponent $N_1$ given by \autoref{res: acylindricity gives uniform exp} applied with the parameters $N$, $L$, $\delta$ and $r$. Observe that $n_{p,r}$ only depends on the chosen $p$ and $r$, not on the specific $\Gamma$. Fix an odd exponent $n \geq n_{p,r}$.
The first three conclusions of Theorems~\ref{res: periodic quotient of sc groups - standard} and \ref{res: periodic quotient of sc groups - product} follow from \autoref{res: acylindricity gives uniform exp}. 
Recall indeed that in the context of small cancellation over free products, the graph $\Gamma$ is its own completion (see \autoref{def: small cancellation condition - free product - completion}), whence every factor $F_i$ we started with acts elliptically on $\dot X$.
The first half of Point~(iv) in both theorems is also a consequence of \autoref{res: acylindricity gives uniform exp}.
The second half follows from this observation: the vertex set of each embedded component of $\Gamma$ in $\Cay(G,S)$ has diameter at most 1 in the metric of $\dot X$.
\end{proof}

We are ready to prove our main result.

\begin{proof}[Proof of \autoref{res: main theo - regular sc}]
Let $p\in \N^*$. 
Let $n_p$ the critical exponent of \autoref{res: periodic quotient of sc groups - standard} for $p$ and $r=2$, i.e. $n_p:=n_{p,2}$. 
Let $\Gamma$ be a graph labelled by a set $S$ containing at least two elements and satisfying the $C'_n(1/6,p)$ condition.
We assume that there is no finite groups $F$ generated by $S$ such that $\Gamma$ contained the Cayley graph of $F$ with respect to $S$.
It follows from \autoref{prop:elementary} that either $G(\Gamma)$ is already an infinite group in $\vburn n$ or the action of  $G(\Gamma)$ on the cone-off Cayley graph $\dot X(\Gamma)$ is non-elementary.
In the first case $G_n(\Gamma) = G(\Gamma)$ and the conclusion follows from the usual graphical small cancellation theory \cite[Theorem~1]{Oll06} or \cite[Lemma~4.1 and Theorem~4.3]{Gruber:2015fu}.

\medskip
Let us focus on the second case, that is when the action of $G(\Gamma)$ on $\dot X(\Gamma)$ is non elementary.
We denote by $Q$ the quotient given by \autoref{res: periodic quotient of sc groups - standard}.
\autoref{prop:elliptic} tells us that every elliptic element for the action of $G(\Gamma)$ on $\dot X$ has order dividing $n$.
Thus the group $G_n(\Gamma)$ coincides with $Q$, which is infinite.
Moreover, every component of $\Gamma$ embeds in $\Cay(G_n(\Gamma),S)$.

\medskip
It remains to prove Point~\ref{enu: main theo - regular sc - coarse embedding}. 
In this claim, $S$ is finite, every component of $\Gamma$ is finite, and $\Gamma$ is countable. 
Since $G_n(\Gamma)$ is infinite and every component of $\Gamma$ embeds in $\Cay(G_n(\Gamma),S)$, we observe that $\Gamma$ embeds in $\Cay(G_n(\Gamma),S)$. 
It remains to prove the coarse embedding result.
We follow the strategy of Gruber \cite[Theorem~4.3]{Gruber:2015fu} to obtain our result. 
For simplicity we let  $G_n=G_n(\Gamma)$, $X=\Cay(G,S)$, and $X_n=\Cay(G_n,S)$.

\medskip
If $\Gamma$ is finite, any map $\Gamma\to X_n$ satisfies the axioms of a coarse embedding, hence there is nothing to prove. 
Therefore we can assume that $\Gamma$ is infinite. 
By lining up the (finite) components of $\Gamma$ on a $1$-infinite geodesic ray in $X_n$, we can choose a label-preserving graph homomorphism $f:\Gamma\to X_n$ such that $d(\Gamma_i,\Gamma_j)\geq \diam(\Gamma_i)+\diam(\Gamma_j)+i+j$. 
We show that $f$ is a coarse embedding.

\medskip
Consider (by abuse of notation) $\Gamma_i$ and $\Gamma_j$ two images in $X_n$ of connected components of $\Gamma$ under any label-preserving graph homomorphism. We claim that $\Gamma_i\cap\Gamma_j$ is empty or connected. 
Let $x$ and $y$ be vertices in $\Gamma_i\cap\Gamma_j$. Let $\tilde x$ be a preimage of $x$ via the map $X\to X_n$. 
According to \cite[Lemma~4.1]{Gruber:2015fu} $\Gamma_i$ also embeds in $X$.
Let us denote by $\widetilde \Gamma_i$ a copy of $\Gamma_i$ in $X$.
Recall that every map we are considering is label-preserving.
Hence up to replacing $\widetilde\Gamma_i$ by a translate of $\widetilde\Gamma_i$ we may always assume that $\tilde x$ belongs to $\widetilde\Gamma_i$ and that the map $X \to X_n$ maps $\widetilde\Gamma_i$ onto $\Gamma_i$.
We build in the same way a lift $\widetilde\Gamma_j$ of $\Gamma_j$ containing $\tilde x$.
Let $\tilde y_i$ be the resulting preimage of $y$ in $\widetilde\Gamma_i$ and $\tilde y_j$ the one in $\widetilde\Gamma_j$, and observe $|\tilde y_i-\tilde y_j|_{\dot X}\leq 2$. Denote by $g$ the element of $G(\Gamma)$ with $g\tilde y_i=\tilde y_j$. Since both $\tilde y_i$ and $\tilde y_j$ map to $y$, the image of $g$ in $G_n$ is trivial. Hence, by \autoref{res: periodic quotient of sc groups - standard}~\ref{enu: periodic quotient of sc groups - one-to-one}, $g$ is trivial in $G(\Gamma)$, and $\tilde y_i=\tilde y_j$. Now by \cite[Lemma~2.17]{Gruber:2014wo}, $\widetilde\Gamma_i\cap\widetilde\Gamma_j$ is connected. Hence, there is a path $\gamma$ in this intersection connecting $\tilde x$ to $\tilde y_i$. The image of $\gamma$ in $X_n$ lies in $\Gamma_i\cap\Gamma_j$ and connects $x$ to $y$, whence $\Gamma_i\cap\Gamma_j$ is connected.

\medskip	
We now prove coarse embedding by contradiction: let $(x_k,y_k)\subset \Gamma\times\Gamma$ be a sequence of pairs of vertices such that $|x_k-y_k|_{\Gamma}\to\infty$ and $|f(x_k)-f(y_k)|_{X_n}$ is bounded. Since $d(f(\Gamma_i),f(\Gamma_j))\geq \diam (\Gamma_i)+\diam(\Gamma_j)+i+j$ we may, by possibly going to a subsequence, assume that for each $k$ there exists $j_k$ such that $x_k,y_k\in \Gamma_{j_k}$. 
Again going to a subsequence and using local finiteness of $X_n$, we may assume that there exists a fixed $g\in G_n$ such that $f(x_k)^{-1}f(y_k)=g$ for every $k$. Now $f(x_1)^{-1}\Gamma_{j_1}\cap f(x_k)^{-1}\Gamma_{j_k}$ contains $g$ and, as shown above, is connected for each $k$. Thus, $f(x_1)^{-1}\Gamma_{j_1}\cap f(x_k)^{-1}\Gamma_{j_k}$ contains a simple path from $1$ to $g$, and the length of such a simple path is bounded by $|V\Gamma_{j_1}|-1$. Hence  $|V\Gamma_{j_1}|-1\geq|x_k-y_k|_{\Gamma_{j_k}}$, which contradicts $|x_k-y_k|_{\Gamma_{j_k}}\to\infty$.
\end{proof}

We are also ready to prove our main theorem in the (classical) free product case. Recall from \autoref{section:definitions_and_notation} that a presentation satisfying the (classical) power-free $C_*'(1/6,p)$-condition for which the generating factors do not have even torsion can be regarded as satisfying the (graphical) $C_*'(1/6,p)$-condition. The graphical version of our result will be stated immediately after.

\begin{proof}[Proof of \autoref{res: periodic quotient of sc groups - free product}]
Let $G$ be a small cancellation quotient of a free product $F = F_1 \ast \dots \ast F_m$ and $\dot X$ be the corresponding cone-off space.
Note that every factor $F_k$ acts elliptically on $\dot X$ (recall that when constructing $\dot X$, we consider the graph $\Gamma$ that is the \emph{completion} of graph that is a disjoint union of cycle graphs labelled by the relators).
Moreover the power-free $C'_*(\lambda,p)$-condition together with \autoref{prop:elliptic} implies that an element $g \in G$ is elliptic if and only if it is conjugate to an element of one of the $F_k$. 
The result is now a direct application of  \autoref{res: periodic quotient of sc groups - product}.
\end{proof}

\begin{theo}
\label{res: main theo - product sc}
	Let $p \in \N^*$.
	There exists a critical exponent $n_p \in \N$ such that for every odd integer $n \geq n_p$ the following holds.
	Let $(F_i)_{i \in I}$ be a collection of groups.
    For each $i \in I$ we fix a generating set $S_i$ of $F_i$ and let $S =\sqcup_{i\in I}S_i$. 
	Let $\Gamma$ be a graph labelled by $S$ satisfying the $C_*'(1/6,p)$-condition such that the action of $G(\Gamma)$ on its cone-off space $\dot X$ is non-elementary. Additionally assume that whenever $w$ is a cyclically reduced word all whose powers label paths in $\Gamma$, then $w^n$ labels a closed path in $\Gamma$. Denote $G_n(\Gamma):=G(\Gamma)/G(\Gamma)^n$. Then the following holds.
	\begin{enumerate}
		\item \label{enu: main theo - product sc - infinite}
		The group $G_n(\Gamma)$ is infinite.
		\item \label{enu: main theo - product sc - embedding}
		 Every connected component of $\Gamma$ embeds into $\Cay(G_n(\Gamma),S)$ via a label-preserving graph homomorphism.	
		\item \label{enu: main theo - product sc - factors}
		 Each one of the generating factors $F_i$ embeds as a subgroup in $G_n(\Gamma)$.
	\end{enumerate}
\end{theo}

\begin{rema}
    Observe that our assumptions imply that every $F_i$ is $n$-periodic.
\end{rema}

\begin{proof}
\autoref{prop:elliptic} and the assumptions imply that every elliptic element for the action of $G(\Gamma)$ on $\dot X$ has order dividing $n$.
The proof now goes as the first part of the one in the non-elementary case of \autoref{res: main theo - regular sc} using \autoref{res: periodic quotient of sc groups - product} instead of \autoref{res: periodic quotient of sc groups - standard}.
\end{proof}

We conclude by providing an example that shows that our restriction on proper powers appearing as subwords of relators is indeed necessary in order to obtain infinite $n$-periodic groups.

\begin{exam}[Pride group \cite{Pride:1989ek}]
\label{exa: pride}
	Let $S = \{a, b\}$.
	We consider relations of the form
	\begin{eqnarray*}
		r_n = a^{-1}b^{p_1n}a^{p_2n}b^{p_3n}\dots b^{p_{i_n}n}\\
		s_n = b^{-1}a^{q_1n}b^{q_2n}a^{q_3n}\dots a^{q_{j_n}n}
	\end{eqnarray*}
	and take the group $G = \left<a,b \mid r_n, s_n, n \in \N\right>$.
	Under a suitable choice of $(p_i)$ and $(q_j)$ this presentation satisfies the $C'(1/6,3)$ assumption.
	Observe though that $G/G^n$ is trivial for every $n$.
	On the other hand, $a$ acts elliptically on the cone-off space $\dot X(\Gamma)$, whence particular the infinite cyclic group $\left<a\right>$ embeds in $Q$, and $Q$ is not a torsion group.
	This does not contradict \autoref{res: main theo - regular sc}, because our presentation does not satisfy  $C'_n(1/6,3)$-condition for any $n$.
	
	Clearly, if one prescribes $n$, already the group given by the 2-generator and 2-relator presentation $\langle a,b\mid r_n,s_n\rangle$ does not admit any non-trivial $n$-periodic quotient. Notice that we may achieve any small cancellation parameter we desire, i.e. make the presentation of Pride's group satisfy any given $C'(\lambda)$-condition. Consequently, we can achieve any positive upper bound for the \emph{relative} lengths of subwords that are proper powers. This shows that an \emph{absolute} upper bound on the powers occurring as subwords of relators is indeed required for making any statement in the nature of our main results.
\end{exam}


\noindent
\emph{R\'emi Coulon} \\
Univ Rennes, CNRS \\
IRMAR - UMR 6625 \\
F-35000 Rennes, France \\
\texttt{remi.coulon@univ-rennes1.fr} \\
\texttt{http://rcoulon.perso.math.cnrs.fr}

\medskip

\noindent
\emph{Dominik Gruber} \\
Department of Mathematics\\
ETH Z\"urich\\
R\"amistrasse 101\\
8092 Z\"urich, Switzerland\\
\texttt{dominik.gruber@math.ethz.ch} \\
\texttt{https://people.math.ethz.ch/$\sim$dgruber/}

\todos
\end{document}